\documentclass[letterpaper,twoside]{article}

\usepackage[letterpaper,left=1.5in,right=1.5in,top=1in,bottom=1in]{geometry}

\newcommand{\myauthor}{Benjamin Antieau and Elden Elmanto}
\newcommand{\mytitle}{A primer for unstable motivic homotopy theory}

\title{\mytitle}
\author{Benjamin Antieau\footnote{Benjamin Antieau was supported
by NSF Grant DMS-1461847} and Elden Elmanto\footnote{Elden Elmanto was supported by
NSF Grant DMS-1508040.}}
\date{}

\usepackage[pdfstartview=FitH,
            pdfauthor={\myauthor},
            pdftitle={\mytitle},
            colorlinks,
            linkcolor=reference,
            citecolor=citation,
            urlcolor=e-mail,
            backref]{hyperref}

\usepackage{amsmath}
\usepackage{amscd}
\usepackage{amsbsy}
\usepackage{amssymb}
\usepackage{verbatim}
\usepackage{eufrak}
\usepackage{eucal}
\usepackage{microtype}
\usepackage{hyperref}
\usepackage{mathrsfs}
\usepackage{amsthm}
\usepackage[all,cmtip]{xy}
\usepackage{tikz}
\usetikzlibrary{matrix,arrows}
\usepackage[abbrev,lite,alphabetic]{amsrefs}
\usepackage{dsfont}

\usepackage{fancyhdr}
\usepackage{color}
\definecolor{todo}{rgb}{1,0,0}
\definecolor{conditional}{rgb}{0,1,0}
\definecolor{e-mail}{rgb}{0,.40,.80}
\definecolor{reference}{rgb}{.20,.60,.22}
\definecolor{mrnumber}{rgb}{.80,.40,0}
\definecolor{citation}{rgb}{0,.40,.80}

\let\oldmarginpar\marginpar
\renewcommand\marginpar[1]{\-\oldmarginpar[\raggedleft\footnotesize #1]%
{\raggedright\footnotesize #1}}

\newcommand{\df}[1]{{\bf #1}}

\newcommand{\Fscr}{\mathcal{F}}

\newcommand{\Kscr}{\mathcal{K}}
\newcommand{\Lscr}{\mathcal{L}}

\newcommand{\Oscr}{\mathcal{O}}

\newcommand{\B}{\mathrm{B}}
\newcommand{\D}{\mathrm{D}}
\newcommand{\E}{\mathrm{E}}

\newcommand{\G}{\mathrm{G}}
\renewcommand{\H}{\mathrm{H}}
\newcommand{\I}{\mathrm{I}}
\newcommand{\K}{\mathrm{K}}
\renewcommand{\L}{\mathrm{L}}
\newcommand{\M}{\mathrm{M}}
\newcommand{\N}{\mathrm{N}}

\newcommand{\U}{\mathrm{U}}

\newcommand{\Lbf}{\mathbf{L}}
\newcommand{\Pbf}{\mathbf{P}}
\newcommand{\Rbf}{\mathbf{R}}

\renewcommand{\AA}{\mathds{A}}
\newcommand{\CC}{\mathds{C}}

\newcommand{\GG}{\mathds{G}}

\newcommand{\NN}{\mathds{N}}
\newcommand{\PP}{\mathds{P}}
\newcommand{\QQ}{\mathds{Q}}
\newcommand{\RR}{\mathds{R}}
\newcommand{\ZZ}{\mathds{Z}}

\newcommand{\gp}{\mathrm{gp}}
\newcommand{\op}{\mathrm{op}}
\newcommand{\topo}{\mathrm{top}}
\newcommand{\proj}{\mathrm{proj}}
\newcommand{\inj}{\mathrm{inj}}

\DeclareMathOperator{\id}{id}
\newcommand{\Ch}{\mathrm{Ch}}
\newcommand{\Bl}{\mathrm{Bl}}
\newcommand{\eq}{\mathrm{eq}}
\newcommand{\Th}{\mathrm{Th}}
\DeclareMathOperator{\Sing}{Sing}

\newcommand{\Sets}{\mathrm{Sets}}

\newcommand{\Ho}{\mathrm{Ho}}

\newcommand{\sPre}{\mathrm{sPre}}

\newcommand{\sSets}{\mathrm{sSets}}
\newcommand{\Spc}{\mathrm{Spc}}

\DeclareMathOperator{\Ext}{Ext}

\DeclareMathOperator{\Pic}{Pic}

\newcommand{\CH}{\mathrm{CH}}

\newcommand{\Fun}{\mathrm{Fun}}

\DeclareMathOperator{\Hom}{Hom}
\newcommand{\map}{\mathrm{map}}

\DeclareMathOperator{\PGL}{PGL}

\DeclareMathOperator{\SL}{SL}
\DeclareMathOperator{\GL}{GL}
\DeclareMathOperator{\Sp}{Sp}

\newcommand{\SU}{\mathrm{SU}}
\newcommand{\SO}{\mathrm{SO}}
\newcommand{\Gm}{\mathds{G}_{m}}

\newcommand{\Vect}{\mathrm{Vect}}

\DeclareMathOperator{\BSL}{BSL}
\DeclareMathOperator{\ESL}{ESL}

\newcommand{\BU}{\mathrm{BU}}

\newcommand{\BGL}{{\mathrm{BGL}}}
\newcommand{\Gr}{\mathrm{Gr}}

\newcommand{\ShBun}{\mathbf{Bun}}
\newcommand{\Tors}{\mathrm{Tors}}

\DeclareMathOperator*{\holim}{holim}
\DeclareMathOperator*{\colim}{colim}
\DeclareMathOperator*{\hocolim}{hocolim}

\newcommand{\et}{\mathrm{\acute{e}t}}
\newcommand{\Nis}{\mathrm{Nis}}

\DeclareMathOperator{\Spec}{Spec}

\newcommand{\Sm}{\mathrm{Sm}}

\newcommand{\we}{\simeq}
\newcommand{\iso}{\cong}

\theoremstyle{plain}
\newtheorem{theorem}{Theorem}[section]
\newtheorem*{theorem*}{Theorem}
\newtheorem{lemma}[theorem]{Lemma}

\newtheorem{proposition}[theorem]{Proposition}

\newtheorem{conjecture}[theorem]{Conjecture}
\newtheorem{corollary}[theorem]{Corollary}

\newtheoremstyle{named}{}{}{\itshape}{}{\bfseries}{.}{.5em}{#1 \thmnote{#3}}
\theoremstyle{named}

\theoremstyle{definition}
\newtheorem{definition}[theorem]{Definition}

\newtheorem{example}[theorem]{Example}
\newtheorem{warning}[theorem]{Warning}
\newtheorem{recipe}[theorem]{Recipe}

\newtheorem{exercise}[theorem]{Exercise}

\newtheorem{notation}[theorem]{Notation}
\newtheorem{remark}[theorem]{Remark}

\begin{document}

\maketitle

\begin{abstract}
    \noindent
    In this expository article, we give the foundations, basic facts, and first
    examples of unstable motivic homotopy theory with a view towards the approach of
    Asok-Fasel to the classification of vector bundles on smooth complex affine varieties.
    Our focus is on making these techniques more accessible to algebraic
    geometers.


    \paragraph{Mathematics Subject Classification 2010.}
    Primary:
    \href{http://www.ams.org/mathscinet/msc/msc2010.html?t=13Cxx&btn=Current}{13C10},
    \href{http://www.ams.org/mathscinet/msc/msc2010.html?t=14Fxx&btn=Current}{14F42},
    \href{http://www.ams.org/mathscinet/msc/msc2010.html?t=19Dxx&btn=Current}{19D06}.
    Secondary:
    \href{http://www.ams.org/mathscinet/msc/msc2010.html?t=55Rxx&btn=Current}{55R50},
    \href{http://www.ams.org/mathscinet/msc/msc2010.html?t=55Sxx&btn=Current}{55S45}.
\end{abstract}

\tableofcontents

% Introduction
\section{Introduction}

This primer is intended to serve as an introduction to the basic facts about
Morel and Voevodsky's motivic, or $\AA^1$, homotopy theory~\cite{morel-voevodsky}, \cite{voevodsky-icm}, with a
focus on the unstable part of the theory. It was
written following a week-long summer school session on this topic
led by Antieau at the University of Utah in July 2015. The choice
of topics reflects what we think might be useful for algebraic geometers interested in
learning the subject.

In our view, the starting point of the development of unstable motivic homotopy theory is the
resolution of Serre's conjecture by Quillen~\cite{quillen-serre}
and Suslin~\cite{suslin-serre}. Serre
asked in~\cite{serre-fac} whether every finitely generated projective module over
$k[x_1,\ldots,x_n]$ is free when $k$ is a field. Put another way, the question is whether
$$\Vect_r(\Spec k)\rightarrow\Vect_r(\AA^n_k)$$ is a bijection for all $n\geq 1$, where
$\Vect_r(X)$ denotes the set of isomorphism classes of rank $r$ vector bundles
on $X$. Quillen and Suslin showed that this is true
and in fact proved the analogous statement when $k$ is replaced by a Dedekind
domain. This suggested the following conjecture.

\begin{conjecture}[Bass-Quillen]\label{conj:bass-quillen}
    Let $X$ be a regular noetherian affine scheme of finite Krull dimension. Then, the pullback map
    $\Vect_r(X)\rightarrow\Vect_r(X\times\AA^n)$ is a bijection for all $r\geq
    1$ and all $n\geq 1$.
\end{conjecture}

The Bass-Quillen conjecture has been proved in many cases, but not yet in full
generality. Lindel~\cite{lindel} prove the conjecture when $X$ is essentially
of finite type over a field, and Popescu proved it when $X$ is the spectrum of
an unramified regular local ring (see~\cite{swan-popescu}*{Theorem~2.2}).
Piecing these results together one can, for example, allow $X$ to be the
spectrum of a ring with the property that all its localizations at maximal ideals
are smooth over a Dedekind ring with a perfect residue field
see~\cite{asok-hoyois-wendt}*{Theorem 5.2.1}. For a
survey of other results in this direction, see~\cite{lam-serre2}*{Section VIII.6}.

If $X$ is a reasonable topological space, such as a manifold, simplicial
complex, or CW complex, then there are also bijections
$\Vect_r^{\topo}(X)\rightarrow\Vect_r^{\topo}(X\times I^1)$, where $I^1$ is the
unit interval and $\Vect_r^{\topo}$ denotes the set of isomorphism classes of
rank $r$ topological complex vector bundles on $X$. Thus, the Quillen-Suslin
theorem and the Bass-Quillen conjecture suggest that there might be a
homotopy theory for schemes in which the affine line $\AA^1$ plays the role of
the contractible unit interval.

Additional evidence for this hypothesis is provided by the fact that many important cohomology
theories for smooth schemes are $\AA^1$-invariant.
For example, the pullback maps in Chow groups
$$\CH^*(X)\rightarrow\CH^*(X\times_k\AA^1),$$ Grothendieck groups
$$\K_0(X)\rightarrow\K_0(X\times_k\AA^1),$$ and \'etale cohomology groups
$$\H^*_{\et}(X,\mu_\ell)\rightarrow\H^*_{\et}(X\times_k\AA^1,\mu_\ell)$$ are isomorphisms when $X$ is smooth
over a field $k$ and $\ell$ is invertible in $k$.

Now, we should note immediately, that the functor
$\Vect_r:\Sm_k^{\op}\rightarrow\mathrm{Sets}$ is \emph{not} itself
$\AA^1$-invariant. Indeed, there are vector bundles on $\PP^1\times_k\AA^1$
that are not pulled back from $\PP^1$. The reader can construct a vector bundle
on the product such that the restriction to $\PP^1\times\{1\}$ is a non-trivial
extension $E$ of $\Oscr(1)$ by $\Oscr(-1)$
while the restriction to $\PP^1\times\{0\}$ is $\Oscr(-1)\oplus\Oscr(1)$.
Surprisingly, for affine schemes, this proves not to be a problem: forcing
$\Vect_r$ to be $\AA^1$-invariant produces an object which still has the
correct values on smooth affine schemes.

The construction of (unstable) motivic homotopy theory over a quasi-compact and
quasi-separated scheme $S$ takes three steps. The first stage
is a homotopical version of the process of passing from a category of schemes to
the topos of presheaves on the category. Specifically, one
enlarges the class of spaces from $\Sm_S$, the category of smooth schemes over
$S$, to the category of presheaves of simplicial sets $\sPre(\Sm_S)$ on $\Sm_S$. An object $X$ of
$\sPre(\Sm_S)$ is a functor $X:\Sm_S^{\op}\rightarrow\sSets$, where $\sSets$ is
the category of simplicial sets, one model for the homotopy theory of
CW complexes. Presheaves of sets give examples of simplicial presheaves by
viewing a set as a discrete space. There is a Yoneda embedding
$\Sm_S\rightarrow\sPre(\Sm_S)$ as usual. In the next stage, one imposes a
descent condition, namely focusing on those presheaves that satisfy the
appropriate homotopical version of Nisnevich
descent. We note that one can construct motivic homotopy theory with other
topologies, as we will do later in Section~\ref{sec:classifying}. The choice of
Nisnevich topology is motivated by the fact that it is the coarsest topology
where we can prove the purity theorem (Section~\ref{sec:purity}) and the finest
where we can prove representability of $K$-theory (Section~\ref{sub:rerpresentability}).
The result is a homotopy theory enlarging the category
of smooth schemes over $S$ but which does not carry any information about the
special role $\AA^1$ is to play. In the third and final stage, the projection
maps $X\times_S\AA^1\rightarrow X$ are formally inverted.

In practice, care must be taken in each stage; the technical framework we
use in this paper is that of model categories, although one could equally use
$\infty$-categories instead, as has been done recently by Robalo~\cite{robalo}. Model categories, Quillen functors (the
homotopical version of adjoint pairs of functors), homotopy limits and
colimits, and Bousfield's theory of localization are all explained in the lead
up to the construction of the motivic homotopy category.

When $S$ is regular and noetherian, algebraic $K$-theory turns out to be
representable in $\Spc_S^{\AA^1}$, as are
many of its variants. A pleasant surprise however is that despite the fact that
$\Vect_r$ is not $\AA^1$-invariant on all of $\Sm_S$, its $\AA^1$-localization
still has the correct values on smooth affine schemes over $k$. This is a
crucial result of Morel~\cite{morel}*{Chapter 8}, which was simplified in the Zariski topology by Schlichting \cite{schlichting-euler}, 
and simplified and generalized by Asok, Hoyois, and
Wendt~\cite{asok-hoyois-wendt}. This fact is at the heart of applications of
motivic homotopy theory to the classification of vector bundles on smooth
affine complex varieties by Asok and Fasel.

\medskip
We describe now the contents of the paper. As mentioned above,
the document below reflects topics the authors decided should belong in a first
introduction to $\AA^1$-homotopy theory, especially for people coming from
algebraic geometry. Other surveys in the field which focus on different aspects
of the theory include  \cite{dro-book}, \cite{levine-milan},
\cite{levine-vietnam}, \cite{morel-icm}; a textbook reference for the ideas
covered in this survey is \cite{morel}. Voevodsky's ICM address pays special
attention to the topological motivation for the theory in~\cite{voevodsky-icm}.

Some of these topics were the focus of
Antieau's summer school course at the AMS Summer Institute in
Algebraic Geometry at the University of Utah in July 2015. This includes the
material in Section~\ref{sec:topvect} on topological vector bundles. This
section is meant to explain the power of the Postnikov tower approach to
classification problems and entice the reader to dream of the possibilities
were this possible in algebraic geometry.

The construction of the motivic homotopy category is
given in Section~\ref{sec:construction} after an extensive introduction to
model categories, simplicial presheaves, and the Nisnevich topology.
Other topics are meant to fill gaps in the literature,
while simultaneously illustrating the techniques common to the field.
Section~\ref{sec:basic} establishes the basic properties of
motivic homotopy theory over $S$. It is meant to be a
kind of cookbook and contains many examples, exercises, and computations.
In Section~\ref{sec:classifying}
we define and give examples of classifying spaces $\B G$ for linear algebraic
groups $G$, and perform some calculations of their homotopy sheaves. The
answers will involve algebraic $K$-theory which is discussed in
Section~\ref{sec:K}. Following \cite{morel-voevodsky}, with some modifications,
we give a self-contained proof that algebraic $K$-theory is representable in
the $\AA^1$-homotopy category, and we identify its representing object as the
$\AA^1$-homotopy type of a classifying space $\BGL_\infty$. In
Section~\ref{sec:purity}, we prove the critical purity theorem which is the
source of Gysin sequences. A brief vista at
the end of the paper, in Section~\ref{sec:asok-fasel}, illustrates how all of
this comes together to classify vector bundles on smooth affine schemes.
Finally, in Section~\ref{sec:exercises}, we gather some miscellaneous additional
exercises.

Many things are not in this paper. We view the biggest omission as the
exclusion of a presentation of the first non-zero homotopy sheaves of punctured
affine spaces. Morel proved that
$$\pi_n^{\AA^1}(\AA^{n+1}-\{0\})\iso\mathbf{K}_{n+1}^{\mathrm{MW}},$$
the $(n+1)$st unramified Milnor-Witt $K$-theory sheaf where $n \geq 1$. A proof may be found
in~\cite{morel}*{Chapter 6}.

Other topics we would include granted unlimited time include stable motivic
homotopy theory, and in particular
the motivic spectral sequence, the stable connectivity theorem of Morel,
the theory of algebraic cobordism due to Levine-Morel~\cite{levine-morel},
motivic cohomology and the work of Voevodsky and Rost on the
Bloch-Kato conjecture, and
the work~\cite{dugger-isaksen} of Dugger and Isaksen on the motivic Adams spectral sequence.

\paragraph{Acknowledgements.} These notes were commissioned by the organizers
of the Graduate Student Bootcamp for the 2015 Algebraic Geometry Summer
Research Institute. Each mentor at the bootcamp gave an hour-long lecture and
then worked with a small group of graduate students and postdocs over the course of the
week to understand their topic in greater detail.
We would like to the thank the organizers, \.{I}zzet Co\c{s}kun, Tommaso de
Fernex, Angela Gibney, and Max Lieblich, for creating a wonderful atmosphere in
which to do this.

It was a pleasure to have the following students
and postdocs in Antieau's group: John Calabrese (Rice), Chang-Yeon Cho (Berkeley),
Ed Dewey (Wisconsin), Elden Elmanto (Northwestern), M\'arton Hablicsek
(Penn), Patrick McFaddin (Georgia), Emad Nasrollahpoursamami (Caltech), 
Yehonatan Sella (UCLA), Emra Sert\"oz (Berlin), Arne Smeets (Imperial),
Arnav Tripathi (Stanford), and Fei Xie (UCLA).

We would also like to thank John Calabrese, Gabriela Guzman, Marc Hoyois, Kirsten
Wickelgren, and Benedict Williams for comments and corrections on earlier drafts of this paper.

Aravind Asok deserves special thanks for several useful conversations
about material to include in the paper.

Finally, we thank the anonymous referee for their careful reading of the paper;
they caught several mistakes, both trivial and non-trivial, that are corrected
in this version. They also supplied many, many additional references.

% Postnikov towers
\section{Classification of topological vector bundles}\label{sec:topvect}

We introduce the language of Postnikov towers and illustrate their use through
several examples involving the classification of topological vector bundles.
The point is to tempt the reader to imagine the power these tools would possess in
algebraic geometry if they existed.
General references for the material here include Hatcher and
Husemoller's books~\cite{hatcher,husemoller}.

\subsection{Postnikov towers and Eilenberg-MacLane spaces}

Let $S^i$ denote the $i$-sphere, embedded in $\RR^{i+1}$ as the unit sphere,
and let $s=(1,0,\ldots,0)$ be the basepoint.
Recall that if $(X,x)$ is a pointed space, then
$$\pi_i(X,x)=[(S^i,s),(X,x)]_*,$$ the set of homotopy classes of pointed maps
from the $i$-sphere to $X$. The set of path-components $\pi_0(X,x)$ is simply a
set pointed by the component containing $x$. The fundamental group is
$\pi_1(X,x)$, a not-necessarily-abelian group. The groups $\pi_i(X,x)$ are
abelian for $i\geq 2$. For a path-connected space, $\pi_i(X,x)$ does not depend
on $x$, so we will often omit $x$ from our notation and write $\pi_i(X)$ or
$\pi_iX$.

\begin{definition}
    A map of spaces $f:X\rightarrow Y$ is an \df{$n$-equivalence} if
    $\pi_0(f):\pi_0(X)\rightarrow\pi_0(Y)$ is a bijection and if for each
    choice of a basepoint $x\in X$, the induced map
    $$\pi_i(f):\pi_i(X,x)\rightarrow\pi_i(Y,f(x))$$ is an isomorphism for each $i<n$ and
    a surjection for $i=n$. The map $f$ is a \df{weak homotopy equivalence} if
    it is an $\infty$-equivalence.
\end{definition}

Typically we are interested in working with spaces up to weak homotopy
equivalence. The correct notion of a fibration in this setting is a Serre
fibration. Let $D^n$ denote the $n$-disk and $I^1$ the unit interval. A map $p:E\rightarrow B$ is a \df{Serre fibration} (or simply a
\df{fibration} as we will not use any other notion of fibration for maps of
topological spaces) if for every diagram
\begin{equation*}
    \xymatrix{
        D^n\times\{0\}\ar[r]\ar[d]^i   &   E\ar[d]^p\\
        D^n\times I^1\ar[r]\ar@{.>}[ur]  &   B
    }
\end{equation*}
of solid arrows, there exists a dotted lift making both triangles commute. In
other words, $p$ has the \df{right lifting property} with respect to the maps
$D^n\times\{0\}\rightarrow D^n\times I^1$. This property is equivalent to
having the right lifting property with respect to all maps $A\times I^1\cup
X\times\{0\}\rightarrow X$ for all CW pairs $(X,A)$.

There is a functorial way of replacing an arbitrary map $f:X\rightarrow Y$
by a Serre fibration. Let $P_f$ be the space consisting of pairs $(x,\omega)$
where $x\in X$ and $\omega:I^1\rightarrow Y$ such that $\omega(0)=f(x)$. There
is a natural inclusion $X\rightarrow P_f$ sending $x$ to $(x,c_x)$, where $c_x$
is the constant path at $f(x)$, and there is a natural map $P_f\rightarrow Y$
sending $(x,\omega)$ to $\omega(1)$.

\begin{exercise}
    Show that the map $X\rightarrow P_f$ is a homotopy equivalence and that
    $P_f\rightarrow Y$ is a fibration.
\end{exercise}

Given a fibration $p:E\rightarrow B$ and a basepoint $e\in E$, the subspace
$F=p^{-1}(p(e))$ is the \df{fiber} of $p$ at $p(e)$. The point $e$ is inside
$F$. The crucial fact about Serre fibrations is that the sequence
$$(F,e)\rightarrow(E,e)\rightarrow (B,p(e))$$ gives rise to a long exact
sequence
$$\cdots\rightarrow\pi_n(F,e)\rightarrow\pi_n(E,e)\rightarrow\pi_n(B,p(e))\rightarrow\pi_{n-1}(F,e)\rightarrow\cdots
\rightarrow\pi_0(F,e)\rightarrow\pi_0(E,e)\rightarrow\pi_0(B,p(e))$$
of homotopy groups. Some explanation of `exactness' is required in low-degrees
as they are only groups or pointed sets. We refer
to~\cite{bousfield-kan}*{Section IX.4}.

\begin{definition}
    The \df{homotopy fiber} $F_f(y)$ of a map $f:X\rightarrow Y$ over a point $y\in Y$ is the
    fiber of $P_f\rightarrow Y$ over $y$. A sequence $(F,x)\rightarrow
    (X,x)\xrightarrow{f}(Y,y)$ of pointed spaces is a \df{homotopy fiber sequence} if
    \begin{enumerate}
        \item   $Y$ is path-connected,
        \item   $f(F)=y$, and
        \item   the natural map $F\rightarrow F_f(y)$ is a weak homotopy equivalence.
    \end{enumerate}
\end{definition}

The homotopy fiber sequences are those which behave just as
well as fiber sequences from the point of view of their homotopy groups.

\begin{exercise}
    Given a space $X$ and a point $x\in X$, the based loop space $\Omega_x X$
    is the homotopy fiber of $x\rightarrow X$. When $X$ is path-connected or
    the basepoint is implicit, we will write $\Omega X$ for $\Omega_xX$.
\end{exercise}

\begin{theorem}
    Let $X$ be a path-connected space, so that $\pi_0(X,x)=*$. There exists a
    commutative diagram of pointed spaces
    
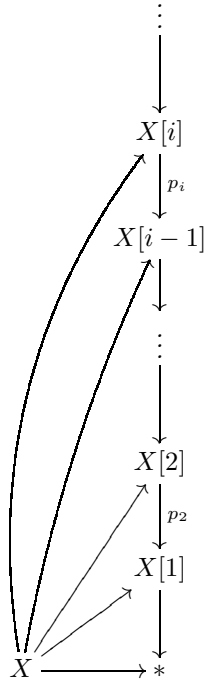
\begin{equation*}
        \xymatrix{
            &\vdots\ar[d]\\
            &X[i]\ar[d]^{p_i}\\
            &X[i-1]\ar[d]\\
            &\vdots\ar[d]\\
            &X[2]\ar[d]^{p_2}\\
            &X[1]\ar[d]\\
            X\ar[r]\ar[ur]\ar[uur]\ar@/^/[uuuur]\ar@/^2pc/[uuuuur]&\ast
        }
    \end{equation*}
    such that
    \begin{enumerate}
        \item   $$\pi_j(X[i])\iso\begin{cases}
                    \pi_jX  &   j\leq i,\\
                    0       &   j>i;
                \end{cases}$$
        \item   $X\rightarrow X[i]$ is an $(i+1)$-equivalence;
        \item   each map $X[i+1] \rightarrow X[i]$ is a Serre fibration;
        \item   the natural map $X\rightarrow\lim_i X[i]$ is a weak homotopy
            equivalence.
    \end{enumerate}
    The space $X[i]$ is the \df{$i$th Postnikov section of $X$}, and the
    diagram is called the \df{Postnikov tower} of $X$.
\end{theorem}

\begin{proof}
    See Hatcher~\cite{hatcher}*{Chapter 4}. The basic idea is that one builds
    $X[i]$ from $X$ by first attaching cells to $X$ to kill $\pi_{i+1}$. Then,
    attaching cells to the result to kill $\pi_{i+2}$, and so on.
\end{proof}

\begin{definition}
    Let $i>0$ and let $G$ be a group, abelian if $i>1$. A \df{$K(G,i)$-space} is a
    connected space $Y$ such that $\pi_i(Y)\iso G$ and $\pi_j(Y)=0$ for $j\neq i$.
    As a class, these are referred to as \df{Eilenberg-MacLane spaces}.
\end{definition}

\begin{exercise}
    The homotopy fiber of $X[i]\rightarrow X[i-1]$ is a $K(\pi_iX,i)$-space.
\end{exercise}

% \begin{proof}
%     Let $F$ denote the homotopy fiber, and choose a basepoint $x$ of $F$
%     mapping to the basepoint of $X[i]$. Since $F\rightarrow X[i]\rightarrow
%     X[i-1]$ is a homotopy fiber sequence, there is a long exact sequence
%     $$\cdots\rightarrow\pi_{j+1}(X[i-1])\rightarrow\pi_j(F)\rightarrow\pi_j(X[i])\rightarrow\pi_j(X[i-1])\rightarrow\pi_{j-1}(F)\rightarrow\cdots$$
%     of homotopy groups. Together with
%     the fact that $\pi_j(X[i])\rightarrow\pi_j(X[i])$ is an isomorphism for
%     $j<i$, this implies that $\pi_i(F)\rightarrow\pi_i(X[i])$ is an isomorphism and
%     $\pi_j(F)=0$ for $j\neq i$.
% \end{proof}

Suppose that $Y$ is another space,
and we want to construct a map $Y\rightarrow X$. We can hope to start with a
map $Y\rightarrow X[1]$, lift it to a map $Y\rightarrow X[2]$, and on up the
Postnikov tower. Using the fact that $X$ is the limit of the tower, we would
have constructed a map $Y\rightarrow X$.

What we need is a way of knowing when a map $Y\rightarrow X[i]$ lifts to a map
$Y\rightarrow X[i+1]$. We need an \df{obstruction theory} for such lifts. Before
getting into the details in the topological setting, we consider an example
from algebra. Let $$0\rightarrow E\rightarrow F\rightarrow G\rightarrow 0$$ be
an exact sequence of abelian groups. Let $g:H\rightarrow G$ be a homomorphism.
When can we lift $g$ to a map $f:H\rightarrow F$? The extension is classified
by a class $p\in\Ext^1(G,E)$. This can be viewed as a map $G\rightarrow E[1]$
in the derived category $\D(\ZZ)$. Composing with $f$, we get the pulled back
extension $f^*p\in\Ext^1(H,E)$, viewed either as the composition $H\rightarrow
G\rightarrow E[1]$ in $\D(\ZZ)$, or as an induced extension $$0\rightarrow E\rightarrow
F'\rightarrow H\rightarrow 0.$$ Now, we know that $g$ lifts if and only if the
extension $F'$ splits if and only if $f^*p=0\in\Ext^1(H,E)$.
The theory we explain now is a \df{nonabelian} version of this example.

\begin{definition}
    A homotopy fiber sequence $F\rightarrow X\xrightarrow{p} Y$ is \df{principal}
    if there is a delooping $B$ of $F$ (so that $\Omega B\we F$) and a map
    $k:Y\rightarrow B$ such that $p$ is homotopy equivalent to the homotopy fiber of $k$. We will call
    $k$ the \df{classifying map} of the principal fiber sequence.
\end{definition}

\begin{example}
    The reduced cohomology $\widetilde{\H}^{i+1}(X,A)$ of a pointed space
    $(X,x)$ with coefficients in an abelian group $A$
    is the kernel of the restriction map
    $\H^{i+1}(X,A)\rightarrow\H^{i+1}(\{x\},A)$. If $i+1>0$, then the reduced
    cohomology is isomorphic to the unreduced cohomology.
    Recall that the reduced cohomology of $X$ can be represented as
    $\widetilde{\H}^{i+1}(X,A)=[(X,x),K(A,i+1)]_*$. That is, Eilenberg-MacLane spaces
    represent cohomology classes. Given a cohomology class
    $k\in\widetilde{\H}^{i+1}(X,A)$ viewed as a map $X\rightarrow K(A,i+1)$,
    the homotopy fiber of $k$ is a space $Y$ with
    $Y\rightarrow X$ having homotopy fiber $K(A,i)$.

    Suppose that $k\in\widetilde{\H}^2(X,\ZZ)$. The homotopy fiber sequence one
    gets is $K(\ZZ,1)\rightarrow Y\rightarrow X$. Since $S^1\we\CC^*\we
    K(\ZZ,1)$, we see that $k$ corresponds to a topological complex line bundle (up to
    homotopy) $Y\rightarrow X$, as expected.
\end{example}

\begin{lemma}\label{lem:lifting}
    Let $F \rightarrow X\rightarrow Y$ be a principal fibration classified by
    $k:Y\rightarrow B$ (so that $F\we\Omega B$). Let $Z$ be a CW complex.
    Then, a map $Z\rightarrow Y$ lifts to $Z\rightarrow X$ if and only if the
    composition $Z\rightarrow Y\rightarrow B$ is nullhomotopic.
\end{lemma}

\begin{proof}
    This follows from the definition of a fibration. Indeed, we can assume that
    $Y\rightarrow B$ is a fibration (by replacing the map by a fibration using
    $P_k$ if necessary) and that $X$ is the fiber over the
    basepoint. Applying the right lifting property to the case at hand,
    where $Z\times I^1\rightarrow B$ is a nullhomotopy from the map
    $Z\rightarrow B$ to the map $Z\rightarrow\{b\}\subseteq B$, we see that the
    initial map $Z\rightarrow Y$ is homotopic to a map landing in the actual
    fiber of $Y\rightarrow B$. This fiber is $X$.
\end{proof}

\begin{theorem}\label{thm:principal}
    If $X$ is simply connected ($X$ is path-connected and $\pi_1(X)=0$), then
    the Postnikov tower of $X$ is a tower of principal fibrations. In
    particular, for each $i\geq 1$ there is a ladder of homotopy fiber
    sequences
    \begin{equation*}
        \xymatrix{
            K(\pi_iX,i)\ar[r]   &   X[i]\ar[d]    &\\
            &X[i-1]\ar[r]^{k_{i-1}}&K(\pi_iX,i+1).
        }
    \end{equation*}
    Specifically, $p_i:X[i]\rightarrow X[i-1]$ is the homotopy fiber of $k_{i-1}$
    and $K(\pi_iX,i)\rightarrow X[i]$ is the homotopy fiber of $p_i$. The map
    $k_{i-1}$ is the \df{$(i-1)$st $k$-invariant} of $X$. It represents a
    cohomology class in $\H^{i+1}(X[i-1],\pi_iX)$.
\end{theorem}

\begin{corollary}
    Let $X$ be a simply connected space. For each $f:Y\rightarrow X[i-1]$, there
    is a uniquely determined class $f^*k_{i-1}\in\H^{i+1}(Y,\pi_iX)$. The map
    $f$ lifts if and only if $f^*k_{i-1}=0$.
\end{corollary}

\begin{proof}
    This follows from Theorem~\ref{thm:principal} and Lemma~\ref{lem:lifting},
    together with the fact that Eilenberg-MacLane spaces represent cohomology
    classes.
\end{proof}

A more complicated, equivariant version of the theory applies in the
non-principal case (so that $X$ is in particular not simply connected), but we ignore that for now as it is unnecessary in the
topological applications we have in mind below. It is explained briefly in
Section~\ref{sec:asok-fasel}.

\subsection{Representability of topological vector bundles}

\begin{definition}
    Let $G$ be a topological group. Recall that a $G$-\df{torsor} on a space
    $X$ is a space $p:Y\rightarrow X$ over $X$ together with a (left) group action $a:G\times
    Y\rightarrow Y$ such that
    \begin{enumerate}
        \item   $p(a(g,y))=p(y)$ (the action preserves fibers), and
        \item   the natural map $G\times Y\rightarrow Y\times_X Y$ given by
            $(g,y)\mapsto (a(g,y),y)$ is an isomorphism.
    \end{enumerate}
    Torsors for $G$ are also called \df{principal $G$-bundles}.
\end{definition}

\begin{example}
    The \df{trivial} $G$-torsor on $X$ is $G\times X\rightarrow X$, with the
    projection map as the structure map.
\end{example}

Given $G$-torsors $p:Y\rightarrow X$ and $p':Y'\rightarrow X'$, a morphism of
$G$-torsors $(f,g):(Y,p,X)\rightarrow (Y',p',X')$ is a map $f:Y\rightarrow Y'$ of
$G$-spaces (i.e., compatible with the $G$-action) together with a map
$g:X\rightarrow X'$ such that $g(p(y))=p'(f(y))$. Given a map $g:X\rightarrow
X'$ and a $G$-torsor $p':Y'\rightarrow X'$, there is a uniquely determined
$G$-torsor structure on $Y=X\times_{X'}Y'$. The projection map $f:Y\rightarrow
X$ makes $(f,g)$ into a morphism of $G$-torsors. We will write $g^*Y'$ for the pull back bundle.

\begin{definition}
    A $G$-torsor $Y\rightarrow X$ is \df{locally trivial} if there is an open
    cover $\{g_i:U_i\rightarrow X\}_{i\in I}$ of $X$ such that $g_i^*Y$ is isomorphic as a
    $G$-torsor to $G\times U_i$ for each $i$. The subcategory of $G$-torsors on
    a fixed base $X$ is naturally a groupoid. We write
    $\ShBun_G(X)$ for the full subcategory of locally trivial $G$-torsors on $X$. The set of
    isomorphism classes of $\ShBun_G(X)$ will be denoted $\mathrm{Tors}_G(X)$.
\end{definition}

\begin{example}
    Let $L\rightarrow X$ be a complex line bundle. The fibers are in particular
    $1$-dimensional complex vector spaces. Let $Y=L-s(X)$, where
    $S:X\rightarrow L$ is the $0$-section. There is a natural action of the
    topological (abelian) group $\CC^*$ on $Y$ given simply by scalar
    multiplication in the fibers. In this case, $Y$ becomes a principal
    $\CC^*$-bundle on $X$.
\end{example}

\begin{theorem} [Steenrod]
    Let $G$ be a topological group. Then, there is a connected space $\B G$ with a
    $G$-torsor $\gamma_{G}$ such that the natural pullback map
    $[X,\B G]\rightarrow\mathrm{Tors}_G(X)$ sending $f:X\rightarrow\B G$ to
    $f^*\gamma_{G}$ is an isomorphism for all paracompact Hausdorff spaces $X$.
    Moreover, $\Omega\B G\we G$.
\end{theorem}

\begin{proof}
    See Husemoller~\cite{husemoller}*{Theorem 4.12.2}. The existence of $\B G$
    can be proved in greater generality as the representability of certain
    functors satisfying Mayer-Vietoris and homotopy invariance properties.
    The total space of $\gamma_G$ is a contractible space with a free
    $G$-action. Hence, $G\rightarrow\gamma_G\rightarrow\B G$ is a fiber
    sequence. It follows that $\Omega\B G\we G$.
\end{proof}

\begin{remark}
    Any CW complex is paracompact Hausdorff.
    Any differentiable manifold is paracompact and Hausdorff, as is the
    underlying topological space associated to a separated complex algebraic variety.
\end{remark}

\begin{example}
    If $A$ is an abelian group, then $K(A,n)$ can be given the structure of a
    topological abelian group. In this case, $\B K(A,n)$ is a $K(A,n+1)$-space.
    Indeed, $\Omega\B K(A,n)\we K(A,n)$.
\end{example}

\begin{definition}
    Let $p:Y\rightarrow X$ be a $G$-torsor, and let $F$ be a space with a
    (left) $G$-action. There is then a left $G$-action on $Y\times F$, the
    diagonal action. Let $F_Y$ denote the quotient $(Y\times F)/G$. There is
    a natural map $F_Y\rightarrow Y/G\iso X$. The fibers of this map are all
    isomorphic to $F$. The space $F_Y\rightarrow X$ is called the
    \df{$F$-bundle} associated to $Y$.
\end{definition}

\begin{example}
    Let $p:Y\rightarrow X$ be a locally trivial $\GL_n(\CC)$-torsor. Let $\GL_n(\CC)$ act on
    $\CC^n$ by matrix multiplication. Then, $\CC^n_Y\rightarrow X$ is a vector bundle. In fact, this
    association gives a natural bijection
    $$\mathrm{Tors}_{\GL_n(\CC)}(X)\iso\Vect_n^{\topo}(X).$$
\end{example}

\begin{corollary}
    If $X$ is a paracompact Hausdorff space, then there is a natural bijection
    \begin{equation*}
        [X,\BGL_n(\CC)]\iso\Vect_n^{\topo}(X).
    \end{equation*}
\end{corollary}

In the case of $\GL_n(\CC)$ we can construct a more explicit version of
$\BGL_n(\CC)$ by using Grassmannians. Let $\Gr_n(\CC^{n+k})$ denote the
Grassmannian of $n$-plane bundles in $\CC^{n+k}$, and let
$\Gr_n=\colim_k\Gr_n(\CC^{n+k})$ denote the colimit. Over each
$\Gr_n(\CC^{n+k})$ there is a canonical $\GL_n(\CC)$-bundle given by the
Stiefel manifold $V_n(\CC^{n+k})$, the space of $n$ linearly independent
vectors in $\CC^{n+k}$. The map sending a set of linearly independent vectors
to the subspace they span gives a surjective map
$$V_n(\CC^{n+k})\rightarrow\Gr_n(\CC^{n+k}).$$ There is a natural free action of
$\GL_n(\CC)$ on $V_n(\CC^{n+k})$, and
$V_n(\CC^{n+k})\rightarrow\Gr_n(\CC^{n+k})$ is a locally trivial $\GL_n(\CC)$-torsor with this
action.

\begin{lemma}
    The space $V_n(\CC^{n+k})$ is $2k$-connected for $n\geq 1$.
\end{lemma}

\begin{proof}
    Consider the map $V_n(\CC^{n+k})\rightarrow \CC^{n+k}-\{0\}$ sending a set
    of linearly independent vectors to the last vector. The fibers are all
    isomorphic to $V_{n-1}(\CC^{n+k-1})$. If $n=2$, the fiber is thus
    $V_1(\CC^{1+k})\iso\CC^{1+k}-\{0\}\we S^{1+2k}$, which is certainly
    $2k$-connected. Therefore, by induction, we can assume that
    $V_{n-1}(\CC^{n+k-1})$ is $2k$-connected, and since $\CC^{n+k}-\{0\}\we
    S^{2n+2k-1}$ is $2n+2k-2$-connected, it follows that
    $V_n(\CC^{n+k})$ is $2k$-connected for $n\geq 1$.
\end{proof}

As a result, the colimit $V_n=\colim_k V_n(\CC^{n+k})$ is a contractible
$\GL_n(\CC)$-torsor over $\Gr_n$. Hence, $\Gr_n\we\BGL_n(\CC)$. There is a
fairly easy way to see why Grassmannians should control $\GL_n(\CC)$-torsors on
$X$, or equivalently vector bundles. Let $p:E\rightarrow X$ be a complex
vector bundle of rank $n$. Suppose that $E$ is a trivial on a finite cover
$\{U_i\}_{i=1}^m$ of $X$. Let $s_i:X\rightarrow[0,1]$ be a partition of unity
subordinate to $\{U_i\}$, so that the support of each $s_i$ is contained in
$U_i$, and $s_1+\cdots+s_m=1_X$. Choose trivializations $t_i:\CC^n\times
U_i\rightarrow E|_{U_i}$. Now, we can define $g:E\rightarrow\CC^{mn}$ by
$g=\oplus_{i=1}^m g_i$, where $g_i=(s_i\circ p)\cdot(p_1\circ t_i^{-1})$; outside
$U_i$, $g_i=0$. Now, the \df{Gauss map} $g$ clearly defines a map
$X\rightarrow\Gr_n(\CC^{nm})$. The entire formalism of vector bundles can be
based on these Gauss maps. See Husemoller~\cite{husemoller}*{Chapter 3}.

We conclude the section with two remarks. First, the classifying space
construction is functorial in homomorphisms of topological groups. That is, if
there is a map of topological groups $G\rightarrow H$, then there is an induced
map $\B G\rightarrow\B H$. The corresponding map
$\mathrm{Tors}_G(X)\rightarrow\mathrm{Tors}_H(X)$ is an example of the fiber
bundle construction. Indeed, since $G$ acts on $H$, we can apply this
construction to produce an $H$-torsor from a $G$-torsor.

\begin{example}
    The most important example for us will be the determinant map
    $\BGL_n(\CC)\rightarrow\BGL_1(\CC)$, which gives the determinant map
    $\Vect_n^{\topo}(X)\rightarrow\Vect_1^{\topo}(X)$. The fiber of this map is
    just as important. Indeed, because
    $$1\rightarrow\SL_n(\CC)\rightarrow\GL_n(\CC)\rightarrow\GL_1(\CC)\rightarrow
    1$$ is an exact sequence of topological groups, the sequence
    $\BSL_n(\CC)\rightarrow\BGL_n(\CC)\rightarrow\BGL_1(\CC)$ turns out to be a
    homotopy fiber sequence. Hence, $[X,\BSL_n(\CC)]$ classifies topological complex
    vector bundles on $X$ with trivial determinant.
\end{example}

The second remark is that one often works with $\BU_n$
rather than $\BGL_n(\CC)$. The natural inclusion
$\U_n\rightarrow\GL_n(\CC)$ of the unitary matrices into all invertible complex
matrices is a homotopy equivalence (using polar decomposition). Hence, $\BU_n\rightarrow\BGL_n(\CC)$ is
also a homotopy equivalence. Philosophically, this corresponds to the fact that
any complex vector bundle on a paracompact Hausdorff space admits a Hermitian
metric.

\subsection{Topological line bundles}

Using the Grassmannian description of $\BGL_1(\CC)$, we find that
$\colim_n\Gr_1(\CC^{n})\we\BGL_1(\CC)$. Of course,
$\Gr_1(\CC^{n})\we\CC\PP^{n-1}$. Hence, $\CC\PP^\infty\we\BGL_1(\CC)$.

\begin{lemma}
    The infinite complex projective space $\CC\PP^\infty$ is a $K(\ZZ,2)$.
\end{lemma}

\begin{proof}
    Indeed, since $\GL_1(\CC)\iso\CC^*\we S^1$ and
    $\Omega\CC\PP^\infty\we\GL_1(\CC)$, we find that $\CC\PP^\infty\we
    K(\ZZ,2)$.
\end{proof}

As a result we may describe the set of line bundles on $X$ in terms of a cohomology group:

\begin{corollary}
    The natural map $\Vect_1^{\topo}(X)\xrightarrow{c_1}\H^2(X,\ZZ)$ is
    a bijection for any paracompact Hausdorff space.
\end{corollary}

\subsection{Rank $2$ bundles in low dimension}

Recall that the cohomology of the infinite Grassmannian is
\begin{equation*}
    \H^*(\Gr_n,\ZZ)\iso\H^*(\BGL_n(\CC),\ZZ)\iso\ZZ[c_1,\ldots,c_n],
\end{equation*}
where $|c_i|=2i$. We will also need Bott's computation~\cite{bott}*{Theorem 5} of the homotopy groups
of $\Gr_n$ in the stable range. If $i\leq 2n+1$, then
\begin{equation*}
    \pi_i\Gr_n\iso\begin{cases}
        \ZZ &   \text{if $i\leq 2n$ is even,}\\
        0   &   \text{if $i\leq 2n$ is odd,}\\
        \ZZ/n!  &   \text{if $i=2n+1$.}
    \end{cases}
\end{equation*}
When $i>2n+1$ much less is known about the homotopy groups of $\Gr_n$ (except
when $n=1$). Playing the computation of the cohomology rings off of these
homotopy groups gives a great deal of insight into the low stages of the
Postnikov tower of $\Gr_n$.

\begin{lemma}
    The map $\Gr_n\rightarrow\Gr_n[3]\we\Gr_n[2]\we K(\ZZ,2)$ is precisely
    $c_1\in\H^2(\Gr_n,\ZZ)$.
\end{lemma}

\begin{proof}
    Indeed, this map is induced by the determinant map
    $\GL_n(\CC)\rightarrow\GL_1(\CC)$.
\end{proof}

We are in particular interested in the following part of the Postnikov tower of
$\Gr_2$:
\begin{equation*}
    \xymatrix{
        K(\ZZ/2,5)\ar[r]    &   \Gr_2[5]\ar[d]  &\\
        K(\ZZ,4)\ar[r]      &   \Gr_2[4]\ar[d]\ar[r]^{k_4}    &   K(\ZZ/2,6)\\
            &   \Gr_2[3]\we K(\ZZ,2)\ar[r]^{k_3}  &   K(\ZZ,5)
    }
\end{equation*}

Note that because $\Gr_2[3]\we\Gr_2[2]$ there is no obstruction to lifting a
map $X\rightarrow\Gr_2[2]$ to a map $X\rightarrow \Gr_2[3]$ if $X$ is a CW
complex.

\begin{lemma}
    The $k$-invariant $k_3$ is nullhomotopic. Hence,
    $$\Gr_2[4]\we K(\ZZ,2)\times K(\ZZ,4).$$ Moreover, this equivalence may be
    chosen so that the composition
    $\Gr_2\rightarrow\Gr_2[4]\rightarrow K(\ZZ,4)$ is $c_2\in\H^4(\Gr_2,\ZZ)$.
\end{lemma}

\begin{proof}
    The class $k_3\in\H^5(K(\ZZ,2),\ZZ)$ vanishes simply because the cohomology
    of $K(\ZZ,2)\we\CC\PP^\infty$ is concentrate in even degrees. This gives
    the splitting claimed (we lifting the identity map $K(\ZZ,2) \rightarrow Gr_2[2]$ up the Postnikov tower). Consider the map $\Gr_2\rightarrow K(\ZZ,2)\times
    K(\ZZ,4)$ classified by the pair $(c_1,c_2)$ in the cohomology of $\Gr_2$.
    By definition, $(c_1,c_2)$ factors through the functorial Postnikov section
    $\Gr_2\rightarrow\Gr_2[4]$. It is enough to check that the induced map
    $\Gr_2[4]\rightarrow K(\ZZ,2)\times K(\ZZ,4)$ is a weak equivalence. We have
    already seen that it is an isomorphism on $\pi_2$. We have a map of fiber
    sequences
    \begin{equation*}
        \xymatrix{
            \BSL_2[4]\ar[r]\ar[d]   &   \Gr_2[4]\ar[r]\ar[d]    &   \Gr_1\ar[d]\\
            K(\ZZ,4)\ar[r]          &   K(\ZZ,2)\times K(\ZZ,4)\ar[r]          &   K(\ZZ,2),
        }
    \end{equation*}
    and the outside vertical arrows are weak equivalences by the Hurewicz
    isomorphism theorem. This proves the lemma.
\end{proof}

We can now classify rank $2$ vector bundles on $4$-dimensional spaces.

\begin{proposition}
    Let $X$ be a $4$-dimensional space having the homotopy type of a CW
    complex. Then, the natural map
    \begin{equation*}
        \Vect_2^{\topo}(X)\rightarrow\H^2(X,\ZZ)\times\H^4(X,\ZZ)
    \end{equation*}
    is a bijection.
\end{proposition}

\begin{proof}
    The previous lemma shows that
    $[X,\Gr_2[4]]\rightarrow\H^2(X,\ZZ)\times\H^4(X,\ZZ)$ is a bijection. The
    obstruction to lifting a given map $f:X\rightarrow\Gr_2[4]$ to $\Gr_2[5]$ is a class 
    $f^*k_4\in\H^6(X,\ZZ/2)=0$. Similarly, the choice of lifts is bijective to
    a quotient of $\H^5(X,\ZZ/2)$, and this group is $0$. Hence, for every such $f$
    there is a unique lift to $\Gr_2[5]$, and then the same reasoning gives a
    unique lift to $\Gr_2[m]$ for all $m\geq 5$. Since $\Gr_2$ is the limit of
    its Postnikov tower, the proposition follows.
\end{proof}

If $\dim X=5$, the situation is similar but more complicated. To state the theorem let us recall that a cohomology operation is a natural transformation of functors $H^i(-, R) \rightarrow H^j(-, R')$; by Yoneda such a map is classified by an element of $[K(i, R), K(j, R')] \simeq H^j(K(i,R), R')$. 

\begin{proposition}
    If $X$ is a $5$-dimensional space having the homotopy type of a CW complex,
    then the map $\Vect_2^{\topo}(X)\rightarrow\H^2(X,\ZZ)\times\H^4(X,\ZZ)$ is
    surjective, and the choice of lifts is parametrized by
    $\H^5(X,\ZZ/2)/\mathrm{im}(\H^3(X,\ZZ)\rightarrow\H^5(X,\ZZ/2)),$ where the
    map $\H^3(X,\ZZ)\rightarrow\H^5(X,\ZZ/2)$ is a certain non-zero cohomology
    operation.
\end{proposition}

\begin{proof}
    Consider the fiber sequence
    $K(\ZZ/2,5)\rightarrow\Gr_2[5]\rightarrow\Gr_2[4]$. As above,
    $[X,\Gr_2[4]]$ is classified by the $1$st and $2$nd Chern classes. On a
    $5$-dimensional space, once a lift to $\Gr_2[5]$ is specified, there is a
    unique lift all the way to $\Gr_2$, just as in the proof of the previous
    proposition. The obstructions to finding a lift from $\Gr_2[4]$ to
    $\Gr_2[5]$ are in $\H^6(X,\ZZ/2)$, and
    hence all lift. Recall that to any fiber sequence there is an associated
    long exact sequence of fibrations. See~\cite{hatcher}*{Section~4.3}.
    Extending to the left a little bit, in our cases this is
    \begin{equation*}
        \Omega\Gr_2[4]\rightarrow
        K(\ZZ/2,5)\rightarrow\Gr_2[5]\rightarrow\Gr_2[4].
    \end{equation*}
    However, $K(\ZZ/2,5)\rightarrow\Gr_2[5]\rightarrow\Gr_2[4]$ is principal,
    so it extends to the right one term as well:
    \begin{equation*}
        \Omega\Gr_2[4]\rightarrow
        \Omega K(\ZZ/2,6)\rightarrow\Gr_2[5]\rightarrow\Gr_2[4]\rightarrow K(\ZZ/2,6),
    \end{equation*}
    where $\Omega K(\ZZ/2,6)\we K(\ZZ/2,5)$.
    It follows that there is an exact sequence of pointed sets
    $$\H^1(X,\ZZ)\times\H^3(X,\ZZ)\rightarrow\H^5(X,\ZZ/2)\rightarrow\Vect_2^{\topo}(X)\rightarrow\H^2(X,\ZZ)\times\H^4(X,\ZZ),$$
    which is surjective on the right. Moreover, the map
    $\H^1(X,\ZZ)\times\H^3(X,\ZZ)\rightarrow\H^5(X,\ZZ/2)$ is a group
    homomorphism because it is induced by taking loops of a map. There is an action of $\H^5(X,\ZZ/2)$ on $\Vect_2^{\topo}(X)$
    such that two rank $2$ vector bundles on $X$ have the same Chern classes if
    and only if they are in the same orbit of $\H^5(X,\ZZ/2)$.
    There are no cohomology operations
    $\H^1(X,\ZZ)\rightarrow\H^5(X,\ZZ/2)$, since $\H^5(S^1,\ZZ/2)=0$. However,
    there is a cohomology operation $\H^3(X,\ZZ)\rightarrow\H^5(X,\ZZ/2)$,
    often denoted $\mathrm{Sq}_\ZZ^2$. Note that this class is precisely
    $\Omega k_4$. That is, since we have $k_4:K(\ZZ,2)\times
    K(\ZZ,4)\rightarrow K(\ZZ/2,6)$, the loop space is $K(\ZZ,1)\times
    K(\ZZ,3)\rightarrow K(\ZZ/2,5)$. One can check, using the Postnikov tower
    and cohomology of $\BSL_2(\CC)$ that this class $\Omega k_4$ is precisely
    the unique non-zero element of $\H^5(K(\ZZ,3),\ZZ/2)\iso\ZZ/2$ by the next
    lemma.
\end{proof}

\begin{lemma}
    The $k$-invariant $k_4:\Gr_2[4]\rightarrow K(\ZZ/2,6)$ is non-trivial.
\end{lemma}

\begin{proof}
    It is enough to show that the corresponding $k$-invariant
    $\BSL_2(\CC)[4]\rightarrow K(\ZZ/2,6)$ is non-trivial. Note that
    $\BSL_2(\CC)\rightarrow\BSL_2(\CC)[4]\we K(\ZZ,4)$ is a $5$-equivalence and
    that $\BSL_2(\CC)\rightarrow\BSL_2(\CC)[5]$ is a $6$-equivalence. It
    follows that $\H^6(\BSL_2(\CC)[5],\ZZ/2)=0$ since
    $\H^*(\BSL_2(\CC),\ZZ/2)\iso\ZZ/2[c_2]$. On the other hand,
    $\H^6(\BSL_2(\CC)[4],\ZZ/2)\iso\H^6(K(\ZZ/4),\ZZ/2)\iso\ZZ/2$. If the
    extension $K(\ZZ/2,5)\rightarrow\BSL_2(\CC)[5]\rightarrow\BSL_2(\CC)[4]$
    were split, the cohomology of $\BSL_2(\CC)[4]$ would inject into the
    cohomology of $\BSL_2(\CC)[5]$. Since this does not happen, we see that
    $k_4$ is non-zero.
\end{proof}

\begin{exercise}
    Describe the obstruction class $k_4\in\H^6(X,\ZZ/2)$ by computing the
    cohomology of $K(\ZZ,2)\times K(\ZZ,4)$ and finding $k_4$.
\end{exercise}

Finally, if $\dim X=6$, there is a similar picture,
except that there is an obstruction to realizing a given pair of Chern classes,
and there is an additional choice of lift.

\begin{example}
    Let $X$ be the $6$-skeleton of $\BGL_3(\CC)$. There is a universal rank $3$
    vector bundle $E$ on $X$ with Chern classes $c_i(E)=c_i\in\H^{2i}(X,\ZZ)$ for
    $i=1,2,3$. On the other hand, we can ask if there is a rank $2$ bundle $F$
    on $X$ with Chern classes $c_i(F)=c_i$ for $i=1,2$. This is the
    universal example where the obstruction above is nonzero and demonstrates
    the \emph{incompressibility} of Grassmannians.
\end{example}

One can use the fact that $\SU_2\iso\SO_3$, which is itself isomorphic to the
$3$-sphere, to find that $\pi_6\Gr_2\iso\pi_5 S^3=\ZZ/2$. This leads to the
following description of rank $2$-bundles on a $6$-dimensional CW complex.

\begin{proposition}
    Let $X$ be a $6$-dimensional space with the homotopy type of a CW complex.
    The map $\Vect_2^{\topo}(X)\rightarrow\H^2(X,\ZZ)\times\H^4(X,\ZZ)$ has
    image precisely those pairs $(c_1,c_2)$ such that $k_4(c_1,c_2)=0$ in
    $\H^6(X,\ZZ)$. If $k_4(c_1,c_2)=0$, the set of lifts to $\Gr_2[5]$ is
    parameterized by a quotient of $\H^5(X,\ZZ/2)$ as above. Each lift then lifts to
    $\Gr_2$, and the set of lifts from $\Gr_2[5]$ to $\Gr_2$ is
    parametrized by a quotient of $\H^6(X,\ZZ/2)$.
\end{proposition}

\subsection{Rank $3$ bundles in low dimension}

\begin{proposition}
    Suppose that $X$ is a $5$-dimensional CW complex. Then, the natural map
    \begin{equation*}
        (c_1,c_2):\Vect_3^{\topo}(X)\rightarrow\H^2(X,\ZZ)\times\H^4(X,\ZZ)
    \end{equation*}
    is an isomorphism.
\end{proposition}

\begin{proof}
    Indeed, $\Gr_3[4]\we K(\ZZ,2)\times K(\ZZ,4)$, just as for $\Gr_2[4]$. But,
    this time, $\pi_5\Gr_3=0$. Hence, the next interesting problem is to lift
    from $\Gr_2[4]$ to $\Gr_2[6]$. The obstructions are in $\H^7(X,\ZZ)$, and
    hence vanish. The lifts of a given map to $\Gr_3[4]$ are a quotient of $\H^6(X,\ZZ)=0$.
\end{proof}

As a consequence, one sees immediately from the last section that every rank
$3$ vector bundle $E$ on a $5$-dimensional CW complex splits as $E_0\oplus\CC$
for some rank $2$ vector bundle $E_0$. This is one example of a more general
phenomenon we leave to the reader to discover.

\begin{proposition}
    If $X$ is a $6$-dimensional closed real orientable manifold, then
    \begin{equation*}
        (c_1,c_2,c_3):\Vect_3^{\topo}(X)\rightarrow\H^2(X,\ZZ)\times\H^4(X,\ZZ)\times\H^6(X,\ZZ)
    \end{equation*}
    is a injection with image the triples with $c_3$ an even multiple of a
    generator of $\H^6(X,\ZZ)\iso\ZZ$.
\end{proposition}

\begin{proof}
    As above, once we have constructed a map $X\rightarrow\Gr_3[6]$, there is a
    unique lift to $\Gr_3$. Given a map $X\rightarrow\Gr_3[4]\we K(\ZZ,2)\times
    K(\ZZ,4)$, the obstruction to lifting to $\Gr_3[6]$ is a class in
    $\H^7(X,\ZZ)=0$ since $X$ is $6$-dimensional.
    We have again an exact sequence of pointed sets
    \begin{equation*}
        \H^1(X,\ZZ)\times\H^3(X,\ZZ)\rightarrow\H^6(X,\ZZ)\rightarrow\Vect_3^{\topo}(X)\rightarrow\H^2(X,\ZZ)\times\H^4(X,\ZZ).
    \end{equation*}
    The map on the left is induced from a map $K(\ZZ,1)\times
    K(\ZZ,3)\rightarrow K(\ZZ,6)$ which is $\Omega k_5$, where $k_5$ is the
    $k$-invariant $K(\ZZ,2)\times K(\ZZ,4)\rightarrow K(\ZZ,7)$. In particular,
    the image in $\H^6(X,\ZZ)$ consists of torsion classes. But,
    $\H^6(X,\ZZ)\iso\ZZ$ by hypothesis.
    One can check that the composition
    $K(\ZZ,6)\rightarrow\Gr_3[6]\xrightarrow{c_3} K(\ZZ,6)$ is multiplication
    by $2$. This completes the proof.
\end{proof}

In general, understanding vector bundles of a fixed dimension becomes more and
more difficult as the dimension of the base space increases. The systematic
approach to this kind of problem uses cohomology and Serre spectral sequences
to determine Postnikov extensions one step a time. For an overview,
see~\cite{thomas}.

% Model categories
\section{The construction of the $\AA^1$-homotopy category}\label{sec:construction}

The first definitions of $\AA^1$-homotopy theory were given
in~\cite{morel-voevodsky} when the base scheme $S$ is noetherian of finite
Krull dimension. An equivalent homotopy theory was constructed by
Dugger~\cite{dugger-universal}, and we will follow Dugger's definition, but
with the added generality of allowing $S$ to be quasi-compact and
quasi-separated using Lurie's Nisnevich
topology~\cite{lurie-sag}*{Section~A.2.4}.
We use model
categories for the construction, but in the Section~\ref{sec:basic}, where we give many
properties of the homotopy theory, we emphasize the model-independence of the
proofs.

\subsection{Model categories}

Model categories are a technical framework for working up to homotopy.
The axioms guarantee that certain category-theoretic localizations
exist without enlarging the ambient set-theoretic universe and that it is possible in some sense to
compute the hom-sets in the localization. The theory generalizes the use of
projective or injective resolutions in the construction of derived categories
of rings or schemes.

References for this material include Quillen's original book on the
theory~\cite{quillen-homotopical}, Dwyer-Spalinski~\cite{dwyer-spalinski}, Goerss-Jardine~\cite{goerss-jardine},
and Goerss-Schemmerhorn~\cite{goerss-schemmerhorn}.
For consistency, we refer the reader where possible to~\cite{goerss-jardine}.
However, unlike some of these references, we assume that the category
underlying $M$ has all small limits and colimits. This is satisfied immediately in all
cases of interest to us.

\begin{definition}
    Let $M$ be a category with all small limits and colimits.
    A model category structure on $M$ consists of three classes
    $W,C,F$ of morphisms in $M$, called \df{weak equivalences},
    \df{cofibrations}, and
    \df{fibrations}, subject to the following set of axioms.
    \begin{enumerate}
        \item[{\bf M1}]   Given $X\xrightarrow{f}Y\xrightarrow{g}Z$ two composable
            morphisms in $M$, if any two of $g\circ f$, $f$, and $g$ are weak
            equivalences, then so is the third.
        \item[{\bf M2}]   Each class $W,C,F$ is closed under retracts.
        \item[{\bf M3}]   Given a diagram
            \begin{equation*}
                \xymatrix{
                    Z\ar[r]\ar[d]^i   &   E\ar[d]^p\\
                    X\ar[r]\ar@{.>}[ur]  &   B
                }
            \end{equation*}
            of solid arrows, a dotted arrow can be found making the diagram
            commutative if either
            \begin{enumerate}
                \item   $p$ is an \df{acyclic fibration} ($p\in W\cap F$) and $i$ is a
                    cofibration, or
                \item   $i$ is an \df{acyclic cofibration} ($i\in W\cap C$) and $p$ is a fibration.
            \end{enumerate}
            (In particular, cofibrations $i$ have the \df{left lifting property} with respect to acyclic fibrations, while fibrations $p$
            have the \df{right lifting property} with respect to acyclic cofibrations.)
        \item[{\bf M4}]   Any map $X\rightarrow Z$ in $M$ admits two factorizations
            $X\xrightarrow{f}E\xrightarrow{p}Z$ and
            $X\xrightarrow{i}Y\xrightarrow{g}Z$, such that $f$ is an acyclic
            cofibration, $p$ is a fibration, $i$ is a cofibration, and $g$ is
            an acyclic fibration.
    \end{enumerate}
\end{definition}

\begin{remark}
    In practice, a model category is determined by only $W$ and either $C$ or $F$. Indeed,
    $C$ is precisely the class of maps in $M$ having the left lifting property
    with respect to acyclic fibrations. Similarly, $F$ consists of exactly
    those maps in $M$ having the right lifting property with respect to acyclic
    cofibrations. The reader can prove this fact using the axioms or refer to~\cite{dwyer-spalinski}*{Proposition 3.13}. However,
    some caution is required. While one often sees model categories specified
    in the literature by just fixing $W$ and either $C$ or $F$, it usually has
    to be checked that these really do give $M$ a model category structure.
\end{remark}

\begin{remark}
    Many authors strengthen {\bf M4} to assume the existence of
    \emph{functorial} factorizations. This is satisfied in all model categories
    of relevance for this paper by~\cite{hovey}*{Section~2.1} as they are all cofibrantly generated.
\end{remark}

\begin{exercise}\label{exercise:chain}
    Let $A$ be an associative ring. Consider $\mathrm{Ch}_{\geq 0}(A)$, the category of
    non-negatively graded chain complexes of right $A$-modules. Since limits and colimits of chain
    complexes are computed degree-wise, $\mathrm{Ch}_{\geq 0}(A)$ is closed under all small
    limits and colimits. Let $W$ be the class of quasi-isomorphisms, i.e.,
    those maps $f:M_\bullet\rightarrow N_\bullet$ of chain complexes such that
    $\H_n(f):\H_n(M_\bullet)\rightarrow\H_n(N_\bullet)$ is an isomorphism for
    all $n\geq 0$. Let $F$ be the class of maps of chain complexes which are
    surjections in positive degrees. Describe
    the class $C$ of maps satisfying the left lifting property with respect to
    $F\cap W$. Prove that $W,C,F$ is a model category structure on
    $\mathrm{Ch}_{\geq 0}(A)$.
\end{exercise}

% \begin{solution}
%     Finally, let
%     $C$ be the class of degree-wise injective maps with projective cokernel in
%     each degree. It is a straightforward exercise to check that these three
%     classes provide $\mathrm{Ch}_A$ with a model category structure.
% \end{solution}

\begin{definition}
    A model category $M$ has an initial object $\emptyset$ and a final object
    $\ast$, since it is closed under colimits and limits.
    An object $X$ of $M$ is \df{fibrant} if $X\rightarrow\ast$ is a fibration, and
    $X$ is \df{cofibrant} if $\emptyset\rightarrow X$ is a cofibration.
    Given an object $X$ of $M$, an acyclic fibration $QX\rightarrow X$ such
    that $QX$ is cofibrant is called a \df{cofibrant replacement}. Similarly, if
    $X\rightarrow RX$ is an acyclic fibration with $RX$ fibrant, then $RX$ is
    called a \df{fibrant replacement} of $X$. These replacements always exist,
    by applying {\bf M4} to $\emptyset\rightarrow X$ or $X\rightarrow\ast$.
\end{definition}

\begin{example}
    In $\mathrm{Ch}_{\geq 0}(A)$, let $M$ be a right $A$-module (viewed as a
    chain complex concentrated in degree zero). A projective resolution
    $P_\bullet\rightarrow M$ is an example of a cofibrant replacement. Indeed, such a
    resolution is an acyclic fibration. Moreover, the map $0\rightarrow
    P_\bullet$ is a cofibration, since the cokernel is projective in each
    degree.
\end{example}

\begin{example} \label{ex:ssets}
    Let $\sSets$ be the category of simplicial sets. This is the category of
    functors $\Delta^{\op}\rightarrow\Sets$, where $\Delta$ is the category
    of finite non-empty ordered sets. (For details, see~\cite{goerss-jardine}.) There is a geometric realization functor
    $\sSets\rightarrow\Spc$, which sends a simplicial set $X_\bullet$ to a
    space $|X_\bullet|$. Let $W$ denote the class of weak homotopy equivalences
    in $\sSets$, i.e., those maps $f:X_\bullet\rightarrow Y_\bullet$ such that
    $|f|:|X_\bullet|\rightarrow|Y_\bullet|$ is a weak homotopy equivalence. Let $C$
    denote the class of level-wise monomorphisms. If $F$ is the class of maps
    having the right lifting property with respect to acyclic cofibrations, then
    $\sSets$ together with $W,C,F$ is a model category. In $\sSets$, every
    object is cofibrant. The fibrant objects are the \df{Kan complexes}, namely
    those simplicial sets having a filling property for all horns.
    See~\cite{goerss-jardine}*{Section I.3}.
\end{example}

\begin{definition}
    A model category $M$ is \df{pointed} if the natural map $\emptyset\rightarrow\ast$ is an isomorphism.
    Examples of pointed model categories include $\mathrm{Ch}_A^{\geq 0}$,
    which is pointed by the $0$ object, and $\sSets_\star$, the category of
    \emph{pointed} simplicial sets.
\end{definition}

Now, we come to the main reason why model categories have been so successful in
encoding homotopical ideas: the homotopy category of a model category.

\begin{definition}
    Let $M$ be a category and $W$ a class of morphisms in $M$. The localization of
    $M$ by $W$, if it exists, is a category $M[W^{-1}]$ with a functor
    $L:M\rightarrow M[W^{-1}]$ such that
    \begin{enumerate}
        \item $L(w)$ is an isomorphism for every $w\in W$,
        \item every functor $F:M\rightarrow N$ having the property that $F(w)$ is an isomorphism for all $w\in W$ factors
            uniquely through $L$ in the sense that there is a functor
            $G:M[W^{-1}]\rightarrow N$ and a natural isomorphism of functors $G\circ L\we
            F$, and
        \item for any category $N$, the functor $\Fun(M[W^{-1}],N)\rightarrow\Fun(M,N)$ induced by composition with $L:M\rightarrow M[W^{-1}]$
        is fully faithful.
     \end{enumerate}
    The localization of $M$ by $W$, if it exists, is unique up to categorical equivalence.
\end{definition}

In general, there is no reason that a localization of $M$ by $W$ should exist
much less be useful. The fundamental problem is that
in attempting to concretely construct the morphisms in $M[W^{-1}]$, for example
by hammock localization (hat piling), one discovers size issues, where it might be necessary
to enlarge the universe in order to obtain a category: the morphisms
sets in a category must be actual sets, not proper classes.

\begin{theorem}[\cite{quillen-homotopical}]
    Let $M$ be a model category with class of weak equivalences $W$. Then, the localization $M[W^{-1}]$ exists.
    It is called the homotopy category of $M$, and we will denote it by $\Ho(M)$.
\end{theorem}

\begin{recipe} \label{recipe1}
    It is generally difficult to compute $[X,Y]=\Hom_{\Ho(M)}(X,Y)$ given two
    objects $X,Y\in M$. We give a recipe. Replace $X$ by a weakly
    equivalent cofibrant object $QX$, and $Y$ by a weakly equivalent fibrant object
    $RY$. Then, $[X,Y]=\Hom_M(QX,RY)/\sim$, where $\sim$ is an equivalence relation
    on $\Hom_M(QX,RY)$ generalizing homotopy equivalence
    (see~\cite{goerss-jardine}*{Section II.1}).
    See~\cite{dwyer-spalinski}*{Proposition 5.11} for a proof that this
    construction does indeed compute the set of maps in the homotopy
    category.
\end{recipe}

\begin{remark}
    In many cases, every object of $M$ might be cofibrant, in which case one
    just needs to replace $Y$ by $RY$ and compute the homotopy classes of maps.
    This is for example the case in $\sSets$.
\end{remark}

\begin{remark}
    In Goerss-Jardine~\cite{goerss-jardine}*{Section II.1},
    the homotopy category $\Ho(M)$ is itself defined to be the category of
    objects of $M$ that are both fibrant and cofibrant, with maps given by
    $\Hom_{\Ho(M)}(A,B)=\Hom(A,B)/\sim$. Given an arbitrary $X$ in $M$ it is
    possible to assign to $X$ a fibrant-cofibrant object $RQX$ as follows.
    First, take, via {\bf M4}, a factorization $\emptyset\rightarrow QX\rightarrow X$ where
    $QX$ is cofibrant $QX\rightarrow X$ is a weak equivalence. Now, take a
    factorization $QX\rightarrow RQX\rightarrow\ast$ of the canonical map
    $QX\rightarrow\ast$ in which $QX\rightarrow RQX$ is an acyclic cofibration
    and $RQX\rightarrow\ast$ is a fibration. In particular, $RQX$ is fibrant.
    Since compositions of cofibrations are cofibrations, $RQX$ is also
    cofibrant. Moreover, if $f:X\rightarrow Y$ is a morphism, then it is
    possible using {\bf M3} to (non-uniquely) assign to $f$ a morphism $RQf:RQX\rightarrow
    RQY$ such that one gets a well-defined functor $M\rightarrow\Ho(M)$ (i.e.,
    after enforcing $\sim$).
\end{remark}

\begin{remark}
    In practice, we will work with simplicial model category structures, for which
    there exist objects $QX\times{\Delta^1}$, where $\Delta^1$ is the standard
    $1$-simplex (so that $|\Delta^1|=I^1$). In this case, the equivalence relation
    $\sim$ is precisely that of (left) homotopy classes of maps. See
    Definition~\ref{def:sm7}.
\end{remark}

\begin{exercise}
    For chain complexes, the equivalence relation $\sim$ is precisely that of
    chain homotopy equivalence. (See~\cite{weibel}*{Section 1.4}.)
    Using the recipe above, compute
    $$\Hom_{\Ho(\Ch_{\geq 0}(\ZZ))}(\ZZ/p,\ZZ[1]),$$
    where $\ZZ/p[1]$ denotes the chain complex with $\ZZ/p$ placed in degree $1$ and zeros elsewhere.
\end{exercise}

\subsection{Mapping spaces}

We will now explain simplicial model categories since we will need to discuss
mapping spaces.  For details, we refer the reader to
~\cite{goerss-jardine}*{II.2-3}.
If $X$ and $Y$ are simplicial sets,
then we may define the \textbf{simplicial mapping space} $\map_{\sSets}(X,Y)$ as the
simplicial set with $n$-simplices given by $$\map_{\sSets}(X,Y)_n := \Hom_{\sSets}(X
\times \Delta^n, Y).$$
This simplicial set fits into a tensor-hom adjunction given by
$$\Hom_{\sSets}(Z \times X, Y)  \iso \Hom_{\sSets}(Z, \map_{\sSets}(X, Y)).$$ Indeed,
from this adjunction we may deduce the formula for $\map(X,Y)_n$ by evaluating
at $Z=\Delta^n$.

Abstracting these formulas, one arrives at the axioms for a \textbf{simplicial category}~\cite{goerss-jardine}*{II Definition 2.1}.
A simplicial category is a category $M$ equipped with
\begin{enumerate}
\item a \textbf{mapping space functor}: $\map: M^{op} \times M \rightarrow
    \sSets$, written $\map_M(X,Y)$,
\item an \textbf{action} of $\sSets$, $M \times \sSets \rightarrow
    M$, written $X\otimes S$, and
\item an \textbf{exponential}, $\sSets^{\op} \times M \rightarrow M$, written
    $X^S$ for an object $X\in M$ and a simplicial set $S$
\end{enumerate}
subject to certain compatibilities. The most important are that $$- \otimes X: \sSets
\leftrightarrows C: \map_M(X, -)$$ should be an adjoint pair of functors and that
$\Hom_M(X, Y) \iso \map(X,Y)_0$ for all $X,Y\in M$.

Suppose that $M$ is a simplicial category simultaneously equipped with a model
structure. We would like the simplicial structure above to play well with the
model structure. For example, if $i: A \rightarrow X$ is a cofibration, we
expect $\map_M(Y, A) \rightarrow \map_M(Y, X)$ to be a fibration (and hence induce
long exact sequences in homotopy groups) for any object $Y$ as is the case in
simplicial sets.

\begin{definition}\label{def:sm7}
    Suppose that $M$ is a model category which is also a
    simplicial category. Then $M$ satisfies \df{SM7}, and is called a
    \textbf{simplicial model category}, if for any cofibration $i: A \rightarrow
    X$ and any fibration: $p: E \rightarrow B$ the map of simplicial sets
    (induced by the functoriality of $\map$)
    $$\map_M(X, E) \rightarrow \map_M(A, E) \times_{\map_M(A, B)} \map_M(X, B)$$ is a
    fibration of simplicial sets which is moreover a weak equivalence if either $i$ or $p$ is.
\end{definition}

\begin{exercise} Show that in a simplicial model category $M$,
    if $A \rightarrow X$ is a cofibration, then for any object $Y$, the natural
    map $\map_M(Y, A) \rightarrow \map_M(Y, X)$ is a fibration of simplicial sets.
\end{exercise}

Another feature of simplicial model categories is the fact that one may define
a concept of homotopy that is more transparent than in an ordinary model
category (where one defines left and right homotopies, see \cite{dwyer-spalinski}).
Suppose that $A \in M$ is a cofibrant object, then we say that two morphisms
$f, g: A \rightarrow X$ are homotopic if there is a morphism: $H: A \otimes
\Delta^1 \rightarrow X$ such that 
$$\xymatrix{
A \coprod A \ar[r]^{d_1 \coprod d_0} \ar[d]_{f \coprod g}  & A \otimes \Delta^1 \ar[dl]_{H}\\
X &}$$ commutes. Write $f\sim g$ if $f$ and $g$ are homotopic.

\begin{exercise} Prove that $\sim$ is an equivalence relation on
    $\Hom_M(A,X)$ when $A$ is cofibrant.
\end{exercise}

In~\ref{recipe1} we stated a recipe for calculating $[X,Y]$, the
hom-sets in $\Ho(M)$. We replace $X$ by a weakly
equivalent cofibrant object $QX$, and $Y$ by a weakly equivalent fibrant object
$RY$. Then, we claimed that $[X,Y]=\Hom_M(QX,RY)/\sim$ where $\sim$ was an
unspecified equivalence relation.  For a simplicial model category, this
equivalence relation can be taken to be the one just given. The fact the this is well defined
is checked in \cite{goerss-jardine}*{Proposition 3.8}.

%In the case we are most interested in, simplicial presheaves, we can regard the mapping space $\map(\Fscr, \Gscr)$ as the simplicial presheaf that represents the functor:

%$$ U \mapsto \Hom_{\sPre(\Sm_S)}(U \times \Fscr, \Gscr)$$

%\begin{proposition} The simplicial pre-sheaf $\map(\Fscr, \Gscr)$ exists!
%\end{proposition}

% \subsection{Technical conditions on model categories}
% 
% \begin{definition}
%     A model category $M$ is \df{combinatorial} if\ldots
% \end{definition}

% \subsection{Functor categories}
% 
% One way of creating new categories is to fix a small category $J$ and a
% category $M$ and considering $\Pre_M(J)$, the category of presheaves on $J$
% with values in $M$. If $M$ is combinatorial model category, then there are
% two natural model category structures on $\Pre_M(J)$.
% 
% One is the \df{projective model category structure}. Here the class of weak
% equivalences in $\Pre_M(J)$ consists of those morphisms $f:X\rightarrow Y$ of
% presheaves such that $f(j):X(j)\rightarrow Y(j)$ is a weak equivalence in $M$
% for every object $j\in J$. The class of fibrations in $\Pre_M(J)$ are similarly those $f$
% such that $f(j)$ is a fibration for each $j$. The cofibrations are those maps
% having the left lifting property with respect to acyclic fibrations. This turns
% out to specify a model category structure on $\Pre_M(J)$~\cite{}.
% 
% The other is the \df{injective model category structure}. The weak equivalences
% are as above, but this time it is the cofibrations that are determined
% pointwise.

\subsection{Bousfield localization of model categories}

One way of creating new model categories from old is via Bousfield
localization. The underlying category remains the same, while the class of weak
equivalences is enlarged. To describe these localizations, we first need to
consider a class of functors between model categories that are well-adapted to
their homotopical nature.

\begin{definition}
    Consider a pair of adjoint functors $$F:M\rightleftarrows N:G$$ between
    model categories $M$ and $N$. The pair is called a \df{Quillen pair}, or a
    pair of Quillen functors, if one of the following equivalent conditions is
    satisfied:
    \begin{itemize}
        \item $F$ preserves cofibrations and acyclic cofibrations;
        \item $G$ preserves fibrations and acyclic fibrations.
    \end{itemize}
    In this case, $F$ is also called a \df{left Quillen functor}, and $G$ a
    \df{right Quillen functor}.
\end{definition}

Quillen pairs provide a sufficient framework for a pair of adjoint functors on
model categories to descend to a pair of adjoint functors on the homotopy categories.

\begin{proposition}\label{prop:quillenexistence}
    Suppose that  $F:M\rightleftarrows N:G$ is a pair of Quillen functors.
    Then, there are functors $\Lbf F:M\rightarrow\Ho(N)$ and $\Rbf
    G:N\rightarrow\Ho(M)$, each of which takes weak equivalences to
    isomorphisms, such that there is an induced adjunction $\Lbf F:\Ho(M)\rightleftarrows \Ho(N):\Rbf G$ between
    homotopy categories.
\end{proposition}

\begin{proof}
    See~\cite{dwyer-spalinski}*{Theorem~9.7}.
\end{proof}

\begin{remark}
    The familiar functors from homological algebra all arise in this way, so
    $\Lbf F$ is called the left derived functor of $F$, while $\Rbf G$ is the
    right derived functor of $G$. There is a recipe for computing the value of
    the derived functors on an arbitrary object $X$ of $M$ and $Y$ of $N$.
    Specifically, $\Lbf F(X)$ is weakly equivalent to $F(QX)$ where
    $QX$ a cofibrant replacement of $X$. Similarly,
    $\Rbf G(Y)$ is weakly equivalent to $G(RY)$ where $RY$ is a
    fibrant replacement of $Y$.
\end{remark}

\begin{remark}
    It follows from the previous remark that when a functorial cofibrant
    replacement functor $Q:M\rightarrow M$ exists, then we can factor $\Lbf
    F:M\rightarrow\Ho(N)$ through
    $M\xrightarrow{Q}M\xrightarrow{F}N\rightarrow\Ho(N)$. As mentioned above, this is the case for
    all model categories in this paper. As such, we will
    often abuse notation and write $\Lbf F$ for the functor $F\circ
    Q:M\rightarrow N$.
\end{remark}

\begin{definition}
    A \df{Quillen equivalence} is a Quillen pair $F:M\rightleftarrows N:G$ such that
    $\Lbf F:\Ho(M)\rightleftarrows \Ho(N):\Rbf G$ is an inverse equivalence.
\end{definition}

\begin{definition}
    Let $M$ be a simplicial model category with class of weak equivalences $W$. Suppose
    that $I$ is a \emph{set} of maps in $M$. An object $X$ of $M$ is
    \df{$I$-local} if it is fibrant and if
    for all $i:A\rightarrow B$ with $i\in I$, the induced morphism on mapping
    spaces $i^*:\map_M(B,X)\rightarrow\map_M(A,X)$ is a weak equivalence (of
    simplicial sets). A morphism $f:A\rightarrow B$ is an \df{$I$-local weak equivalence}
    if for every $I$-local object $X$, the induced morphism on mapping spaces
    $f^*:\map_M(B,X)\rightarrow\map_M(A,X)$ is a weak equivalence. Let $W_I$ be the
    class of all $I$-local weak equivalences. By using \df{SM7}, $W\subseteq I$.

    Let $F_I$ denote the class of maps satisfying the
    right lifting property with respect to $W_I$-acyclic cofibrations ($W_I\cap
    C$). If $(W_I,C,F_I)$ is a model category structure on $M$, we call this the
    \df{left Bousfield localization} of $M$ with respect to $I$.

    To distinguish between the model category structures on $M$, we will write
    $\L_I M$ for the left Bousfield model category structure on $M$. We will
    only write $\L_I M$ when the classes of morphisms defined above do define a
    model category structure.
\end{definition}

When it exists, the Bousfield localization of $M$ with respect to $I$ is
universal with respect to Quillen pairs $F:M\rightleftarrows N:G$ such that
$\Lbf F(i)$ is a weak equivalence in $N$ for all $i\in I$.

\begin{exercise}
    Show that if it exists, then the identity functors $\id_M:M\rightleftarrows
    M:\id_M$ induce a Quillen pair between $M$ (on the left) and $\L_I M$.
\end{exercise}

We want to quote an important theorem asserting that in good cases the left
Bousfield localization of a model category with respect to a set of morphisms
exists. Some conditions, which we now define, are needed on the model category.

\begin{definition}\label{def:leftproper}
    A model category $M$ is \df{left proper} if pushouts of weak equivalences
    along cofibrations are weak equivalences.
\end{definition}

Note that this is a condition about how weak equivalences and cofibrations
behave with respect to ordinary categorical pushouts.
Model categories in which all objects are cofibrant are left
proper~\cite{htt}*{Proposition~A.2.4.2}.

The next condition we need is for $M$
to be \df{combinatorial}. This definition, due to Jeff Smith, is rather
technical, so we leave it to the interested reader to refer
to~\cite{htt}*{Definition~A.2.6.1}.
Recall that a category is \df{presentable} if it has all small colimits and is
$\kappa$-compactly generated for some regular cardinal $\kappa$. For details,
see the book of Ad\'amek-Rosicky~\cite{adamek-rosicky}, although note that they
call this condition \emph{locally presentable}. We keep Lurie's terminology for the
sake of consistency. The most
important thing to know about combinatorial model categories for the purposes
of this paper is that they are presentable as categories.

\begin{exercise}
    Show that the model category structure on $\Ch_{\geq 0}(A)$ of
    Exercise~\ref{exercise:chain} is left proper.
\end{exercise}

\begin{theorem}\label{thm:existence}
    If $M$ is a left proper and combinatorial simplicial model category and $I$
    is a set of morphisms in $M$, then the left Bousfield localization $\L_I M$
    exists and inherits a simplicial model category structure from $M$.
\end{theorem}

\begin{proof}
    This is~\cite{htt}*{Proposition~A.3.7.3}.
\end{proof}

We refer to~\cite{hirschhorn}*{Proposition~3.4.1} for the next result, which
identifies the fibrant objects in the Bousfield localization.

\begin{proposition}\label{prop:ilocal}
    If $M$ is a left proper simplicial model category and $I$ is a set of maps such
    that $\L_I M$ exists as a model category, then the fibrant objects of $\L_I
    M$ are precisely the $I$-local objects of $M$.
\end{proposition}

\begin{exercise}
    Consider the model category structure given in Exercise~\ref{exercise:chain} on
    $\Ch_{\geq 0}(\ZZ)$. It is not hard to show that this is a simplicial
    model category using the Dold-Kan correspondence
    (see~\cite{goerss-jardine}).
    Let $I$ be the set of all morphisms between chain complexes of finitely
    generated abelian groups inducing isomorphisms on
    \emph{rational} homology groups.
    Then, $\L_I\Ch_{\geq 0}(\ZZ)$ is Quillen
    equivalent to $\Ch_{\geq 0}(\QQ)$ with the model category structure of
    Exercise~\ref{exercise:chain}. Show that every rational homology
    equivalence is an isomorphism in $\Ho(\L_I\Ch_{\geq 0}(\ZZ))$.
\end{exercise}

\begin{exercise}
    Construct a category of $\QQ$-local spaces, by letting $I$ be a set of
    maps $f:X\rightarrow Y$ of simplicial sets such that $\H_*(f,\QQ)$ is an
    isomorphism.
\end{exercise}

\subsection{Simplicial presheaves with descent}

Let $C$ be an essentially small category. Let $\sPre(C)$ denote
the category of functors $X:C^{\op}\rightarrow\sSets$. This is the category of
\df{simplicial presheaves} on $C$, and there is a Yoneda functor
$h:C\rightarrow\sPre(C)$. Bousfield and Kan~\cite{bousfield-kan} defined a model category
structure on $\sPre(C)$, the \df{projective model category structure}, which has a
special universal property highlighted by Dugger~\cite{dugger-universal}: it is
the initial model category into which $C$ embeds. Consider the following classes of morphisms in
$\sPre(C)$:
\begin{itemize}
    \item   objectwise weak equivalences: those maps $w:X\rightarrow Y$ such that $w(V):X(V)\rightarrow
        Y(V)$ is a weak equivalence of simplicial sets for all objects $V$ of $C$,
    \item   objectwise fibrations, and
    \item   projective cofibrations, those maps having the left lifting property with
        respect to acyclic objectwise fibrations.
\end{itemize}

\begin{proposition}
    The category of simplicial presheaves with the weak equivalences,
    fibrations, and cofibrations as above is a left proper combinatorial
    simplicial model category.
\end{proposition}

\begin{proof}
    The reader can find a proof in~\cite{htt}*{Proposition~A.2.8.2}.
    See~\cite{htt}*{Remark~A.2.8.4} for left properness.
\end{proof}

Suppose now that $C$ has a Grothendieck topology $\tau$. Let $U_\bullet$ be an object
of $\sPre(C)$ (so that each $U_n$ is a presheaf of \emph{sets} on $C$),
and suppose that there is a map $U\rightarrow V$, where $V$ is a
representable object.
We call $U\rightarrow V$ a \df{hypercover} if each $U_n$ is a coproduct of
representables, the induced map $U_0\rightarrow V$ is a $\tau$-cover, and
each $U^{\Delta^n}\rightarrow U^{\partial\Delta^n}$ is a
$\tau$-cover in degree $0$. For details about hypercovers,
see~\cite{artin-mazur}*{Section 8}. Except for the definition of the $\tau$-local
category below, we will not need hypercovers in the rest of the paper. The
reader may safely just imagine these to be \v{C}ech complexes.

The standard example of a hypercover is the \df{\v{C}ech} complex
$\check{U}\rightarrow V$ associated to a $\tau$-cover $U\rightarrow V$.
So, $\check{U}_n=U\times_V\cdots\times_V U$, the product of $U$ with
itself $n+1$ times over $V$. Roughly speaking, a hypercover looks just like a
\v{C}ech complex, except that
one is allowed to refine the \v{C}ech simplicial object by iteratively taking
covers of the fiber products.

\begin{theorem}
    The Bousfield localization of $\sPre(C)$ with respect to the class of
    hypercovers
    \begin{equation*}
        \check{U}_\bullet\rightarrow V
    \end{equation*}
    exists. We will denote this model category throughout the paper by
    $\L_\tau\sPre(C)$.
\end{theorem}

\begin{proof}
    By Theorem~\ref{thm:existence}, we only have to remark that there is up to isomorphism only a set of
    $\tau$-hypercovers since $C$ is small.
\end{proof}

\begin{remark}\label{rem:fibrant}
    We will refer to $\tau$-local objects and $\tau$-local weak equivalences
    for the $I$-local notions when $I$ is the class of morphisms in the
    theorem. In the $\tau$-local model category $\L_\tau\sPre(C)$, an object $V$ of $C$
    (viewed as the functor it represents) is equivalent to the \v{C}ech complex
    of any $\tau$-covering. Since $\sPre(C)$ with its projective model category
    structure is left proper, the fibrant objects of $\L_\tau\sPre(C)$ are
    precisely the $\tau$-local objects by Proposition~\ref{prop:ilocal}.
    Hence, the fibrant objects of $\L_\tau\sPre(C)$ are precisely the
    presheaves of Kan complexes $X$ such that
    \begin{equation}\label{eq:descent}
        X(V)\rightarrow\holim_\Delta X(U)
    \end{equation}
    is a weak equivalence for every $\tau$-hypercover
    $U\rightarrow A$. In other words, the fibrant objects are the
    \df{homotopy sheaves of spaces}.
\end{remark}

There is another, older definition of the homotopy theory of $\tau$-homotopy
sheaves due to Joyal and Jardine. It is useful to know
that it is Quillen equivalent to the one given above.

\begin{definition} \label{defin:hoshv}
    Let $X$ be an object of $\sPre(C)$, $V$ an object of $C$, and $x\in X(V)$ a
    basepoint. We can define a
    presheaf of sets (or groups or abelian groups) $\pi_n(X,x)$ on $C_{/V}$,
    the category of objects in $C$ over $V$, by letting
    $$\pi_n(X,x)(U)=\pi_n(X(U),f^*(x))$$ for $g:U\rightarrow V$ an object of
    $C_{/V}$. Let
    $\pi_n^{\tau}(X,x)$ be the sheafification of $\pi_n(X,x)$ in the $\tau$-topology
    restricted to $C_{/V}$. These are the \df{$\tau$-homotopy sheaves} of $X$.
\end{definition}

Let $W_\tau$ denote the class of maps $s:X\rightarrow Y$ in $\sPre(C)$ such that
$s_*:\pi_n^{\tau}(X,x)\rightarrow \pi_n^{\tau}(Y,s(x))$ is an isomorphism for all
$V$ and all basepoints $x\in X(V)$. Jardine proved that together with $W_\tau$, the class of
objectwise cofibrations determines a model category structure $\sPre_J(C)$ on $\sPre(C)$.

\begin{theorem}[Dugger-Hollander-Isaksen~\cite{dugger-hollander-isaksen}]\label{thm:jardine}
    The identity functor $\L_\tau\sPre(C)\rightarrow\sPre_J(C)$ is a Quillen
    equivalence.
\end{theorem}

\begin{example}
    A $\tau$-sheaf of sets on $C$, when viewed as a presheaf of simplicial
    sets, is in particular fibrant. It follows that when $\tau$ is subcanonical (i.e., every
    representable presheaf is in fact a sheaf) the Yoneda embedding
    $C\rightarrow\sPre(C)$ factors through the category of fibrant objects for
    the $\tau$-local model category on $\sPre(C)$. Thus, there is a fully
    faithful Yoneda
    embedding $C\rightarrow\Ho(\L_\tau\sPre(C))$.
\end{example}

\subsection{The Nisnevich topology}\label{sub:nis}

In this section $S$ denotes a quasi-compact and quasi-separated scheme. We denote by
$\Sm_S$ the category of finitely presented smooth schemes over $S$. Recall
that while all smooth schemes $U$ over $S$ are locally of finite presentation
by definition, saying that $U\rightarrow S$ is finitely presented means in
addition to local finite presentation that the morphism is quasi-compact and
quasi-separated. Note that $\Sm_S$ is an
essentially small category because smooth implies locally of finite
presentation and because $S$ is quasi-compact and quasi-separated.

\begin{definition}[Lurie~\cite{lurie-sag}*{Section~A.2.4}]
    The Nisnevich topology on $\Sm_S$ is the topology generated by those finite
    families of \'etale morphisms $\{p_i:U_i\rightarrow X\}_{i\in I}$ such that
    there is a finite sequence $\emptyset\subseteq Z_n\subseteq
    Z_{n-1}\subseteq \cdots\subseteq Z_1\subseteq Z_0=X$ of finitely presented
    closed subschemes of $X$ such that $$\coprod_{i\in
    I}p_i^{-1}\left(Z_m-Z_{m+1}\right)\rightarrow Z_m-Z_{m+1}$$ admits a section for
    $0\leq m\leq n-1$.
\end{definition}

\begin{remark}
    The referee pointed out that Hoyois has proved in a preprint~\cite{hoyois-allagree} that this definition is
    equivalent (for $S$ quasi-compact and quasi-separated) to the original
    definition of Nisnevich~\cite{nisnevich}, which says that an \'etale cover
    $U\rightarrow X$ is Nisnevich if it is surjective on $k$-points for all
    fields $k$.
\end{remark}

\begin{exercise}
    Show that when $S$ is noetherian of finite Krull dimension, then a finite
    family of \'etale morphisms $\{p_i:U_i\rightarrow X\}_{i\in I}$ is a
    Nisnevich cover if and only if for each point $x\in X$ there is
    an index $i\in I$ and a point $y\in U_i$ over $x$ such that the induced map $k(x)\rightarrow
    k(y)$ is an isomorphism. This is the usual definition of a Nisnevich cover,
    as used for example by~\cite{morel-voevodsky}.
\end{exercise}

\begin{example}\label{ex:niscover}
    Let $k$ be a field of characteristic different than $2$ and $a\in k$ a
    non-zero element. We cover $\AA^1$
    by the Zariski open immersion $\AA^1-\{a\}\rightarrow\AA^1$ and the \'etale
    map $\AA^1-\{0\}\rightarrow\AA^1$ given by $x\mapsto x^2$. This \'etale
    cover is Nisnevich if and only if $a$ is a square in $k$.
\end{example}

\begin{exercise}
    Zariski covers are in particular Nisnevich covers. For example, we will use
    later the standard cover of $\PP^1$ by two copies of $\AA^1$.
\end{exercise}

Of particular importance in the Nisnevich topology are the so-called
elementary distinguished squares.

\begin{definition}
    A pullback diagram
    \begin{equation*}
        \xymatrix{
            U\times_X V\ar[r]\ar[d] &   V\ar[d]^p\\
            U\ar[r]^i & X
        }
    \end{equation*}
    of $S$-schemes in $\Sm_S$
    is an \df{elementary distinguished (Nisnevich) square} if $i$ is a Zariski open
    immersion, $p$ is \'etale, and $p^{-1}(X-U)\rightarrow(X-U)$ is an
    isomorphism of schemes where $X-U$ is equipped with the reduced induced
    scheme structure.
\end{definition}

The proof of the following lemma is left as an easy exercise for the reader.

\begin{lemma}
    In the notation above, $\{i:U\rightarrow X,p:V\rightarrow X\}$ is a Nisnevich cover of $X$.
\end{lemma}

\begin{example}
    If $a$ is a square in Example~\ref{ex:niscover}, then we obtain a Nisnevich
    cover which does not come from an elementary distinguished square. However,
    if we remove one of the square roots of $a$ from $\AA^1-\{0\}$, then we do
    obtain an elementary distinguished square.
\end{example}

\begin{exercise}
    Let $p$ be a prime, let $X=\Spec\ZZ_{(p)}$, and let
    $V=\Spec\ZZ_{(p)}[i]\rightarrow X$,
    where $i^2+1=0$. Let $U=\Spec\QQ\rightarrow X$. Then, $\{U,V\}$ is an
    \'etale cover of $\Spec\ZZ_{(p)}$ for all odd $p$. It is
    Nisnevich if and only if in addition $p\equiv 1\mod 4$.
\end{exercise}

\begin{example}
    Let $X=\Spec R$, where $R$ is a discrete valuation ring with field of
    fractions $K$. Suppose that $p:V\rightarrow X$ is an \'etale map where $V$ is
    the spectrum of another discrete valuation ring $S$. Then, the square
    \begin{equation*}
        \xymatrix{
            U\times_X V\ar[r]\ar[d] &   V\ar[d]^p\\
            U\ar[r]^i & X
        }
    \end{equation*}
    with $i:U=\Spec K\rightarrow X$ is an elementary distinguished square if and
    only if the inertial degree of $R\rightarrow S$ is $1$.
\end{example}

\begin{definition}
    The \df{Nisnevich-local model category} $\L_{\Nis}\sPre(\Sm_S)$ will be denoted
    simply by $\mathrm{Spc}_S$, and the fibrant objects of $\Spc_S$ will be
    called \df{spaces}. So, a space is a presheaf of Kan complexes on
    $\Sm_S$ satisfying Nisnevich hyperdescent in the sense that
    the arrows~\eqref{eq:descent} are weak equivalences for Nisnevich hypercovers.
\end{definition}

\begin{warning}
    There are three candidates for the $\AA^1$-homotopy theory over $S$. One is
    the $\AA^1$-localization of the Joyal-Jardine Nisnevich-local model
    structure~\cite{jardine}. The other is that used
    by~\cite{asok-hoyois-wendt}, which imposes descent only for covers.
    Finally, we impose descent for all hypercovers. When $S$ is noetherian of
    finite Krull dimension, all three definitions are Quillen equivalent. In
    all cases, our definition is equivalent to the Joyal-Jardine definition,
    by the main result of~\cite{dugger-hollander-isaksen}.
\end{warning}

\begin{notation}
    If $X$ and $Y$ are presheaves of simplicial sets on $\Sm_S$, we will write
    $[X,Y]_{\Nis}$ for the set of Nisnevich homotopy classes of maps from $X$
    to $Y$, which is the hom-set from $X$ to $Y$ in the homotopy category of
    $\L_{\Nis}\sPre(\Sm_S)$. The pointed version is written
    $[X,Y]_{\Nis,\star}$. When necessary, we will write $[X,Y]_s$ for the
    homotopy classes of maps from $X$ to $Y$ in $\sPre(\Sm_S)$, and similarly
    we write $[X,Y]_{s,\star}$ for the homotopy classes of pointed maps.
\end{notation}

\begin{notation}
    We will write $\L_\Nis$ for the left derived functor of the identity
    functor $\sPre(\Sm_S)\rightarrow\L_{\Nis}\sPre(\Sm_S)$. Thus, $\L_\Nis$ is
    computed by taking a cofibrant replacement functor with respect to the
    Nisnevich-local model category structure on $\sPre(\Sm_S)$.
\end{notation}

\begin{example}
    Given a scheme $X$ essentially of finite presentation over $S$, we abuse notation
    and also view $X$ as the presheaf it represents on $\Sm_S$. So, if $Y$ is a
    finitely presented smooth $S$-scheme, then $X(Y)=\Hom_S(Y,X)$. Since the Nisnevich topology is
    subcanonical, it is an easy exercise to see that $X$ is fibrant. Indeed,
    homotopy limits of discrete spaces are just computed as limits of the
    underlying sets of components. We will discuss homotopy limits and colimits
    further in Section~\ref{sec:basic}.
\end{example}

The next proposition is a key tool for practically verifying Nisnevich fibrancy for
a given presheaf of simplicial sets on $\Sm_S$.

\begin{proposition}\label{prop:niscrit}
    Suppose that $S$ is a noetherian scheme of finite Krull dimension.
    A simplicial presheaf $F$ on $\Sm_S$ is Nisnevich-fibrant if and only if
    for every elementary distinguished square
    \begin{equation*}
        \xymatrix{
            U\times_X V\ar[r]\ar[d] &   V\ar[d]^p\\
            U\ar[r]^i & X
        }
    \end{equation*}
    the natural map
    \begin{equation*}
        F(X)\rightarrow F(V)\times_{F(U\times_X V)}F(U)
    \end{equation*}
    is a weak equivalence of simplicial sets and $F(\emptyset)$ is a final object. 
\end{proposition}

Presheaves possessing the
property in the proposition are said to satisfy the \df{Brown-Gersten
property}~\cite{brown-gersten} or the \df{excision} property, although Brown
and Gersten studied the Zariski analog.

\begin{proof}
    Let us first indicate the references for this theorem. The proof in
    ~\cite{morel-voevodsky}*{Section 3.1} applies to sheaves of sets (i.e.
    sheaves valued in discrete simplicial sets); to deduce the
    simplicial version, one uses the techniques in~\cite{brown-gersten}.

    One way the reader can get straight to the case of simplicial presheaves is
    via the following argument. The Nisnevich topology is generated by a
    \df{cd-structure}, a collection of squares in $\Sm_S$ stable
    under isomorphism;  the cd-structure corresponding to the Nisnevich topology
    is given by the elementary distinguished squares. This
    observation amounts to \cite{voevodsky-jpaa2}*{Proposition~2.17,
    Remark~2.18}.
    In \cite{voevodsky-jpaa1}*{Section 2}, Voevodsky gives conditions on a
    category $C$ equipped with a cd-structure for when the sheaf condition on a
    presheaf of sets on $C$ coincides with the excision condition with respect
    to the cd-structure.
    For a proof of the corresponding claim for simplicial presheaves (and
    hyperdescent), one can refer
    to \cite{asok-hoyois-wendt}*{Theorem 3.2.5}.

    For another, direct approach see~\cite{dugger-bg}.
\end{proof}
   
    %We indicate the idea of a proof. Suppose that we have an elementary distinguished square as above which we name $Q$. Consider the homotopy pushout $W_Q:= U \coprod^L_{U \times_X V} V$ calculated in $\sPre(\Sm_S)$. There is a natural map $p_Q: W_Q \rightarrow X$. On the other hand we can consider the Cech nerve of the Nisnevich cover $\{U \rightarrow X , V \rightarrow X\}$ which we call $k_Q: \widehat{C}(Q) \rightarrow X$. Let $e: \emptyset \rightarrow \star$ be the map from the empty presheaf to the terminal presheaf. Consider the classes of maps: $A:= \{k_Q, e\}$ and $B := \{p_Q e\}$. We note that $F$ is fibrant if and only if it is $A$-local, while it satisfies the condition above, which we call Nisnevich excisive, if and only if it is $B$-local. Thus the goal is to show that $F$ is $A$-local if and only if $F$ is $B$-local. 
    
    %For any given $Q$ we can take the homotopy pullback:
    
    %$$\xymatrix{
    %W_Q \times_X \widehat{C}(Q) \ar[r]_{\pi_1} \ar[d]_{\pi_2} & W_Q \ar[d] \\
    %\widehat{C}(Q) \ar[r] & X}$$

    %Now if we show that $\pi_2$ and $\pi_1$ are weak equivalences, then we can use the 2 out of 3 property to prove what we want. 

Let $\Sm_S^{\mathrm{Aff}}$ denote the full subcategory category of $\Sm_S$ consisting
of (absolutely) affine schemes. A presheaf $X$ on $\Sm_S^{\mathrm{Aff}}$
satisfies \df{affine Nisnevich excision} if it satisfies excision for the
cd-structure on $\Sm_S^{\mathrm{Aff}}$ consisting of cartesian squares
$$\xymatrix{
    \Spec R'_{f}\ar[r]\ar[d]    &   \Spec R'\ar[d]\\
    \Spec R_f\ar[r] &   \Spec R,
}$$
where $\Spec R'\rightarrow\Spec R$ is \'etale, $f\in R$, and $R/(f)\iso
R'/(f)$. An important result of~\cite{asok-hoyois-wendt} says that the topology
generated by the affine Nisnevich cd-structure is the same as the Nisnevich
topology restricted to $\Sm_S^{\mathrm{Aff}}$.

\subsection{The $\AA^1$-homotopy category}

To define the $\AA^1$-homotopy category, we perform a further left Bousfield
localization of $\L_{\Nis}\sPre(\Sm_S)$. As above, $S$ denotes a quasi-compact
and quasi-separated scheme.

\begin{definition}
    Let $I$ be the class of maps $\AA^1\times_S X\rightarrow X$ in
    $\L_{\Nis}\sPre(\Sm_S)$ as $X$ ranges over all objects of $\Sm_S$. Since
    $\Sm_S$ is essentially small, we can choose a subset $J\subseteq I$
    containing maps $\AA^1\times_S X\rightarrow X$ as $X$ ranges over a
    representative of each isomorphism class of $\Sm_S$.

    The \df{$\AA^1$-homotopy theory} of $S$ is the left Bousfield
    localization $\L_{\AA^1}\L_{\Nis}\sPre(\Sm_S)$ of $\L_{\Nis}\sPre(\Sm_S)$
    with respect to $J$. Its homotopy category will be called the
    \df{$\AA^1$-homotopy category} of $S$. Let
    $\Spc_S^{\AA^1}$ be $\L_{\AA^1}\L_{\Nis}\sPre(\Sm_S)$. Fibrant objects
    of $\Spc_S^{\AA^1}$ will be called \df{$\AA^1$-spaces} or \df{$\AA^1$-local
    spaces}. Note that the simplicial presheaf
    underlying any $\AA^1$-space is in particular a space in the sense that it is fibrant in
    $\Spc_S$. The homotopy category of $\Spc_S^{\AA^1}$ will always be written as
    $\Ho(\Spc_S^{\AA^1})$, and usually functoriality or naturality statements will be made
    with respect to the homotopy category.
\end{definition}

\begin{proposition}
    The Bousfield localization $\Spc_S^{\AA^1}=\L_{\AA^1}\L_{\Nis}\sPre(\Sm_S)$ exists.
\end{proposition}

\begin{proof}
    The simplicial structure, left properness, and combinatoriality are inherited
    by $\Spc_S$ from $\sPre(\Sm_S)$, and hence by Theorem~\ref{thm:existence}
    the Bousfield localization exists.
\end{proof}

\begin{notation}
    If $X$ and $Y$ are presheaves of simplicial sets on $\Sm_S$, we will write
    $[X,Y]_{\AA^1}$ for the set of $\AA^1$-homotopy classes of maps from $X$
    to $Y$, which is the hom-set from $X$ to $Y$ in the homotopy category of
    $\L_{\AA^1}\L_{\Nis}\sPre(\Sm_S)$. The pointed version is written
    $[X,Y]_{\AA^1,\star}$.
\end{notation}

\begin{notation}
    We will write $\L_{\AA^1}$ for the left derived functor of the identity
    functor $\L_{\Nis}\sPre(\Sm_S)\rightarrow\L_{\AA^1}\L_{\Nis}\sPre(\Sm_S)$.
    Thus, $\L_{\AA^1}\L_\Nis$ is
    computed by taking a cofibrant replacement functor with respect to the
    $\AA^1$-local model category structure on $\sPre(\Sm_S)$.
\end{notation}

\begin{remark}
    It is common to call an $\AA^1$-space, an $\AA^1$-local space, and indeed
    the fibrant objects of $\Spc_S^{\AA^1}$ are $\AA^1$-local. In fact, a simplicial presheaf $X$ in $\sPre(\Sm_S)$
    is $\AA^1$-local, i.e., a fibrant object of $\Spc_S^{\AA^1}$, if it
    \begin{enumerate}
        \item   takes values in Kan complexes (so that it is fibrant in
            $\sPre(\Sm_S)$),
        \item   satisfies Nisnevich hyperdescent (so that it is fibrant
            in $\Spc_S$), and
        \item   if $X(U)\rightarrow X(\AA^1\times_S U)$ is a weak equivalence
            of simplicial sets for all $U$ in $\Sm_S$.
    \end{enumerate}
\end{remark}

\begin{exercise} \label{ex:pointed}
    Construct the pointed version $\Spc_{S,\star}^{\AA^1}$ of $\Spc_S^{\AA^1}$, the homotopy theory of
    pointed $\AA^1$-spaces. We will have occasion to use this pointed version as
    well as the Quillen adjunction
    $$\Spc_S^{\AA^1}\rightleftarrows\Spc_{S,\star}^{\AA^1},$$ which sends a
    presheaf of spaces $X$ to the pointed presheaf of spaces $X_+$ obtained by adding a disjoint
    basepoint.
\end{exercise}

\begin{definition}
    The weak equivalences in $\Spc_S^{\AA^1}$ are called \df{$\AA^1$-weak
    equivalences} or \df{$\AA^1$-local weak equivalences}.
\end{definition}

Here is an expected class of $\AA^1$-weak equivalences.

\begin{definition}
    Let $f,g : X \rightarrow Y$ be maps of simplicial presheaves. We say that
    $f$ and $g$ are $\AA^1$-homotopic if there exists a map $H: F \times \AA^1
    \rightarrow G$ such that $H \circ (\id_F \times i_0) = f$ and $H \circ (\id_F
    \times i_1) = g$. A map $g: F \rightarrow G$ is an $\AA^1$-homotopy
    equivalence if there exists morphisms $h: G \rightarrow F$ and that $h
    \circ g$ and $g \circ h$ are $\AA^1$-homotopic to $\id_F$ and $\id_G$
    respectively.
\end{definition}

\begin{exercise} \label{exer:vb}Show that if $p: E \rightarrow X$ is a vector bundle in $\Sm_S$, then $p$ is an $\AA^1$-homotopy equivalence.
\end{exercise}

\begin{exercise}\label{exer:a1equiv}
    Show that any $\AA^1$-homotopy equivalence $f: F \rightarrow G$ is an
    $\AA^1$-weak equivalence. Note that there are many more $\AA^1$-weak
    equivalences.
\end{exercise}

\section{Basic properties of $\AA^1$-algebraic topology}\label{sec:basic}

This long section is dedicated to outlining the basic facts that form the
substrate of the unstable motivic homotopy theorists'
work. Examples and basic
theorems abound, and we hope that it provides a helpful user's manual.
Most non-model category theoretic results below are due to Morel and
Voevodsky~\cite{morel-voevodsky}.

Throughout this section we fix a quasi-compact and quasi-separated base scheme $S$,
and we study the model category $$\Spc_S^{\AA^1}=\L_{\AA^1}\L_{\Nis}\sPre(\Sm_S).$$

\subsection{Computing homotopy limits and colimits through examples}

An excellent source for the construction of homotopy limits or colimits is the
exposition of Dwyer and Spalinski~\cite{dwyer-spalinski}. We start with an
example from ordinary homotopy theory. Consider the following morphism of
pullback diagrams of topological spaces:
\begin{equation*}
    \left(\star\rightarrow S^1\leftarrow \star\right)\rightarrow\left(\star\rightarrow
    S^1\leftarrow P_\star S^1\right),
\end{equation*}
where $P_\star S^1$ is the path space of $S^1$ consisting of paths beginning at the
basepoint of $S^1$. This diagram is a homotopy equivalence in each spot.
However, the pullback of the first is just a point, while the pullback of the
second is the loop space $\Omega S^1$, which is homotopy equivalent to the
discrete space $\ZZ$. This example illustrates that some care is needed when
forming the homotopically correct notion of pullback.

Similarly, consider the maps of pushout diagrams
\begin{equation*}
    \left(*\leftarrow S^0\rightarrow D^1\right)\rightarrow\left(*\leftarrow
    S^0\rightarrow
    *\right),
\end{equation*}
where $D^1$ is the $1$-disk. Again, this map is a homotopy equivalence in each
place. But, the pushout in the first case is $S^1$, and in the second case it
is just a point. Again, care is required in order to compute the \emph{correct}
pushout.

The key in these examples is that $P_\star S^1\rightarrow S^1$ is a fibration, while
$S^0\rightarrow D^1$ is a cofibration. By uniformly replacing pullback diagrams
with pullback diagrams where the maps are fibrations, and then taking the
pullback, one obtains a \emph{homotopy-invariant} notion of pullback, the
homotopy pullback. Similarly, by replacing pushout diagrams with homotopy
equivalent diagrams in which the morphisms are cofibrations, one obtains
homotopy pushouts.

\begin{definition}
    A \df{homotopy pullback} diagram in a model category $M$ is a pullback diagram
    \begin{equation*}
        \xymatrix{
            c\ar[d]\ar[r]   &   d\ar[d]\\
            e\ar[r]         &   f
        }
    \end{equation*}
    in $M$ where at least one of $e\rightarrow f$ or $d\rightarrow f$ is a
    fibration in $M$. Given a pullback diagram $e\rightarrow f\leftarrow d$,
    the \df{homotopy pullback} is the pullback of either the diagram $e'\rightarrow
    f\leftarrow d$ or $e\rightarrow f\leftarrow d'$ where $e'\rightarrow f$
    (resp $d'\rightarrow f$)  is the fibrant replacement via {\bf M4} of
    $e\rightarrow f$ (resp. $d\rightarrow f$).
\end{definition}

\begin{exercise}
    Show that homotopy pullbacks are independent up to weak equivalence of any choices made.
\end{exercise}

To put this notion on a more precise footing, we make the following
construction.

\begin{proposition}[\cite{htt}*{Proposition A.2.8.2}]
    Let $M$ be a combinatorial model category and $I$ a small
    category. The pointwise weak equivalences and pointwise fibrations
    determine a model category structure on $M^I$ called the \df{projective model
    category structure}, which we will denote by $M^I_{\proj}$. The pointwise weak equivalences and pointwise
    cofibrations determine a model category structure on $M^I$ called the
    \df{injective model category structure}, which we will denote by
    $M^I_{\inj}$.
\end{proposition}

We have already seen the projective model category structure in our discussion
of presheaves of spaces on a small category. These two model categories on
$M^I$ can be used to compute homotopy limits and homotopy colimits.

\begin{lemma}
    The functor $\Delta:M\rightarrow M^I$ taking $m\in M$ to the constant
    functor $I\rightarrow M$ on $m$ admits both a left and a right adjoint.
\end{lemma}

\begin{proof}
    Note that the category $M$ is presentable by the definition of a
    combinatorial model
    category. This means that $M$ has all small colimits and is
    $\lambda$-compactly generated for some regular cardinal $\lambda$.
    By the adjoint functor
    theorem~\cite{adamek-rosicky}*{Theorem~1.66}\footnote{This gives the
    criterion for the existence of a left adjoint for a functor between
    locally presentable categories; it is somewhat easier to prove that a
    functor between locally presentable categories which preserves small
    colimits is a left adjoint. There is a good discussion of these issues on
    the \texttt{nLab}.}, it suffices to prove that $\Delta$
    is accessible, preserves limits, and preserves small colimits. However,
    accessibility of $\Delta$ simply means that it commutes with
    $\kappa$-filtered colimits for some regular cardinal $\kappa$. Since we
    will show that
    it commutes with all small colimits, accessibility is an
    immediate consequence.

    To prove that $\Delta$ commutes with small limits, let $y\iso\lim_k y_k$ be
    a limit in $M$. Consider an object $x:I\rightarrow M$ of $M^I$. Then,
    \begin{align*}
        \hom_{M^I}(x,\Delta(y)) &\iso\eq\left(\prod_{i\in
        I}\hom_M(x(i),y(i))\rightrightarrows\prod_{f\in\mathrm{Ar}(I)}\hom_M(x(i),y(j))\right)\\
        &\iso\eq\left(\prod_{i\in
        I}\lim_k\hom_M(x(i),y_k(j))
        \rightrightarrows\prod_{f\in\mathrm{Ar}(I)}\lim_k\hom_M(x(i),y_k(j))\right)\\
        &\iso\lim_k\eq\left(\prod_{i\in
        I}\hom_M(x(i),y_k(j))\rightrightarrows\prod_{f\in\mathrm{Ar}(I)}\hom_M(x(i),y_k(j))\right)\\
        &\iso\lim_k\hom_{M^I}(x,\Delta(y_k)),
    \end{align*}
    using the fact that small limits commute with small limits and hence in
    particular equalizers and small products. It follows that
    $\Delta(y)\iso\lim_k\Delta(y_k)$, as desired.
    The proof that $\Delta$ preserves small colimits is left as an exercise.
\end{proof}

\begin{exercise}
    Show that $\Delta$ preserves small colimits.
\end{exercise}

\begin{definition}
    We will call the right adjoint to $\Delta$ the limit functor $\lim_I$, while the left
    adjoint is the colimit functor $\colim_I$.
\end{definition}

\begin{lemma}
    The pairs of adjoint functors $$\Delta:M\rightleftarrows M^I_{\inj}:\lim_I$$ and
    $$\colim_I:M^I_{\proj}\rightleftarrows M:\Delta$$ are Quillen pairs.
\end{lemma}

\begin{proof}
    Note that $\Delta$ preserves pointwise weak equivalences, pointwise
    fibrations, and pointwise cofibrations.
\end{proof}

\begin{definition}
    We will write $\holim_I$ for $\Rbf\lim_I$ and $\hocolim_I$ for
    $\Lbf\colim_I$, and call theese the homotopy limit and homotopy colimit
    functors.
\end{definition}

\begin{exercise}
    Let $I$ be the small category $\bullet\leftarrow\bullet\rightarrow\bullet$,
    which classifies pushouts. To compute the homotopy pushout $x\leftarrow
    y\rightarrow z$ in $M$, we must take an cofibrant replacement $x'\leftarrow
    y'\rightarrow z'$ in $M^I_{\proj}$, and then we can compute the categorical pushout of the new
    diagram. Describe the cofibrant objects of $M^I$.
    Show that the homotopy pushout can be computed as the pushout of
    $x'\leftarrow y'\rightarrow z'$ where $x'$ and $y'$ are cofibrant and
    $y'\rightarrow z'$ is a cofibration. Show however that such diagrams are
    not in general cofibrant in $M^I_{\proj}$.
\end{exercise}

\begin{proposition}\label{prop:adjointholim}
    Right derived functors of right Quillen functors commute with homotopy
    limits and left derived functors of left Quillen functors commute with homotopy colimits.
\end{proposition}

\begin{proof}
    We prove the result for right Quillen functors and homotopy limits.
    Suppose that we have a Quillen adjunction: $$F: M \leftrightarrows N: G.$$
    It is easy to check that this induces a Quillen adjunction $F^I: M^I_{\inj}
    \leftrightarrows N^I_{\inj}: G^I.$ Indeed, it is enough to check that $F^I$
    preserves cofibrations and acyclic cofibrations, but these are defined
    pointwise in the injective model category structure, so the fact that $F$
    is a left Quillen functor implies that $F^I$ is as well.
    Consider the following diagram    $$\xymatrix{
        M^I_{\inj}\ar[r]^{F^I} &  N^I_{\inj} \\
        M   \ar[u]_{\Delta}\ar[r]^F & N \ar[u]_{\Delta}}$$ 
    of left Quillen functors. This diagram commutes on the level of underlying
    categories; picking appropriate fibrant replacements to compute the right
    adjoints, the right derived versions of the functors commute which induces
    a commutative diagram
    $$\xymatrix{
        \Ho(M^I_{\inj})\ar[r]^{\Lbf F^I} & \Ho(N^I_{\inj}) \\
        \Ho(M)   \ar[u]_{\Lbf\Delta}\ar[r]^{\Lbf F} &\Ho(N)  \ar[u]_{\Lbf\Delta}}$$ 
    of left adjoints on the level of homotopy categories.
    This means means that the diagram 
      $$\xymatrix{
          \Ho(M^I_{\inj}) \ar[d]_{\Rbf \holim_I} &  \Ho(N^I_{\inj})\ar[l]_{\Rbf G^I} \ar[d]^{\Rbf\holim_I} \\
          \Ho(M)  & \Ho(N)\ar[l]_{\Rbf G} }$$ 
    of right adjoints commutes.
\end{proof}

We are now in a position to give examples.

\begin{exercise}
    One should be careful when trying to commute homotopy limits or colimits
    using the above proposition --- the functors must be derived. Construct an
    example using a morphism of commutative rings $R \rightarrow S$,
    the functor $\otimes_R S: \mathrm{Ch}_R^{\geq 0} \rightarrow
    \mathrm{Ch}_S^{\geq 0}$, and the mapping cone of an $R$-module $M
    \rightarrow N$ thought of as chain complexes concentrated in a single
    degree to show that preservation of homotopy colimits fail if $\otimes_RS$
    is not derived. Hint: see the example of mapping cones worked out in
    Example~\ref{ex:cone}.
\end{exercise}

\begin{example}\label{ex:cone}
    Let $A$ be an associative ring, and consider $\Ch_A^{\geq 0}$, the category
    of non-negatively graded chain complexes equipped with the projective model category structure.
    Let $M$ be an $A$-module viewed as a chain complex concentrated in degree
    $0$, and let $N_{\bullet}$ be a chain complex. The actual pushout of a map
    $M\rightarrow N_\bullet$ along $M\rightarrow 0$ is just the cokernel of the
    map of complexes. If $N_\bullet=0$, this cokernel is zero. However, by the
    recipe above, we should replace $0$ with a quasi-isomorphic fibrant model
    $P_\bullet$ with a map $M\rightarrow P_\bullet$ that is a cofibration. A
    functorial choice turns out to be the cone on the identity of $M$. This is
    the complex $M\xrightarrow{\id_M}M$ with $M$ placed in degrees $1$ and $0$.
    This time, when we take the cokernel, we get the complex $M[1]$. This
    confirms what everyone wants: that $M\rightarrow 0\rightarrow M[1]$ should
    be a distinguished triangle in the derived category of $A$, which is what is
    needed to to have long exact sequences in homology.
\end{example}

Let us now turn to examples in $\AA^1$-homotopy theory. The following
proposition gives a way of constructing many examples of homotopy pushouts in
the category $\Spc_S$ and is a consequence of the characterization of fibrant
objects in $\Spc_S$.

\begin{proposition}\label{prop:edspushout}
    If $S$ is a noetherian scheme of finite Krull dimension, then
    an elementary distinguished (Nisnevich) square
    \begin{equation*}
        \xymatrix{
            U\times_X V\ar[r]\ar[d] &   V\ar[d]^p\\
            U\ar[r]^i & X,
        }
    \end{equation*}
    in $\Sm_S$
    thought of as a diagram of simplicial presheaves is a homotopy pushout in $\Spc_S$.    
\end{proposition}

\begin{proof}
    Since the Nisnevich topology is subcanonical (it is coarser than the
    \'{e}tale topology which is subcanonical) we may regard these squares as
    diagrams in $\Spc_S$ via the Yoneda embedding (or, rather, its simplicial
    analogue --- we think of schemes as sheaves of discrete simplicial sets).
    Let $X$ be a space, i.e., a fibrant object of $\Spc_S$. Proposition~\ref{prop:niscrit} tells
    us that applying $X$ to an elementary distinguished square gives rise to a
    homotopy pullback square. This verifies the universal property for a
    homotopy pushout.
\end{proof}

One problem with the category of schemes, as mentioned above, is that it lacks
general colimits, even finite colimits. In particular, general quotient spaces
do not exist in $\Sm_S$.

\begin{definition}
    For the purposes of this paper,
    the \df{quotient} $X/Y$ of a map $X\rightarrow Y$ of schemes in $\Sm_S$ is
    always defined to be the homotopy cofiber of the map in $\Spc_S^{\AA^1}$. Recall
    that the homotopy cofiber is the homotopy pushout of $\star\leftarrow
    X\rightarrow Y$. Note that since localization is a left adjoint, this
    definition agrees up to homotopy with the $\AA^1$-localization of the
    homotopy cofiber computed in $\Spc_S$ by
    Proposition~\ref{prop:adjointholim}.
\end{definition}

\begin{example}
    Proposition~\ref{prop:edspushout} implies that in the situation of an elementary distinguished
    square, the natural map $$\frac{V}{U\times_XV}\rightarrow\frac{X}{U}$$ is an
    $\AA^1$-local weak equivalence. To see this, we see that
    Proposition~\ref{prop:edspushout} gives a Nisnevich local weak equivalence of
    the cofibers of the top and bottom horizontal arrows; since $\L_{\AA^1}$ is
    a left adjoint, we see that it preserves cofibers and thus gives rise to
    the desired $\AA^1$-local weak equivalence.
\end{example}

\begin{example} \label{exmp:susp}
    A particularly important example of a quotient or homotopy cofiber is the
    \df{suspension} of a pointed object $X$ in $\Spc_{S,\star}^{\AA^1}$. This is simply the homotopy
    cofiber of $X\rightarrow\star$, or in other
    words, the homotopy pushout of the diagram
    \begin{equation*}
        \xymatrix{
            X\ar[r]\ar[d]   &   \ast\\
            \ast            &
        }
    \end{equation*}
    which we denote by $\Sigma X$. See Section~\ref{sub:spheres} for one use of the
    construction.
\end{example}

\subsection{$\AA^1$-homotopy fiber sequences and long exact sequences in homotopy sheaves}

\begin{definition}
    Let $X\rightarrow Y$ be a map of pointed objects in a model category. The
    \df{homotopy fiber} $F$ is the homotopy pullback of $\star \rightarrow
    Y\leftarrow X$. In general, if $F\rightarrow X\rightarrow Y$ is a sequence
    of spaces and if $F$ is weak equivalent to the homotopy fiber of
    $X\rightarrow Y$, then
    we call $F\rightarrow X\rightarrow Y$ a homotopy fiber sequence.
\end{definition}

Recall that in ordinary algebraic topology, given a homotopy fiber sequence
$$F\rightarrow X\rightarrow Y$$ of pointed spaces, there is a long exact
sequence
$$\cdots\pi_{n+1}Y\rightarrow\pi_nF\rightarrow\pi_nX\rightarrow\pi_nY\rightarrow\pi_{n-1}F\rightarrow\cdots$$
of homotopy groups, where we omit the basepoint for simplicity.  Exactness
should be carefully interpreted for $n=0,1$, when these are only pointed sets
or not-necessarily-abelian groups. For details, consult Bousfield and
Kan~\cite{bousfield-kan}*{Section IX.4.1}.

\begin{definition}
    The \df{Nisnevich homotopy sheaf} $\pi_n^{\Nis}(X)$ of a pointed object $X$ of
    $\Spc_S$ is the Nisnevich sheafification of the presheaf
    $$U\mapsto[S^n\wedge U_+,X]_{\Nis,\star}.$$
\end{definition}

\begin{definition}
    The \df{$\AA^1$-homotopy sheaf} $\pi_n^{\AA^1}(X)$ of a pointed object $X$ of $\Spc_S^{\AA^1}$
    is the Nisnevich sheafification of the presheaf
    $$U\mapsto[S^n\wedge U_+,X]_{\AA^1,\star}.$$
\end{definition}

\begin{exercise}
    Show that if $X$ is weakly equivalent to
    $\L_{\AA^1}\L_{\Nis}X$, where $X$ is a pointed simplicial presheaf,
    then the natural map $\pi_n^{\Nis}(X)\rightarrow\pi_n^{\AA^1}(X)$ is an
    isomorphism of Nisnevich sheaves.
\end{exercise}

The following result is a good illustration of the theory we have developed so
far.

\begin{proposition}
    Let $F\rightarrow X\rightarrow Y$ be a homotopy fiber sequence in
    $\Spc_S^{\AA^1}$. Then, there is a natural long exact sequence
    $$\cdots\rightarrow\pi_{n+1}^{\AA^1}Y\rightarrow\pi_n^{\AA^1}F\rightarrow\pi_n^{\AA^1}X\rightarrow\pi_n^{\AA^1}Y\rightarrow\cdots$$
    of Nisnevich sheaves.
\end{proposition}

\begin{proof}
    The forgetful functor $\Spc_S^{\AA^1}\rightarrow\sPre(\Sm_S)$ is a right
    adjoint, and hence it preserves homotopy fiber sequences. It follows from
    the fact that fibrations are defined as object-wise fibrations that
    there is a natural long exact sequence of homotopy presheaves. Since
    sheafification, and in particular Nisnevich sheafification, is
    exact~\cite{stacks-project}*{Tag 03CN}, the claim follows.
\end{proof}

\begin{remark}
    We caution the reader that although the functor: $\L_{\AA^1}\L_{\Nis}:
    \sPre(\Sm_S) \rightarrow \Spc^{\AA^1}_S$ preserves homotopy colimits, it is
    not clear that resulting homotopy colimit diagram in $\Spc^{\AA^1}_S$
    possess any exactness properties. To be more explicit, let $i:
    \Spc^{\AA^1}_S \rightarrow \sPre(\Sm_S)$ be the forgetful functor. Suppose
    that we have a homotopy cofiber sequence: $X \rightarrow Y \rightarrow Z$
    in $\sPre(\Sm_S)$, then the it is not clear that $i\L_{\AA^1}\L_{\Nis}(Z)$
    is equivalent to the cofiber of $i\L_{\AA^1}\L_{\Nis}(X) \rightarrow
    i\L_{\AA^1}\L_{\Nis}(Y)$ since we are composing a Quillen left adjoint with
    a Quillen right adjoint. Consequently, long exact sequences which arise out
    of cofiber sequences (such as mapping into Eilenberg-MacLane spaces which
    produces the long exact sequences in ordinary cohomology) in $\sPre(\Sm_S)$ will not apply to this situation.
\end{remark}

\subsection{The $\mathrm{Sing}^{\AA^1}$-construction}

While the process of Nisnevich localization, which produces objects of $\Spc_S$, is
familiar from ordinary sheaf theory, the localization
$\L_{\AA^1}:\Spc_S\rightarrow\Spc_S^{\AA^1}$ is more difficult to grasp concretely. This
section describes one model for the localization functor $\L_{\AA^1}$.

Consider the cosimplicial scheme
$\Delta^{\bullet}$ where $$\Delta^n = \Spec\,k[x_0, ..., x_n] / (x_0 + ...
+ x_1 =1)$$ with the face and degeneracy maps familiar from the standard
topological simplex.
The scheme $\Delta^n$ is a closed subscheme of $\AA^{n+1}$ isomorphic to
$\AA^n$, the $i$th coface map $\partial_j: \Delta^n \rightarrow \Delta^{n+1}$
is defined by setting $x_j =0$, and the $i$th codegeneracy $\sigma_i: \Delta^n
\rightarrow \Delta^{n-1}$ is given by summing the $i$th and $i+1$st coordinates.

\begin{definition}
    Let $X$ be a simplicial presheaf. We define the simplicial presheaf
    $\Sing^{\AA^1}X:=|X(-\times\Delta^\bullet)|$.
    This gives the \df{singular construction} functor
    $$\Sing^{\AA^1}:\sPre(\Sm_S)\rightarrow\sPre(\Sm_S).$$ We will also write
    $\Sing^{\AA^1}$ for the restriction of the singular construction to $\Spc_S\subseteq\sPre(\Sm_S)$.
\end{definition}

\begin{remark}
    Since geometric realizations do not commute in general with homotopy limits,
    there is no reason to expect $\Sing^{\AA^1}$ to preserve the property of
    being Nisnevich-local. At heart this is the reason for both the subtlety
    and the depth of motivic homotopy theory.
\end{remark}

From the above remark it is thus useful to introduce a new terminology: we say
that a simplicial presheaf $X$ is \df{$\AA^1$-invariant} if
$X(U)\rightarrow X(U\times\AA^1)$ is a weak homotopy equivalence of
simplicial sets for every $U$ in $\Sm_S$.

\begin{theorem}\label{thm:Sing}
    Let $S$ be a base separated noetherian scheme and $X$ a simplicial
    presheaf. Then,
\begin{enumerate}
    \item $\Sing^{\AA^1}X$ is $\AA^1$-invariant, and
    \item the natural map $g: X \rightarrow \Sing^{\AA^1}X$ induces a
        weak equivalence $\map(\Sing^{\AA^1}X,Y) \rightarrow \map(X, Y)$ for any
        $\AA^1$-invariant simplicial presheaf $Y$.
\end{enumerate}
\end{theorem}

\begin{proof} 
    For $i= 0, ..., n$ we have maps $\theta_i: \AA^{n+1} \simeq \Delta^{n+1}
    \rightarrow \AA^n \simeq \Delta^n \times_S \AA^1$ corresponding to a
    ``simplicial decomposition" of $\Delta^n \times_S \AA^1$ made up of
    $\Delta^{n+1}$'s (see, for example, \cite{mazza-weibel-voevodsky}*{Figure
    2.1}). For an arbitrary $S$-scheme $U$, the $\theta_i$ maps induce a morphism of cosimplicial schemes
    $$
    \xymatrix{
    \cdots &  \Delta^2 \times_S U
    \ar@<1.5ex>[l]
    \ar@<0.5ex>[l]
    \ar@<-0.5ex>[l]
    \ar@<-1.5ex>[l]
    \ar[dr]_{\theta_i}
    &
    \Delta^1 \times_S U
    \ar@<0ex>[l]
    \ar@<-1ex>[l]
    \ar@<1ex>[l] \ar[dr]_{\theta_i}
    &
    U\ar@<0.5ex>[l]\ar@<-0.5ex>[l]\\
    \cdots &  \Delta^2 \times_S \AA^1 \times_S U
    \ar@<1.5ex>[l]
    \ar@<0.5ex>[l]
    \ar@<-0.5ex>[l]
    \ar@<-1.5ex>[l]
    &
    \Delta^1 \times_S \AA^1 \times_S U\ar@<0ex>[l]
    \ar@<-1ex>[l]
    \ar@<1ex>[l]
    &
    \AA^1 \times_S U\ar@<0.5ex>[l]\ar@<-0.5ex>[l]
    }
    $$
    such that, upon applying a
    simplicial presheaf $X$, we get a simplicial homotopy
    \cite{weibel}*{Section 8.3.11} between the maps $\partial_0^*, \partial_1^*: \Sing^{\AA^1}X(U
    \times \AA^1) \rightarrow \Sing^{\AA^1}X(U)$ induced by the $0$ and
    $1$-section respectively. Hence, as shown in the exercise below,
    $\Sing^{\AA^1}X$ is $\AA^1$-invariant.

    Observe that the functor $U \mapsto X(U
    \times_S \Delta^n)$ is the same as the functor $U \mapsto \map(U\times_S\Delta^n,
    X)$. We have a natural map $X \simeq \map(\Delta^0, X) \rightarrow
    \map(\Delta^n,X)$ for each $n$, so we think of the map $X
    \rightarrow \Sing^{\AA^1}X$ as the canonical map from the zero simplices.

    To check the second claim, it is enough to prove that for all $n \geq 0$,
    we have a weak equivalence $$\mathrm{map}(X, Y) \rightarrow
    \mathrm{map}(\map(\Delta^n, X), Y)$$ whenever $Y$ is $\AA^1$-invariant. Furthermore, $\map(\Delta^n,
    X) \simeq \map(\Delta^1, \map(\Delta^{n-1}, X))$ so by induction we just need to
    prove the claim for $n=1$. To do so, we claim that the map $f: X
    \rightarrow \map(\Delta^1, X)$ induced by the projection $\AA^1 \rightarrow
    S$ is an $\AA^1$-homotopy equivalence, from which we conclude the desired claim
    from Exercise~\ref{exer:a1equiv}.

    There is a map $g: \map(\Delta^1, X) \rightarrow X$ induced by the zero
    section, from which we automatically have $f \circ g = \id$. We then have to
    construct an $\AA^1$-homotopy between $g \circ f$ and $\id_{\map(\Delta^1,
    X)}$ so we look for a map $ H: \map(\Delta^1, X) \times \AA^1
    \rightarrow  \map(\Delta^1, X)$. By adjunction this is the same data as a
    map $ H: \map(\Delta^1, X) \rightarrow \map(\Delta^1 \times \Delta^1,
    X)$. To construct this map we use the multiplication map $\AA^1 \times
    \AA^1 \rightarrow \AA^1$, $(x, y) \mapsto xy$ from which it is easy to see that
    $H^*(\id \times \partial_0)^* =\id$ and $H^*(\id \times \partial_1) = g \circ f$.
\end{proof}

\begin{exercise}
    If $X$ is a simplicial presheaf, then $X$ is $\AA^1$-invariant if
    and only if for any $U \in \Sm_S$ the morphisms $\partial_0^*, \partial_1^*: X(U \times_S \AA^1)
    \rightarrow X(U)$ induced by the $0$ and $1$-sections are homotopic. Hint:
    use again the multiplication map $\AA^1 \times \AA^1 \rightarrow \AA^1, (x,
    y) \mapsto xy$ as a homotopy.
    See~\cite{mazza-weibel-voevodsky}*{Lemma~2.16}.
\end{exercise}

We conclude from the above results that $\Sing^{\AA^1}X$ is
$\AA^1$-invariant and, furthermore, $X$ and $\Sing^{\AA^1}X$ are $\AA^1$-weak equivalent which means,
more explicitly, that they become weakly equivalent in $\Spc^{\AA^1}_S$ after
applying $\L_{\AA^1}\L_{\Nis}$.

\begin{theorem}\label{thm:explicit}
    The functor $\L_{\AA^1}\L_{\Nis}: \sPre(\Sm_S)
    \rightarrow\Spc_S^{\AA^1}$ is equivalent to the countable iteration
    $(\L_{\Nis}\Sing^{\AA^1})^{\circ\NN}$.
\end{theorem}

\begin{proof}
    Let $\Phi=\L_{\Nis}\Sing^{\AA^1}$, so that the theorem claims that
    $\Phi^{\circ\NN}\we\L_{\Nis}$.
    We first argue that $(\L_{\Nis}\Sing^{\AA^1})^{\circ\NN}X$ is fibrant
    in $\Spc_S^{\AA^1}$ for any $X$ in $\sPre(\Sm_S)$. We must simply check
    that it is Nisnevich and $\AA^1$-local. To check that it is Nisnevich
    local, write
    $$\Phi^{\circ\NN}(X)\we\hocolim_{n\rightarrow\infty}(\L_{\Nis}\Sing^{\AA^1})^{\circ
    n}(X),$$ a filtered homotopy colimit of Nisnevich local presheaves of spaces.
    It hence suffices to show that the forgetful functor
    $\Spc_S\rightarrow\sPre(\Sm_S)$ preserves filtered homotopy colimits.
    However, since the sheaf condition is checked on the finite homotopy limits induced
    from the elementary distinguished squares by Proposition~\ref{prop:niscrit}, and since filtered homotopy
    limits commute with finite homotopy limits, the result is immediate.
    At this point we must be honest and point out that the main reference we
    know for the commutativity of finite homotopy limits and filtered homotopy
    colimits, namely~\cite{htt}*{Proposition~5.3.3.3},
    is for $\infty$-categories rather than model categories. 
    However, since homotopy limits and colimits in combinatorial simplicial
    model categories (such as all model categories in this paper) agree with
    the corresponding $\infty$-categorical limits and colimits
    by~\cite{htt}*{Section~4.2.4}, this should be no cause for concern.

    To check that $(\L_{\Nis}\Sing^{\AA^1})^{\circ\NN}X$ is $\AA^1$-local, note
    that we can write
    $$\Phi^{\circ\NN}X\we\hocolim_{n\rightarrow\infty}(\Sing^{\AA^1}\L_{\Nis})^{\circ
    n}\left(\Sing^{\AA^1}X\right),$$
    a filtered homotopy colimit of $\AA^1$-invariant presheaves by
    Theorem~\ref{thm:Sing}. But, filtered homotopy
    colimits of $\AA^1$-invariant presheaves are
    $\AA^1$-invariant. Since $\Phi^{\circ\NN}(X)$ is Nisnevich local and
    $\AA^1$-invariant, it is $\AA^1$-local.

    Thus, we have seen that $\Phi^{\circ\NN}$
    does indeed take values in the fibrant objects of $\Spc_S^{\AA^1}$.
    Finally, we claim that it suffices to show that $\Phi\we\L_{\Nis}\Sing^{\AA^1}$
    preserves $\AA^1$-local weak equivalences. Indeed, if this is the case,
    then so does $\Phi^{\circ\NN}$, which will show that
    $$\Phi^{\circ\NN}(X)\we\Phi^{\circ\NN}(\L_{\AA^1}\L_{\Nis}X)\we\L_{\AA^1}\L_{\Nis}X,$$
    since it is clear that $X\we\Phi(X)$ when $X$ is $\AA^1$-local.
    For the remainder of the proof, write
    $\map(-,-)$ for the mapping spaces in $\sPre(\Sm_S)$. We want to show that
    $$\map(\Phi(X),Y)\we\map(X,Y)$$ for all $\AA^1$-local
    objects $Y$ of $\sPre(\Sm_S)$. But,
    \begin{align*}
        \map(\Phi(X),Y)&\we\map_{\Spc_S}(\L_{\Nis}\Sing^{\AA^1}X,Y)\\
        &\we\map(\Sing^{\AA^1}X,Y)
    \end{align*}
    since $Y$ is in particular Nisnevich local. As the singular
    construction functor $\Sing^{\AA^1}$ is a homotopy colimit, it commutes with
    homotopy colimits. Since $X\we\hocolim_{U\rightarrow X}U$, where the
    colimit is over maps from smooth $S$-schemes $U$, it follows that it is
    enough to show that $$\map(\Sing^{\AA^1}U,X)\we\map(U,Y)$$
    for $U$ a smooth $S$-scheme and $Y$ an $\AA^1$-local presheaf. To prove
    this, it is enough in turn to show that
    $$\map(U(-\times\AA^n),Y)\we\map(U,Y),$$ where $U(-\times\AA^n)$ is the
    presheaf of spaces $V\mapsto U(V\times\AA^n)$. Note that because there is
    an $S$-point of $\AA^n$, the representable presheaf $U$ is a retract of
    $U(-\times\AA^n)$, so it suffices to show that
    $$\pi_0\map(U(-\times\AA^n),Y)\iso\pi_0\map(U,Y),$$
    or even just that the map
    $$\pi_0\map(U,Y)\rightarrow\pi_0\map(U(-\times\AA^n),Y)$$ induced by an
    $S$-point of $\AA^n$ is a surjection. Now,
    $U(-\times\AA^n)\we\hocolim_{V\times\AA^n\rightarrow U}V\times\AA^n$, so
    $$\pi_0\map(U(-\times\AA^n),Y)\iso\pi_0\lim_{V\times\AA^n\rightarrow
    U}\map(V\times\AA^n,Y)\iso\pi_0
    \lim_{V\times\AA^n\rightarrow U}\map(V,Y),$$
    the last weak equivalence owing to the fact that $Y$ is $\AA^1$-local. This
    limit can be computed as $\lim_{V\times\AA^n\rightarrow
    U}\pi_0\map(V,Y)$ since $\pi_0$ commutes with all colimits (being left
    adjoint to the inclusion of discrete spaces in all spaces). Picking an $S$-point of $\AA^n$ gives a
    compatible family $$\lim_{V\rightarrow
    U}\pi_0\map(V,Y)\iso\pi_0\map(U,Y),$$ giving a
    section of the natural map
    $\pi_0\map(U,Y)\rightarrow\lim_{V\times\AA^n\rightarrow
    U}\pi_0\map(V,Y)$.
\end{proof}

From this description, we get a number of non-formal consequences.

\begin{corollary}
    The $\AA^1$-localization functor commutes with finite products. 
\end{corollary}

\begin{proof}
    Both $\Sing^{\AA^1}$ (being a sifted colimit) and $\L_{\Nis}$ have this
    property. For $\L_{\Nis}$ the fact is clear because it is the left adjoint
    of a geometric morphism of $\infty$-topoi and hence left exact
    (see~\cite{htt}), while for $\Sing^{\AA^1}$ we
    refer to~\cite{adamek-rosicky-vitale}. Alternatively, it is easy to check
    directly that the singular construction commutes with finite products and
    it is shown in~\cite{morel-voevodsky}*{Theorem~1.66} that Nisnevich
    localization commutes with finite products. (Note that since finite
    products and finite homotopy products agree, it is easy to transfer the
    Morel-Voevodsky proof along the Quillen equivalences necessary to bring it
    over to our model for $\Spc_S^{\AA^1}$.)
\end{proof}

The corollary is important in proving that certain functors which are symmetric
monoidal on the level of presheaves, remain symmetric monoidal after
$\AA^1$-localization.

\begin{definition}
    If $X \in \sPre(\Sm_S)$, then $X$ is $\AA^1$-\df{connected} if the
    canonical map $X\rightarrow S$ induces an isomorphism of sheaves
    $\pi_0^{\AA^1}X \rightarrow \pi_0^{\AA^1}S = \star$. We say
    that $X$ is \df{naively}-$\AA^1$-\df{connected} if the canonical map
    $\Sing^{\AA^1}X \rightarrow S$ induces an isomorphism
    $\pi_0^{\Nis}\Sing^{\AA^1}X \rightarrow \pi_0^{\Nis}S = \star$.
\end{definition}

\begin{corollary}[Unstable $\AA^1$-connectivity theorem]\label{thm:unstabconn}
    Suppose that $X$ is a simplicial presheaf on $\Sm_S$. The canonical
    morphism $X \rightarrow \L_{\AA^1}\L_{\Nis}X$ induces an
    epimorphism $\pi_0^{\Nis}X \rightarrow
    \pi_0^{\Nis}\L_{\AA^1}\L_{\Nis}X = \pi_0^{\AA^1}X$. Hence, if
    $\pi_0^{\Nis}X  = \star$, then $X$ is $\AA^1$-connected.
\end{corollary}  

\begin{proof}
    By Theorem~\ref{thm:jardine}, it follows that $X\rightarrow\L_{\Nis}X$
    induces isomorphisms on homotopy sheaves
    $\pi_0^{\Nis}X\rightarrow\pi_0^{\Nis}\L_{\Nis}X$. Hence, using the fact
    that sheafification preserves epimorphisms and
    Theorem~\ref{thm:explicit}, it suffices to show that
    $\pi_0X(U)\rightarrow\pi_0\Sing^{\AA^1}X(U)$ is surjective for all
    $X\in\sPre(\Sm_S)$ and all $U\in\Sm_S$.
    To do so, we note that $\pi_0 \Sing^{\AA^1}X(U)$ is calculated as $\pi_0$ of the bisimplicial set
    $X_{\bullet}(U \times \Delta^{\bullet})$. This is in turn calculated as
    the coequalizer of the diagram $$\pi_0 X(U \times_S \AA^1)
    \rightrightarrows \pi_0 X(U),$$ where the maps are
    induced by $\Delta^{\bullet}$ and thus we get the desired surjection.
\end{proof}

%\begin{remark} The functor $Sing^I$ is a homotopical left adjoint to inclusion of $I$-invariant presheaves to all presheaves, hence the derived functor of $Sing^I$ should preserve all homotopy colimits.
%\end{remark}

% \subsubsection{Description of $Sing^{\AA^1}$}

Consequently, to determine if a simplicial presheaf is $\AA^1$-connected, it
suffices to calculate its sheaf of ``naive'' $\AA^1$-connected components. We
will use this observation later to prove that $\SL_n$ is $\AA^1$-connected.

\begin{corollary}\label{thm:naiveconn} If $X \in \sPre(\Sm_S)$,
    then the natural morphism $\pi_0^{\Nis}\Sing^{\AA^1}(X) \rightarrow
    \pi_0^{\AA^1}X$ is an epimorphism. Hence if $X$ is naively
    $\AA^1$-connected, then it is $\AA^1$-connected.
\end{corollary}

\begin{proof}
    Since the natural map $X \rightarrow \Sing^{\AA^1}X$ is an
    $\AA^1$-local weak equivalence by Theorem~\ref{thm:Sing}, we
    deduce that $\L_{\AA^1}\L_{\Nis}X\we\L_{\AA^1}\L_{\Nis}\Sing^{\AA^1}X$, so we
    may apply Theorem~\ref{thm:unstabconn} to $\Sing^{\AA^1}X$ to get the
    desired conclusion.
\end{proof}

\subsection{The sheaf of $\AA^1$-connected components}

The $0$-th $\AA^1$-homotopy sheaf, or the \df{sheaf of $\AA^1$-connected
components}, admits a simple interpretation: it is the Nisnevich
sheafification of the presheaf $U \mapsto [U_+, X]_{\AA^1} \simeq [U_+,
\L_{\AA^1}\L_{\Nis}X]_s$. With this description, we may perform some calculations whose
results deviate from our intuition from topology.

\begin{definition} Let $X$ be an $S$-scheme. We say that $X$ is
    \df{$\AA^1$-rigid} if $\L_{\AA^1}X \simeq X$ in $\Spc_S$.
    Concretely, this condition amounts to saying that $X(U \times_S \AA^1_S)
    \simeq X(U)$ for any finitely presented smooth $S$-scheme $U$. 
\end{definition} 

\begin{exercise} Let $k$ be a field. Prove that the following $k$-schemes are all $\AA^1$-rigid:
    \begin{enumerate}
    \item $\GG_m$;
    \item smooth projective $k$-curves of positive genus;
    \item abelian varieties.
    \end{enumerate}
    In fact, if $S$ is a reduced scheme of finite Krull dimension,
    show that $\Gm$ is rigid in $\Spc_S^{\AA^1}$.
\end{exercise}

\begin{proposition}\label{prop:rigid}
    Let $X$ be an $\AA^1$-rigid $S$-scheme. Then $\pi_0^{\AA^1}(X) \simeq X$ as
    Nisnevich sheaves, and $\pi_n^{\AA^1}(X) = 0$ for $n >0$.
\end{proposition}

\begin{proof}
    The homotopy set  $U \mapsto [U, X]_{\AA^1} \simeq [U, X]_s =
    \pi_0(\map_{\Spc_S}(U, X))$ is equivalent to the set of $S$-scheme maps
    from $U$ to $X$ as $U$ and $X$ are discrete simplicial sets. Hence this
    presheaf is equivalent to the presheaf represented by $X$ which is already
    a Nisnevich sheaf on $\Sm_S$.
    Now, $[S^n \wedge U_+, X]_s = [S^n, \map_{\Spc_S}(U_+, X)]_{\sSets}$,
    which is trivial since the target is a discrete simplicial set. Since $X$
    is $\AA^1$-rigid we see that $[S^n\wedge U_+,X]_s\iso[S^n \wedge U_+,
    X]_{\AA^1}$,
    and thus the sheafification is also trivial.
\end{proof}

\begin{exercise}
    Let $\Sm^{\AA^1}_S \hookrightarrow \Sm_S$ be the full subcategory spanned
    by $\AA^1$-rigid schemes. Then the natural functor $\Sm_S^{\AA^1} \rightarrow
    \Spc^{\AA^1}_S$ which is the composite of $\L_{\AA^1}: \Spc_S \rightarrow
    \Spc^{\AA^1}_S$ and the Yoneda embedding is fully faithful.
    In other words, two $\AA^1$-rigid schemes are isomorphic as schemes if and
    only if they are $\AA^1$-equivalent.
\end{exercise}

\subsection{The smash product and the loops-suspension adjunction}

% Consider the pointed analog of the Jardine model structure above on
% $\sPre(\Sm_S)_{\star}$. We will later need the analogue of following standard
% adjunction in the category of based simplicial sets: $$[\Sigma X, Y] \simeq [X,
% \Omega Y].$$  In order to do this, we need to discuss smash product on the
% pointed versions of the categories we considered above. The main point is
% that the smash product, defined on pointwise on $\sPre(\Sm_S)_{\star}$,
% descends to the homotopy category of $\Spc^{\AA^1}_{S,\star}$. 

Let us begin with some recollection about smash products in simplicial sets.
Let $(X, x), (Y,y)$ be two pointed simplicial sets, then we can form the smash product
$(X, x) \wedge (Y, y)$ which is defined to be the pushout:
$$\xymatrix{
(X, x) \vee (Y, y) \ar[r] \ar[d] &  (X, x) \times (Y, y) \ar[d] \\
\star \ar[r] & (X, x) \wedge (Y, y)}$$
The functor $- \wedge (X, x)$ is then a left Quillen endofunctor on the
category of simplicial sets by the following argument: if $(Z, z)
\rightarrow (Y, y)$ is a cofibration of simplicial sets, then we note that $(X,
x) \wedge (Z, z) \rightarrow  (X, x) \wedge (Y, y)$ is cofibration since
cofibrations are stable under pushouts.
% Now since cofibrations are
% monomorphisms, and products are calculated levelwise we see that $(X \times Z, x
% \times z) \rightarrow (X \times Y, x \times y)$ is also a cofibration.
% To conclude, we again use the fact that cofibrations are stable under pushouts to
% conclude that the induced map: $(X, x) \wedge (Y,y) \rightarrow (X, x)
% \rightarrow (Z, z)$ is stable under pushouts.
The case of acyclic cofibrations
is left to the reader. The right adjoint to $- \wedge (X, x)$ is given by the
pointed mapping space $\map_{\star}(X, -)$ which is given by
the formula: $$\map_{\star}(X, Y)_n \iso \Hom_{\sSets_\star}(X \wedge
\Delta^n_+, Y).$$

To promote the smash product to the level of simplicial presheaves, we first
take the pointwise smash product, i.e., if $(X, x), (Y,y)$ are objects in
$\sPre(\Sm_S)_{\star}$, then we form the smash product $(X, x) \wedge (Y,y)$ as
the simplicial presheaf: $$U \mapsto (X, x)(U) \wedge (Y,y)(U).$$ An analogous
pointwise formula is used for the pointed mapping space functor.

\begin{proposition} The Quillen adjunction $(X, x) \wedge - :
    \sPre(\Sm_S)_{\star} \rightarrow  \sPre(\Sm_S)_{\star}:\map_{\star}(X, -)$
    descends to a Quillen adjunction: $(X, x) \wedge - :
    \Spc^{\AA^1}(S)_{\star} \rightarrow \Spc^{\AA^1}(S)_{\star}:
    \map_{\star}(X, -)$ and thus there are natural isomorphisms
    categories:
    $$[(X, x) \wedge^{\Lbf} (Z,z), (Y,y)]_{\AA^1} \simeq [(X, x), \Rbf\map_{\star}((X, x), (Y,y))]_{\AA^1}$$
    for $X,Y,Z\in\Spc_S^{\AA^1}$.
\end{proposition}

\begin{proof} The question of whether a monoidal structure defined on the
    underlying category of a model category descends to a Quillen adjunction
    with mapping spaces as its right adjoint is answered in the paper of
    Schwede and Shipley
    \cite{schwede-shipley}. The necessary conditions are checked in
    \cite{dundas-rondigs-ostvaer}*{Section 2.1}.
\end{proof}

Now, recall that the suspension of $(X,x)$ is calculated either as the homotopy cofiber
of the canonical morphism $(X,x) \rightarrow \star$ or, equivalently, as $S^1 \wedge
(X, x)$ (check this!), while the loop space is calculated as the homotopy
pullback of the diagram
$$\xymatrix{
\Omega(X, x) \ar[r] \ar[d]  & \star \ar[d]\\
\star \ar[r] & (X, x)}$$
or, equivalently, as $\map_{\star}(S^1, (Y,y))$. Consequently: 

\begin{corollary}
    For any objects $(X,x), (Y,y) \in\Spc^{\AA^1}(\Sm_S)_{\star}$,there is an
    isomorphism $$[\Lbf\Sigma (X,x), (Y,y)]_{\mathbb{A}^1} \simeq [(X,x), \Rbf
        \Omega (Y,y)]_{\mathbb{A}^1}.$$
\end{corollary}

\subsection{The bigraded spheres}\label{sub:spheres}

We will now delve into some calculations in $\AA^1$-homotopy theory. More
precisely, these are calculations in the pointed category
$\Spc^{\AA^1}_{S,\star}$. We use the following conventions for base points
of certain schemes which will play a major role in the theory.

\begin{enumerate}
    \item $\AA^1$ is pointed by $1$
    \item $\GG_m$ is pointed by $1$.
    \item $\PP^1$ is pointed by $\infty$. 
    \item $X_+$ denotes $X$ with a disjoint base point for a space $X$. 
\end{enumerate}

In particular, we only write pointed objects as $(X, x)$ when the base points
are not the  ones indicated above. We also note that the forgetful functors
$\Spc_{S,\star} \rightarrow \Spc_S$ and $\Spc^{\AA^1}_{S,\star}
\rightarrow \Spc^{\AA^1}_S$ preserve and detect weak equivalences. Hence when the
context is clear, we will say Nisnevich or $\AA^1$-weak equivalence as opposed to
\emph{pointed} Nisnevich or $\AA^1$-weak equivalence.

\begin{remark}
    In many cases, base points of schemes are negotiable in the sense
    that there is an explicit pointed $\AA^1$-local weak equivalence between $(X,x)$ and
    $(X, y)$ for two base points $x, y$. For example $(\PP^1, \infty)$ is
    $\AA^1$-equivalent in the pointed category to $(\PP^1, x)$ for any other
    point $x \in \PP^1$ via an explicit $\AA^1$-homotopy.
\end{remark}

Of course, if one takes a cofiber of pointed schemes (or even simplicial presheaves), the cofiber is automatically pointed: if $X \rightarrow Y \rightarrow X/Y$ is a cofiber sequence, then $X/Y$ is pointed by the image of $Y$. 

The first calculation one encounters in $\AA^1$-homotopy theory is the
following.

\begin{lemma}\label{lem:gmsuspension}
    In $\Spc^{\AA^1}_{S,\star}$, there are $\AA^1$-weak equivalences $\Sigma (\mathbb{G}_m,1) \simeq
    (\mathbb{P}^1, \infty) \simeq \AA^1/(\AA^1-\{0\})$.
\end{lemma}

\begin{proof} Consider the distinguished Nisnevich square
    $$ \xymatrix{
    \mathbb{G}_m \ar[r] \ar[d] & \AA^1 \ar[d] \\
    \AA^1 \ar[r] & (\PP^1,1)}$$
     in $\Sm_S$.
    By Proposition~\ref{prop:edspushout}, this can be viewed as a homotopy pushout
    in $\Spc_{S,\star}$ as well. Since the localization functor
    $\Spc_{S,\star} \rightarrow \Spc^{\AA^1}_{S,\star}$ is a Quillen left adjoint, it commutes with
    homotopy colimits, and in particular with homotopy pushouts. Therefore,
    when viewed in $\Spc^{\AA^1}_{S,\star}$ the square above is a homotopy pushout.
    However, since $\AA^1\we\star$ in the $\AA^1$-homotopy theory, it follows
    that $\Sigma\Gm\we(\PP^1,\infty)$ (by
    contracting both copies of $\AA^1$ and noting that $(\PP^1, 1) \we (\PP^1, \infty)$) or $\Sigma\Gm\we\AA^1/(\AA^1-\{0\})$
    (by contracting one of the copies of $\AA^1$).
\end{proof}

The above calculation justifies the idea that in $\AA^1$-homotopy theory there are two kinds of circles: the simplicial
circle $S^1$ and the ``Tate" circle $\mathbb{G}_m$.
The usual convention (which matches up with the grading in motivic cohomology)
is to define $$S^{1,1}= \mathbb{G}_m,$$ and $$S^{1,0} = S^1.$$
Consequently, by the lemma, we have an $\AA^1$-weak equivalence $S^{2,1} \simeq \PP^1$. 

Now, given a pair $a,b$ of non-negative integers satisfying $a\geq b$, we can
define $S^{a,b}=\Gm^{\wedge b}\wedge(S^1)^{\wedge(a-b)}$. In general, there is
no known nice description of these motivic spheres. However, the next two
results give two important classes of exceptions.

\begin{proposition} \label{prop:S2n-1,n}
     In $\Spc^{\AA^1}_{S,\star}$, there are $\AA^1$-weak equivalences $S^{2n-1,n}\we\AA^n-\{0\}$ for $n\geq 1$.
\end{proposition}

\begin{proof}  The case $n=1$ is Lemma~\ref{lem:gmsuspension}; we need to do
    the $n=2$ case and then perform induction.
    Specifically, the claim for $n=2$ says that $\AA^2 -\{0\} \simeq S^1 \wedge
    (\GG_m^{\wedge 2})$. First, observe that we have a homotopy push-out diagram:
    $$
    \xymatrix{
    \mathbb{G}_m \times \mathbb{G}_m \ar[r] \ar[d] & \GG_m \times \AA^1 \ar[d] \\
    \GG_m \times \AA^1 \ar[r] & \AA^2 - \{0\},}$$
    from which we conclude that $\AA^2 -\{0\}$ is calculated as the homotopy
    push-out of $$\mathbb{G}_m \leftarrow  \mathbb{G}_m \times \mathbb{G}_m
    \rightarrow \mathbb{G}_m.$$
On the other hand we may calculate this homotopy push-out using the diagram
$$\xymatrix{
\star & \star \ar[l] \ar[r] & \star\\
\GG_m \ar[d]\ar[u] & \GG_m \vee \GG_m \ar[l]\ar[u] \ar[d] \ar[r] & \GG_m
\ar[d]\ar[u]\\
\GG_m  & \GG_m \times \GG_m \ar[l]  \ar[r] & \GG_m.}$$
Taking the homotopy push-out across the horizontal rows gives us $\star
\leftarrow \star \rightarrow \AA^2 -\{0\}$; taking this homotopy push-out gives
back $\AA^2 -\{0\}$.
On the other hand, taking the homotopy push-put across the vertical rows give us
$\star \leftarrow \GG_m \wedge \GG_m \rightarrow \star$ which calculates the homotopy push-out $S^1 \wedge (\GG_m \wedge \GG_m)$.

Let us now carry out the induction. We have a distinguished Nisnevich square
$$\xymatrix{
    \AA^{n-1} -\{0\} \times \mathbb{G}_m \ar[r] \ar[d] & \AA^n \times \GG_m \ar[d] \\
    \AA^{n-1} -\{0\} \times \AA^1 \ar[r] & \AA^n -\{0\},}$$
from which we conclude that $\AA^n -\{0\}$ is calculated as the homotopy push-out of 
$$\AA^{n-1} -\{0\} \leftarrow \AA^{n-1} -\{0\} \times \GG_m \rightarrow \GG_m.$$
Hence we can set-up an analogous diagram:
$$\xymatrix{
\star & \star \ar[l] \ar[r] & \star\\
\AA^{n-1} -\{0\}\ar[d]\ar[u] & \AA^{n-1} -\{0\} \vee \GG_m \ar[l] \ar[d]\ar[u]
\ar[r] & \GG_m \ar[d]\ar[u]\\
\AA^{n-1} -\{0\} & \AA^{n-1} -\{0\} \times \GG_m \ar[l]  \ar[r] & \GG_m }$$
to conclude as in the base case that $S^1\wedge((\AA^{n-1} -\{0\})\wedge \GG_m) \simeq \AA^n-\{0\}$.
\end{proof}

\begin{corollary}\label{prop:S2n,n}
    In $\Spc^{\AA^1}_{S,\star}$ 
    there are $\AA^1$-weak equivalences $\AA^n/\AA^n-\{0\} \simeq S^n \wedge
    \mathbb{G}_m\we S^{2n,n}$ for $n\geq 1$.
\end{corollary}

\begin{proof} The homotopy cofiber of the inclusion $\AA^n-0 \hookrightarrow \AA^n$ is calculated as the homotopy pushout
$$\xymatrix{
    \AA^n-\{0\} \ar[r] \ar[d] & \AA^n \ar[d] \\
    \star \ar[r] & \AA^n/\AA^n -\{0\}.\\}$$
    In $\Spc_S^{\AA^1}$, this cofiber can be calculated as the homotopy pushout
    $$\xymatrix{
    \AA^n-\{0\} \ar[r] \ar[d] & \star \ar[d] \\
    \star \ar[r] & \AA^n/\AA^n -\{0\}.\\}$$
    Therefore $\AA^n/\AA^n -\{0\} \simeq S^1\wedge(\AA^n-\{0\})\simeq S^1
    \wedge (S^{2n-1, n})$ by Proposition~\ref{prop:S2n-1,n}.
\end{proof}

\begin{remark}
    In~\cite{asok-doran-fasel}, the authors study the question of when the
    motivic sphere $S^{a,b}$ is $\AA^1$-weak equivalent to a smooth scheme.
    Proposition~\ref{prop:S2n-1,n} shows that this is the case for
    $S^{2n-1,n}\we\AA^n-\{0\}$. Asok, Doran, and Fasel prove that it is also
    the case for $S^{2n,n}$, which they show is $\AA^1$-weak equivalent to  
    to the affine quadric with coordinate ring
    $$k[x_1,\ldots,x_n,y_1,\ldots,y_n,z]/\left(\sum_{i}x_iy_i-z(1+z)\right)$$
    when $S=\Spec k$ for a commutative ring $k$.
    They also show that $S^{a,b}$ is \emph{not} $\AA^1$-weak equivalent to a
    smooth affine scheme if $a>2b$. Conjecturally, the only motivic spheres
    $S^{a,b}$ admitting smooth models are those above, when
    $(a,b)=(2n-1,n)$ or $(a,b)=(2n,n)$.
\end{remark}

\begin{remark}
    If we impose only Nisnevich descent rather than Nisnevich hyperdescent,
    the results in this section remain true. This might provide one compelling
    reason to do so. For details, see~\cite{asok-hoyois-wendt}.
\end{remark}

\subsection{Affine and projective bundles}

\begin{proposition}  \label{prop:affinebundle} Let $p:E\rightarrow X$ be a
    Nisnevich-locally trivial affine space bundle. Then, $E\rightarrow X$ is an
    $\AA^1$-weak equivalence.
\end{proposition}

\begin{proof}  Pick a Zariski cover $\mathcal{U} :=
    \{U_{\alpha}\}$ of $X$ that
    trivializes $E$.  Suppose that $\check{C}(U)_{\bullet}$ is the \v{C}ech nerve of the cover,
    then we have a weak equivalence
    $$\hocolim_{\Delta^{op}}\check{C}(U)_{\bullet}\we X$$ in $\Spc_S$ and
    an $\AA^1$-weak equivalence $$\hocolim_{\Delta^{op}} \check{C}(U)_{\bullet}
    \times_X \AA^n\we E$$ in $\Spc_S^{\AA^1}$.
    But now, we have an levelwise-$\AA^1$-weak equivalence of simplicial
    objects $$\check{C}(U)_\bullet\times_X\AA^n\rightarrow\check{C}(U)_\bullet.$$
    Hence, the homotopy colimits are equivalent by construction.
\end{proof}

Note that the above proposition covers a larger class of morphisms than just
vector bundles $p:E\rightarrow X$. For these, the homotopy inverse of the
projection map is the zero section as per Exercise~\ref{exer:vb}.

We obtain immediate applications of this proposition in the form of certain
presentations of $\AA^n-0$ in terms of a homogeneous space and an affine scheme
in the $\AA^1$-homotopy category. We leave the proofs to the reader.

\begin{corollary}\label{cor:sln/n-1} Let $n \geq 2$, and let $\SL_n \rightarrow \AA^n -\{0\}$ be the
    map defined by taking the last column of a matrix in $\SL_n$. In
    $\Spc_S^{\AA^1}$, this map factors through the cofiber $\SL_n
    \rightarrow\SL_n/\SL_{n-1}$, and the map
    $\SL_n/\SL_{n-1}\rightarrow\AA^n-\{0\}$ is an $\AA^1$-weak equivalence.
\end{corollary}

\begin{corollary}\label{prop:quadrics} Let $S$ be the spectrum of a field,
    and give $\AA^{2n}$ the coordinates $x_1, ..., x_n, y_1, ...,y_n$.
    Consider the quadric $Q_{2n-1} = V(x_1y_1 + ... +x_ny_n =1)$. The
    map $Q_{2n-1} \rightarrow \AA^n -\{0\}$ induced by the projection to the
    $x$-coordinates is an $\AA^1$-weak equivalence.
\end{corollary}

Note that some authors might write $Q_{2n}$ for what we have written
$Q_{2n-1}$.

Now we will use the above proposition to deduce results about projective bundles. We will then recover yet another presentation of the spheres $S^{2n,n}$. Furthermore we will also introduce an important construction on vector bundles that we will encounter later.

\begin{definition} If $\nu:E\rightarrow X$ is a vector bundle, then the Thom space
    $\Th(\nu)$ of $E$ (sometimes also written $\Th(E)$) is defined as the cofiber $$E/(E - X),$$ where the
    embedding of $X$ into $E$ is given by the zero section.
\end{definition}

The Thom space construction plays a central role in algebraic topology and
homotopy theory, and is intimately wrapped up in computations of the bordism
ring for manifolds and in the representation of homology classes by manifolds~\cite{thom}. 

\begin{example} Let $S$ be a base scheme, then $\AA^n_S \rightarrow S$
    is a trivial vector bundle over $S$.
    The Thom space of the trivial rank $n$ vector bundle is then by definition
    $\frac{\AA^n}{\AA^n -\{0\}}$. From Proposition~\ref{prop:S2n,n} we conclude that Thom space in this case is given by $S^{2n,n}$.
\end{example}

In topology, one has a weak homotopy equivalence: $\CC\PP^n/\CC\PP^{n-1} \simeq S^{2n}$,
thanks to the standard cell decomposition of projective space.
One of the benefits of having this decomposition is that for
a suitable class of generalized cohomology theories, the complex orientable
theories,
there exists a theory of Chern classes similar to the theory in ordinary
cohomology. We would like a similar story in $\AA^1$-homotopy theory, and this indeed exists.

\begin{exercise}
    Let $E \rightarrow X$ be a trivial rank $n$ vector bundle, then there is an
    $\AA^1$-weak equivalence: $\Th(E) \we \PP^{1^{\wedge n}} \wedge X_+$. Hint:
    use Corollary ~\ref{prop:S2n,n}.
\end{exercise}

\begin{proposition} Suppose that $E \rightarrow X$ is a vector bundle and
    $\PP(E) \rightarrow \PP(E \oplus \mathcal{O})$ is the closed embedding at infinity.
    Then, there is an $\AA^1$-weak equivalence
    $$\frac{\PP(E \oplus \mathcal{O})}{\PP(E)} \rightarrow
    \Th(E).$$ 
\end{proposition}

\begin{proof}
    Throughout, $X$ is identified with its zero section for ease of notation.
    Observe that we have a morphism $X \rightarrow E \rightarrow \PP(E \oplus \mathcal{O})$ where the first map is the closed embedding of $X$ via the zero section and the second map is the embedding complementary to the embedding $\PP(E) \rightarrow \PP(E \oplus \mathcal{O})$ at infinity. We also identify $X$ in $\PP(E \oplus \mathcal{O})$ via this embedding.
    Hence, there is an elementary distinguished square
    $$\xymatrix{
    E - X \ar[r] \ar[d] & \PP(E \oplus \mathcal{O}) - X \ar[d]\\
    E \ar[r] & \PP(E \oplus \mathcal{O}),}$$
    which means that we have an weak equivalence of simplicial presheaves:
    $$\Th(E) \simeq \frac{\PP(E \oplus \mathcal{O})}{ \PP(E \oplus \mathcal{O}) - X }.$$
    We have a map $\frac{\PP(E \oplus \mathcal{O})}{\PP(E)}
    \rightarrow \frac{\PP(E \oplus \mathcal{O})}{ \PP(E \oplus
    \mathcal{O}) - X }$ because $\PP(E)$ avoids the embedding of $X$
    described above. This is the map that we want to be an $\AA^1$-weak
    equivalence, so it suffices to prove that we have an $\AA^1$-weak
    equivalence $ \PP(E) \rightarrow  \PP(E \oplus
    \mathcal{O}) - X$. The lemma below shows that the map is indeed the zero
    section of an affine bundle, and so we are done by
    Proposition~\ref{prop:affinebundle}.
\end{proof}

\begin{lemma} Let $X$ be a scheme, $p:E\rightarrow X$
    a vector bundle with $s$ its zero section.
    Consider the open embedding $j: E  \rightarrow \PP(E \oplus \mathcal{O}_X)$
    and its closed complement $i: \PP(E) \rightarrow \PP(E \oplus
    \mathcal{O}_S)$. In this case, there is a morphism $$q: \PP(E \oplus \mathcal{O}_X)
    \setminus j(s(X)) \rightarrow \PP(E)$$
    such that $q \circ i = id$ and $q$ is an $\AA^1$-bundle over $\PP(E)$.
\end{lemma}

\begin{proof}
    Recall that to give a morphism $T \rightarrow \mathbb{P}(E)$ over
    $X$, one must give a morphism $h: T \rightarrow X$ and a surjection $h^*(E)
    \rightarrow \mathcal{L}$ where $\mathcal{L}$ is a line bundle on $T$. Now
    $\bar{p}: \mathbb{P}(E \oplus \mathcal{O}_X) \rightarrow X$ has a universal
    line bundle $\mathcal{L}_{univ}$ and a universal quotient map $\bar{p}^*(E
    \oplus \mathcal{O}_X) \simeq \bar{p}^*(E) \oplus \mathcal{O}_{\mathbb{P}(E
    \oplus \mathcal{O}_X)} \rightarrow \mathcal{L}_{univ}$. Restricting to the
    first factor gives us a map $\bar{p}^*(E) \rightarrow \mathcal{L}_{univ}$
    and hence a rational map $t: \PP(E \oplus \mathcal{O}_X) \dashrightarrow
    \PP(E)$. Over a point of $X$, $t$ is given by projection onto
    $E$-coordinates. Hence, this map is well-defined away from $j(s(X))$, so we
    get a morphism $q: \PP(E \oplus \mathcal{O}_X)
    \setminus j(s(X)) \rightarrow \PP(E)$; by construction $q \circ i =
    \id_{\PP(E)}$.
    To check the last claim, since it is local on the base, one may assume that $E
    \iso \mathcal{O}^{n+1}_X$, so we are looking at $\PP^{n+1}_X \setminus X
    \rightarrow \PP^{n}$. In coordinates $X$ embeds as $[0: ... 0: 1]$, and the map
    is projection onto the first $n$ coordinates, which is an
    $\AA^1$-bundle.
\end{proof}

\begin{corollary} There are $\AA^1$-weak equivalences $\PP^n/\PP^{n-1} \simeq
    S^{2n,n}$ for $n\geq 1$ when $S$ is noetherian of finite Krull dimension.
\end{corollary}

\section{Classifying spaces in $\AA^1$-homotopy theory}\label{sec:classifying}

One of the main takeaways from Section~\ref{sec:topvect} is that one can go
very far using homotopical methods to study topological vector bundles on CW complexes. The
key inputs in this technique are the existence of the Postnikov tower and knowledge of the
homotopy groups of the classifying spaces
$\BGL_n$ in low degrees. In this section, we will give a sampler of the techniques involved in
accessing the $\AA^1$-homotopy sheaves of the classifying spaces $\BGL_n$.
In the end will identify a ``stable" range for these homotopy sheaves, which will naturally
lead us to a discussion of algebraic $K$-theory in the next section.

% A central theme in these calculations is the long exact sequence induced by a fibration.
% In topology, a principal $G$-bundle is in particular a Serre fibration of topological spaces
% which allows one to use the induced long exact sequence in homotopy groups to
% study them.
% Now, suppose that $G$ is a Nisnevich sheaf of groups on $\Sm_S$ and let $$G
% \rightarrow P \rightarrow X$$ be a $G$-torsor over a smooth $S$-scheme
% $X$. One is immediately tempted to say that the above induces a fiber sequence
% in $\AA^1$-homotopy theory. The problem turns out to be more subtle than
% expected. Before we can give results for when this happens, we will first study
% classifying spaces in $\AA^1$-homotopy theory.

As usual, $S$ is a quasi-compact and quasi-separated unless stated otherwise.

\subsection{Simplicial models for classifying spaces}\label{sub:classifying}

\begin{definition} Let $\tau$ be a topology (typically this will be Zariski,
    Nisnevich or \'{e}tale) on $\Sm_S$, and let $G$ a $\tau$-sheaf of groups. A
    $\tau$-$G$-torsor over $X \in \Sm_S$ is the data of a $\tau$-sheaf of sets $P$ on
    $\Sm_S$, a right action $a: P \times G \rightarrow P$ of $G$ on $P$, and a
    $G$-equivariant morphism $\pi: P \rightarrow X$ (where $X$ has the trivial
    $G$-action) such that
    \begin{enumerate}
        \item the morphism $(\pi, a): P \times G \rightarrow P \times_X P$ is
            an isomorphism, and
        \item there exists a $\tau$-cover $\{U_i\rightarrow X\}_{i\in I}$ of $X$ such
            that $U_i\times_X P\rightarrow U_i$ has a section for all $i\in I$.
    \end{enumerate}
\end{definition}

Let $G$ be a $\tau$-sheaf of groups.
Consider the simplicial presheaf $\E G$ described section-wise in the following
way: $\E G_{n}(U) =:  G(U)^{\times n +1}$ with the usual faces and
degeneracies. We write $\E_{\tau}G$ as a fibrant replacement in the model
category $\L_{\tau}(\sPre(\Sm_S))$.

\begin{proposition}
    There is a weak equivalence $\E_{\tau} G \simeq \star$ in $\L_{\tau}\sPre(\Sm_S).$
\end{proposition}

\begin{proof}
    The fact that each $\E G(U)$ is contractible is standard: the diagonal
    morphism: $G(U) \rightarrow G(U) \times G(U)$ produces an extra degeneracy.
    See Goerss and Jardine~\cite{goerss-jardine}*{Lemma~III.5.1}. Thus, $\E G\rightarrow\star$
    is a weak
    equivalence in $\sPre(\Sm_S)$. Since localization ($\tau$-sheafification) preserves weak equivalences,
    it follows that $\E_\tau G$ is contractible.
\end{proof}

There is a right $G$-action on $\E G$ by letting $G$ act on the last
coordinate in each simplicial degree. The level-wise quotient is the simplicial
presheaf we christen $\B G$. We write $\B_{\tau} G$ for a fibrant replacement
in the model category $\L_{\tau}\sPre(\Sm_S)$. We would like to make sense of
$\B_{\tau}G$ as a simplicial presheaf classifying $\tau$-$G$-torsors.

\begin{definition}
    Let $\B\mathrm{Tors}_{\tau}(G)$ be the simplicial presheaf which assigns to $U \in
    \Sm_S$ the nerve of the groupoid of $G$-torsors on $U$ and to a morphism $f: U'
    \rightarrow U$ a map of simplicial presheaves $\B\mathrm{Tors}_{\tau}(G)(U) \rightarrow
    \B\mathrm{Tors}_{\tau}(G)(U')$ induced by pullback.
\end{definition}

\begin{remark} The above definition is valid by the work of Hollander \cite{hollander}*{Section 3}; the functor that assigns to $U$ the nerve of the groupoid of $G$-torsors over $U$ does not have strictly functorial pullbacks and thus one needs to appeal to some rectification procedure.
\end{remark}

The following proposition is well known.

\begin{proposition}
    The simplicial presheaf $\B\mathrm{Tors}_{\tau}(G)$ is $\tau$-local.
\end{proposition}

\begin{proof}
    This follows from the local triviality condition and the fact that we can construct
    $\tau$-$G$-torsors by gluing; see, for example, \cite{vistoli}.
\end{proof}

Let $U \in \Sm_S$, we denote by $\H^1_{\tau}(U, G)$ be the (non-abelian)
cohomology set of $\tau$-$G$-bundles on $U$. More precisely, we set
$$\H^1_{\tau}(U, G) = \pi_0(\B\mathrm{Tors}_{\tau}(G)(U)).$$

\begin{proposition}\label{prop:BG-tors}
    Let $G$ be a $\tau$-sheaf of groups, then there is a natural weak equivalence $$\B_{\tau}G
    \rightarrow \B\Tors_{\tau}(G).$$
    Hence, for all $U \in \Sm_S$, there is a natural isomorphism $\pi_0(\B_{\tau} G(U))
    \iso \H^1_{\tau}(U;G)$ and a natural weak equivalence $\Rbf\Omega \B_{\tau} G(U) \simeq  G(U)$.
\end{proposition}

\begin{proof}
    A proof is given in \cite{morel-voevodsky}*{section 4.1}, we also recommend
    \cite{asok-hoyois-wendt-II}*{Lemma 2.2.2} and the references therein. Let us sketch the
    main ideas. To define a map to $\B\Tors_{\tau}(G)$, we can first define a map $\B G
    \rightarrow B\Tors_{\tau}(G)$ of presheaves and then use the fact that the target is
    $\tau$-local to get a map $\B_{\tau}G \rightarrow B\Tors_{\tau}(G)$. The former map is
    given by sending the unique vertex of $\B G(U)$ to the trivial $G$-torsor over $U$. Since
    $G$-torsors with respect to $\tau$ are $\tau$-locally trivial, we conclude that the map
    must be a $\tau$-local weak equivalence. The fact that $\B_\tau G$ is
    fibrant is by definition, and for $\B_\tau G$ it follows from~\cite{asok-hoyois-wendt-II}*{Lemma
    2.2.2}. The second part of the assertion then follows by definition, and the
    standard fact that loop space of the nerve of a groupoid is homotopy equivalent to
    the automorphism group of a fixed object.
\end{proof}

Many interesting objects in algebraic geometry, such as Azumaya algebras and the associated
$\PGL_n$-torsors, are only \'etale locally trivial. The classifying spaces of these
torsors are indeed objects of $\AA^1$-homotopy theory as we shall explain.
We can consider $\Sm_{S, \et}$, the full subcategory of the big \'{e}tale site over $S$
spanned by smooth $S$-schemes. Completely analogous to the Nisnevich case, one can develop
\'{e}tale-$\AA^1$-homotopy theory by the formula $\Spc^{\AA^1}_{S, \et} =
\L_{\AA^1}\L_{\et}\sPre(\Sm_S)$.

\begin{theorem}
    The morphism of sites: $\pi: \Sm_{S,\et} \rightarrow \Sm_{S, \Nis}$ induced by the
    identity functor induces a Quillen pair $$\pi^*: \Spc^{\AA^1}_{S} \leftrightarrows
    \Spc^{\AA^1}_{S, \et}: \pi_*,$$ and hence an adjunction $$\Lbf \pi^*: \Ho(\Spc^{\AA^1}_{S}) \leftrightarrows
    \Ho(\Spc^{\AA^1}_{S,\et}): \Rbf \pi_*$$ on the level of homotopy categories.
\end{theorem}
 
\begin{proof}
    Since our categories are constructed via Bousfield-localization of $\sPre(\Sm_S)$, the
    universal property tells us that to define a Quillen pair $$\pi^*: \Spc^{\AA^1}_{S}
    \leftrightarrows \Spc^{\AA^1}_{S, \et}: \pi_*$$ it suffices to define a Quillen pair:
    $$\pi^*: \sPre(\Sm_S) \leftrightarrows \Spc^{\AA^1}_{S, \et}: \pi_*$$ such that
    $\pi^*(i)$ is a weak equivalence for $i$ belonging to the class of Nisnevich hypercovers
    and $\AA^1$-weak equivalences. However, the model category $\Spc^{\AA^1}_{S, \et}$ is
    also constructed via Bousfield localization, so we use the Quillen pair from this
    Bousfield localization. But, it is clear that the identity functor
    $\pi^*:\sPre(\Sm_S)\rightarrow\sPre(\Sm_S)$ takes Nisnevich hypercovers to \'etale
    hypercovers and the morphisms $X\times_S\AA^1\rightarrow X$ to $X\times_S\AA^1$. Hence,
    the Quillen pair exists by the universal property of Bousfield localization.
\end{proof}
 
 %The proof of this theorem is mostly formal --- we indicate what must be checked. Firstly, the Quillen pair: $\pi^*: \L_{Nis}\sPre(\Sm_{S,Nis}) \leftrightarrows \L_{et}\sPre(\Sm_{S,et}): R\pi_*$ arises from the fact that $R\pi_*(\Gscr)$ is calculated by taking a fibrant replacement for $\Gscr$ and then applying the functor $\pi_*$ which has the effect of considering $\Gscr$ as Nisnevich sheaf. Since the \'{e}tale site is finer than the Nisnevich site $R\pi_*(\Gscr)$ is indeed a fibrant object in $L_{Nis}\sPre(Sm_{S,Nis})$. This checks that $R\pi_*$ preserves fibrant objects. 
 
 %Next, we need to check that $R\pi_*$ and $\pi^*$ preserves $\AA^1$-invariant objects. We note that $\pi^*(X \times_S \AA^1) \simeq X \times_S \AA^1$ which means that $R\pi_*$ preserves $\AA^1$-invariant objects.  

% \begin{remark}
%     Recall that a morphism of sites $f: T_1 \rightarrow T_2$ is a functor $f^{-1}:T_2
%     \rightarrow T_1$ such that $f_*: \Shv(T_1) \rightarrow \PreShv(T_2)$ defined by $\Gscr
%     \mapsto \Gscr \circ f^{-1}$ is a sheaf such that the adjoint $f^*$ (which exists by the
%     adjoint functor theorem) preserves limits. The above theorem is a consequence of
%     the analogous story working out for simplicial presheaves.
% \end{remark}

\begin{proposition}
    There are natural isomorphisms of Nisnevich sheaves $\pi_0^{\Nis}(\Rbf\pi_*\B_{\et} G) \simeq
    \H^1_{\et}(-; G)$ and $\pi_1^{\Nis}(\Rbf\pi_*\B_{\et} G) \simeq  G$, where the \'{e}tale sheaf of
    groups $G$ is considered as a Nisnevich sheaf.
\end{proposition}

\begin{proof}
    By adjunction, $[U, \Rbf \pi_*\B_{\et} G]_{\AA^1} \simeq [\Lbf \pi^*U,
    \B_{\et} G]_{\AA^1}
    \iso\H^1_{\et}(U,G)$. To see the $\pi_1$-statement, we note that $\Rbf \pi_*$ is a
    right Quillen functor and hence commutes with homotopy limits. Since the loop space is
    calculated via a homotopy limit, we have that $\Omega\Rbf\pi_* \B_{\et} G \simeq \Rbf
    \pi_* \Omega \B_{\et} G \simeq \Rbf \pi_* G$, as desired.
\end{proof}

\begin{example}
    Let $\G = \GL_n, \SL_n$ or $\Sp_{2n}$; these are the special groups in the sense of
    Serre. In this case, any \'{e}tale-$\G$-torsor is also a
    Zariski-locally trivial and hence a Zariski-$\G$-torsor (or a Nisnevich-$\G$-torsor).
    One way to say this in our language is to consider the
    Quillen adjunction $$\pi^*: \L_{\Nis}(\sPre(\Sm_S)) \leftrightarrows
    \L_{\et}(\sPre(\Sm_S)): \Rbf \pi_*.$$ Then there is a unit map $\B_{\Nis}G \rightarrow
    \Rbf\pi_*\pi^*\B_{\et}G$, which is an weak equivalence in the cases above.
\end{example}

\subsection{Some calculations with classifying spaces}

We are now interested in the $\AA^1$-homotopy sheaves of classifying spaces.
The first calculation is a direct consequence of the
unstable-$\AA^1$-0-connectivity theorem. We work over an arbitrary Noetherian
base in this section, unless specified otherwise.

\begin{proposition}\label{pi0BG}
    If $G$ is a Nisnevich sheaf of groups, then $\pi_0^{\AA^1}(\B G) = \star$.
\end{proposition}

\begin{proof}
    By Theorem~\ref{thm:unstabconn}, it suffices to prove that
    $\pi_0^{\Nis}(\B G)$ is trivial. Note that this is the sheafification  of
    the functor $U \mapsto\H^1_{\Nis}(U, G)$. The claim follows from the fact
    that we are considering $G$-torsors which are Nisnevich-locally trivial.
\end{proof}

\begin{remark}
    Let $G$ be an \'{e}tale sheaf of groups. If we replace $\B G$ by $\Rbf
    \pi_*\B_{\et}G$, then the above result will \emph{not} hold unless \'etale
    $G$-torsors are also Nisnevich locally trivial. This is not the case for
    example for $\PGL_n$. For more about $\B_{\et}\PGL_n$, see \cite{aravind-fundamentalgrp}*{Corollary 3.16}.
\end{remark}

In order to proceed further, we need a theorem of Asok-Hoyois-Wendt
\cite{asok-hoyois-wendt-II}.

% \begin{lemma}\label{lem:universaltorsor}
%     If $G$ is a Nisnevich sheaf of groups and $G \rightarrow \E G
%     \rightarrow \B G$ is an $\AA^1$-fiber sequence, then for any $G$-torsor
%     $P\rightarrow X$, the sequence
%     $G\rightarrow P \rightarrow X$ is also an $\AA^1$-fiber sequence.
% \end{lemma}
% 
% \begin{proof}
%     In any model category, fibrations are preserved under pullbacks.
% \end{proof}

\begin{theorem}[\cite{asok-hoyois-wendt-II}]\label{prop:fib-crit}
    If $X\rightarrow Y\rightarrow Z$ is a
    fiber sequence in $\sPre(\Sm_S)$ such that $Z$ satisfies affine
    Nisnevich excision and $\pi_0(Z)$ satisfies affine
    $\AA$-invariance, then $X\rightarrow Y\rightarrow Z$ is an $\AA^1$-fiber sequence.
\end{theorem}

\begin{corollary}
    If $\H^1_{\Nis}(-, G)$ is $\AA^1$-invariant, then the sequence
    $G \rightarrow \E G \rightarrow \B G$ is an $\AA^1$-fiber sequence.
%     Therefore if $P \rightarrow X$ is a $G$-torsor where
%     $\H^1_{\Nis}(-,G)$ is affine $\AA^1$-invariant, then  $G \rightarrow P
%     \rightarrow X$ is an $\AA^1$-fiber sequence.
\end{corollary}

From now on to the end of this section, we will need the base scheme to be a
field (although we can do better --- see the discussions in
\cite{asok-hoyois-wendt-II}) in order to utilize
$\AA^1$-invariance of various cohomology sets and apply Theorem~\ref{prop:fib-crit}
above. As a first example, we let $T$ be a split torus over a field $k$.

\begin{proposition}
    Let $T$ be a split torus over a field $k$.
    If $P \rightarrow X$ is a $T$-torsor with a $k$-point $x:\Spec k\rightarrow
    P$, then we have a short exact sequence $$1 \rightarrow \pi_1^{\AA^1}(P, x)
    \rightarrow \pi_1^{\AA^1}(X, x) \rightarrow T.$$
\end{proposition}

\begin{proof} We need to check that $\pi_0(\B T)$ is $\AA^1$-invariant. Recall that a split torus over a field simply means that it is isomorphic over $k$ to products of $\GG_m$, and so
    $\pi_0(\B T)\iso\Pic(-)^{\oplus n}$ where $n$ is the number of copies of $\GG_m$. Therefore
    it is indeed $\AA^1$-invariant on smooth $k$-schemes. This shows that $T$ is an
    $\AA^1$-rigid scheme over $k$, hence $\pi_0^{\AA^1}(T) \simeq T$ and the
    higher homotopy groups are zero by Proposition~\ref{prop:rigid}, giving us the short exact sequence above.
\end{proof}

\begin{remark} The result is true in greater generality for not-necessarily-split tori with some assumptions on the base field, see \cite{asok-toric} for details.
\end{remark}

\subsection{$\BGL$ and $\BSL$}\label{sec:kstablerange}

In our classification of vector bundles, on affine schemes, we need to
calculate the homotopy sheaves of $\BGL_n$. We use the machinery above to
highlight two features of this calculation. First, just like in topology, we
may reduce the calculation of homotopy sheaves of $\BGL_n$ to that of $\BSL_n$,
save for $\pi_1$. Secondly, the $\AA^1$-homotopy sheaves of $\BSL_n$
stabilize: for each $i$, $\pi_i^{\AA^1}(\BSL_n)$ is independent of the value
of $n$ as $n$ tends to $\infty$.

%In the case of $\G =\SL_,\GL_n$-torsors, classical results on the so-called special groups of Serre gives us that $H^1_{Nis}(X, G) \simeq H^1_{et}(X, G) \simeq H^1_{Zar}(X, G)$ which means there's no topological ambiguity when we are talking about $G$-torsors. 

\begin{proposition}\label{prop:naivea1sln}
    Let $S$ be a regular noetherian affine scheme of finite Krull dimension,
    and suppose that the Bass-Quillen conjecture holds for smooth schemes of finite
    presentation over $S$.
    The space $\SL_n$ in $\Spc_S^{\AA^1}$ is $\AA^1$-connected and $\BSL_n$ is $\AA^1$-1-connected,
    i.e. $\pi_1^{\AA^1}(\BSL_n) = \star$.
\end{proposition}

\begin{proof}
    We show that the sheaf $\pi_0^{\AA^1}(\SL_n)$ is trivial by showing that
    the stalks of $\pi_0^{\AA^1}(\SL_n)$ are trivial.
    To show this it suffices by Theorem~\ref{thm:naiveconn} to show that for any henselian local ring $R$,
    $$[\Spec\,R, \Sing^{\AA^1}(\SL_n)]_s = \star$$ (i.e. the set of naive
    $\AA^1$-homotopy classes is trivial), where
    we view $\Spec R$ as an object of $\Spc_S^{\AA^1}$ via the functor of
    points it represents.
    
   In fact we will prove the above claim for $R$, any local ring. We want to 
    connect any matrix $M\in\SL_n(R)$ to the identity via a chain of naive
    $\AA^1$-homotopies. Let $\mathfrak{m}$ be the maximal ideal of $R$, and let
    $k=R/\mathfrak{m}$ be the residue field. The subgroup
    $\E_n(k)\subseteq\SL_n(k)$ generated by the elementary matrices is actually all of $\SL_n(k)$, so we can write
    $\overline{M}$, the image of $M$ in $\SL_n(k)$, as a product of elementary
    matrices. Recall that an elementary matrix in $\SL_n(k)$ is the identity
    matrix except for a single off-diagonal entry. Since we can lift each of
    these to $\SL_n(R)$, we can write $M=EN$, where $E$ is a product of
    elementary matrices in $\SL_n(R)$ and $$N=\I_n+P,$$ where
    $P=(p_{ij})\in\M_n(\mathfrak{m})$ is a matrix with entries in $\mathfrak{m}$. Note
    that the condition that $N\in\SL_n(R)$ means that we can solve for
    $p_{11}$. Indeed,
    $$1=\det(N)=(1+p_{11})|C_{11}|-p_{12}|C_{12}|+\cdots+(-1)^np_{1n}|C_{1n}|,$$
    where $C_{ij}$ is the $ij$th minor of $N$. Each $p_{1r}$ is in $\mathfrak{m}$ for
    $2\leq r\leq n$. Hence, $1-n=(1+p_{11})|C_{11}|$, where $n\in\mathfrak{m}$.
    Since $1-n$ and $1+p_{11}$ are units, $|C_{11}|$ must be a unit in $R$ as
    well. Thus, we can solve $$p_{11}=\frac{1-n}{|C_{11}|}-1.$$ Now, define a
    new matrix $Q=(q_{ij})$ in $\M_n(\mathfrak{m}[t])$ by $q_{ij}=tp_{ij}$
    unless $(i,j)=(1,1)$, in which case set $q_{11}$ so that $\det(1+Q)=1$, using the
    formula above. Then, we see that $Q(0)=\I_n$, while $Q(1)=P$. It follows that
    $1+Q$ defines an explicit homotopy from $\I_n$ to $N=\I_n+P$. It follows that
    $M$ is $\AA^1$-homotopic to a product of elementary matrices. Since each
    elementary matrix is $\AA^1$-homotopic to $\I_n$, we have proved the claim.

    Now, by Theorem~\ref{prop:fib-crit}, $\SL_n \rightarrow \ESL_n \rightarrow \BSL_n$ is
    an $\AA^1$-fiber sequence due to the fact that $\SL_n$-torsors are
    $\AA^1$-invariant on smooth affine schemes (since $\GL_n$-torsors are
    $\AA^1$-invariant on smooth affine schemes). Therefore we have
    an exact sequence: $$\pi_1^{\AA^1}(\ESL_n) \rightarrow \pi_1^{\AA^1}(\BSL_n)
    \rightarrow \pi_0^{\AA^1}(\SL_n).$$ The left term is $\star$ since $\ESL_n$
    is simplicially (and hence $\AA^1$-)contractible and the right term is a
    singleton due to the first part of this proposition.
\end{proof}

\begin{exercise} Prove the following statements when $S$ satisfies the
    hypotheses of the previous theorem. For $i >2$, $\pi_i^{\AA^1}(\B\GG_m) = 0$. For
    $i = 1$, the sheaf of groups $\pi_i^{\AA^1}(\B\GG_m) \iso \GG_m$. Finally, $\pi_0^{\AA^1}(\B\GG_m) =
    \star$. Hint: use the $\AA^1$-rigidity of $\GG_m$ and
    Theorem~\ref{prop:fib-crit}.
\end{exercise}

\begin{proposition}\label{prop:slncover}
    Let $S$ be a regular noetherian affine scheme of finite Krull dimension,
    and suppose that the Bass-Quillen conjecture holds for smooth schemes of finite
    presentation over $S$.
    For $i >1$, the map $\SL_n \rightarrow \GL_n$ induces an isomorphism
    $\pi_i^{\AA^1}(\BSL_n) \rightarrow \pi_i^{\AA^1}(\BGL_n)$.
\end{proposition}

\begin{proof} By Theorem~\ref{prop:fib-crit},
    the sequence $\BSL_n \rightarrow \BGL_n \rightarrow \B\GG_m$ induces a long
    exact sequence of $\AA^1$-homotopy sheaves and the result for $i > 1$
    follows from the above proposition above. However we note that the case of
    $\pi_1^{\AA^1}$ is different: we have an exact sequence
    $\pi_1^{\AA^1}(\BSL_n) \rightarrow \pi_1^{\AA^1}(\BGL_n) \rightarrow
    \pi_1^{\AA^1}(\B\GG_m) \rightarrow \pi_0^{\AA^1}(\BSL_n)$; The groups on
    the right are zero by Proposition~\ref{pi0BG}, and the group on the left is zero by
    Proposition~\ref{prop:naivea1sln}.
\end{proof}

%\begin{theorem} Suppose that $n = 2$, then $\pi_1^{\AA^1}(SL_2) \simeq K^{MW}_2$
%\end{theorem}

Recall from Corollary~\ref{cor:sln/n-1} that we have an $\AA^1$-weak equivalence:
$\SL_{n+1}/\SL_{n}\rightarrow\AA^{n+1}-\{0\}$ for $n \geq 1$. Moreover,
$\AA^{n+1}-\{0\}$ is $\AA^1$-weak equivalent to $(S^1)^{\wedge n}\wedge\Gm^{\wedge n+1}$.
Our intuition from topology suggests therefore that
$\SL_{n+1}/\SL_n$ should be $(n-1)$-$\AA^1$-connected. This is indeed the case
but it relies on a difficult theorem of Morel, the unstable
$\AA^1$-connectivity theorem~\cite{morel}*{Theorem~6.38}. That theorem uses an
$\AA^1$-homotopy theoretic version of Hurewicz theorem and of $\AA^1$-homology
sheaves, which are defined not by pointwise sheafification but instead using
the so-called $\AA^1$-derived category.

We may apply Theorem~\ref{prop:fib-crit} to the fiber sequence of
simplicial presheaves: $\SL_{n+1}/\SL_{n} \rightarrow \BSL_n \rightarrow
\BSL_{n+1}$ to see that this is also an $\AA^1$-fiber sequence. We have thus
proved the following important stability result.

\begin{theorem}[Stability]\label{thm:stabilitysln}
    Let $S$ be a regular noetherian affine scheme of finite Krull dimension,
    and suppose that the Bass-Quillen conjecture holds for smooth schemes of finite
    presentation over $S$.
    Let $i>0$ and $n\geq 1$. The morphism $$\pi_{i}^{\AA^1}(\BSL_{n})
    \rightarrow \pi_{i}^{\AA^1}(\BSL_{n+1})$$ is an epimorphism if $i \leq n$
    and an isomorphism if $i \leq n-1$. 
\end{theorem}

Setting $\GL=\colim_{n \rightarrow \infty}\GL_n$ and similarly for
$\SL$, we obtain the following corollary.

\begin{corollary}\label{cor:stabilitygln}
    Let $S$ be a regular noetherian affine scheme of finite Krull dimension,
    and suppose that the Bass-Quillen conjecture holds for smooth schemes of finite
    presentation over $S$.
    For $i \geq 2$, we have $\pi_i^{\AA^1}(\BSL) \simeq
    \pi_i^{\AA^1}(\BGL)$.
\end{corollary}

\section{Representing algebraic $K$-theory}\label{sec:K}

One reason to contemplate the $\AA^1$-homotopy category is the fact that many invariants of schemes are $\AA^1$-invariant; one important example is algebraic $K$-theory, at least for regular schemes. The goal of this section is to prove the representability of algebraic
$K$-theory in $\AA^1$-homotopy theory and identify its representing space when
the base scheme $S$ is regular and noetherian. 
One important consequence is an
identification of the $\AA^1$-homotopy sheaves of the classifying
spaces $\BGL_n$'s in the stable range, which plays a crucial role in the classification of
algebraic vector bundles via $\AA^1$-homotopy theory. Indeed, it turns out that the
relationship between algebraic $K$-theory and these classifying spaces is just like what
happens in topology --- the latter assembles into a representing space for the former. This
is perhaps a little surprising as one way to define the algebraic $K$-theory of rings is
via the complicated $+$-construction which alters the homotopy type of $\BGL_n(R)$ rather
drastically. The key insight is that the $\Sing^{\AA^1}$ construction is an alternative to
the $+$-construction in nice cases, which leads to the identification of the representing
space in $\Spc^{\AA^1}(S)_{\star}$.

Throughout this section, we let $S$ be a fixed \emph{regular} noetherian scheme of finite Krull dimension. An argument
using Weibel's homotopy invariant $K$-theory will yield a similar result over an arbitrary
noetherian base, but for homotopy $K$-theory; for details
see \cite{cisinski}.

\subsection{Representability of algebraic $K$-theory}\label{sub:rerpresentability}

The first thing we note is that representability of algebraic $K$-theory in the
$\AA^1$-homotopy category itself is a formal consequence of basic properties of algebraic
$K$-theory.

\begin{proposition} \label{prop:formal}
    Let $S$ be a regular noetherian scheme of finite Krull dimension. Then, the $K$-theory space functor
    $\Kscr$ is a fibrant object of $\Spc^{\AA^1}_{S,\star}$. In particular,
    there are natural isomorphisms
    \begin{equation*}
        \K_i(X)\iso[\Sigma^i_+X,\Kscr]_{\AA^1}
    \end{equation*}
    for all finitely presented smooth $S$-schemes $X$ and all $i\geq 0$.
\end{proposition}

\begin{proof}
    It is enough to show that $\mathcal{K}$ is an $\AA^1$-local object of 
    $\sPre(\Sm_S)_{\star}$. For this, we must show that $\mathcal{K}$ is both a Nisnevich-local
    object and satisfies $\AA^1$-homotopy invariance. The first property
    follows from~\cite{thomason-trobaugh}*{Proposition 6.8}. The second
    property is proved in~\cite{thomason-trobaugh}*{Theorem 10.8} for the
    $K$-theory spectra. Since $\Rbf\Omega^\infty$ is a Quillen right adjoint, it preserves
    homotopy limits, and hence $\mathcal{K}$ also satisfies descent.
\end{proof}

Therefore algebraic $K$-theory is indeed representable in $\Spc^{\AA^1}_{S,
\star}$ by an object which we denote by $\mathcal{K}$. This argument is purely
formal. Next, we need to get a better grasp of the representing object
$\mathcal{K}$. To do so, we need some review on $H$-spaces.

%Let $\Vect(U)$ denote the category of big vector bundles on an $S$-scheme $U$ \cite{weibel-kbook}, \cite{Grayson}. Using this model, we define a pre-sheaf of categories: $U \mapsto \Vect(U)$ such that $\Vect(U)$ is equivalent to the usual category of vector bundles over $U$. Hence we get a simplicial pre-sheaf: $U \mapsto \B\Vect(U)$.

%We denote by $\Vect_d(U)$ and $\B\Vect_d(U)$, the analogs for Vector bundles of rank $d$.

%\begin{proposition} For any $d \in \mathbb{N}$ there is a Nisnevich local equivalence of simplicial pre-sheaves $\BGL_d\rightarrow \B\Vect_d$
%\end{proposition}

%\begin{proof} The map is induced by one that classifies the trivial Vector bundle over $BGL_d$. Since vector bundles are Nisnevich locally trivial, the map is an equivalence Nisnevich-locally.
%\end{proof}

%\begin{corollary} \label{cor:nis-loc} The morphism: $\coprod_{d \geq 0} \BGL_d \rightarrow \B\Vect$ is a Nisnevich local equivalence.
%\end{corollary}

\begin{definition}
    Let $X$ be a simplicial set. We say that $X$ is an
    $H$-space if it has a map $m: X \times X \rightarrow X$ and a point $e \in
    X$ which is a homotopy identity, that is, the maps $m(e, -), m(-, e): X
    \rightarrow X$ are homotopic to the identity map.
\end{definition}

\begin{exercise}
    Prove that the fundamental group of any $H$-space is always abelian.
\end{exercise}

\begin{definition}
    Let $X$ be a homotopy commutative and associative $H$-space. A group
    completion of $X$ is an $H$-space $Y$ together with an $H$-map $X
    \rightarrow Y$ such that
    \begin{enumerate}
        \item $\pi_0(X) \rightarrow \pi_0(Y)$ is a group completion of the
            abelian monoid $\pi_0(X)$, and
        \item for any commutative ring $R$, the homomorphism $\H_*(X; R)
            \rightarrow\H_*(Y;R)$ is a localization of the graded commutative
            ring $\H_*(X;R)$ at the multiplicative subset $\pi_0(X) \subset
            \H_0(X, R)$.
    \end{enumerate}
    We denote by $X^{\gp}$ the group completion of $X$. 
\end{definition}

There is a simple criterion for checking if a commutative and
associative $H$-space is indeed its own group completion.

\begin{definition}
    Let $X$ be an $H$-space. We say that $X$ is group-like if the monoid $\pi_0(X)$ is a group.
\end{definition}

The following proposition is standard. See~\cite{mcduff-segal} for example. A
specific model of the group completion of $X$ is $\Omega\B X$ when $X$ is
homotopy commutative and associative~\cite{segal}.

\begin{proposition}
    Let $X$ be a homotopy commutative and associative $H$-space, then the group
    completion of $X$ is unique up to homotopy and further, if $X$ is
    group-like, then $X$ is weakly equivalent to its own group-completion.
\end{proposition}

\begin{example}\label{ex:product}
    Let $R$ be an associative ring. We have maps $m: \GL_n(R) \times
    \GL_m(R) \rightarrow \GL_{m+n}(R)$ defined by block sum.
    This map is a group homomorphism and thus induces a
    map $$m: \left(\coprod_{n \geq 0} \BGL_n(R)\right)^{\times 2} \rightarrow \coprod_{n
    \geq 0} \BGL_n(R).$$ One easily checks that this is indeed a homotopy
    associative and homotopy commutative $H$-space.
\end{example}

\begin{remark}
    On the other hand we have the group $\GL(R) = \colim\,\GL_d(R)$ where the
    transition maps are induced by adding a single entry ``1" at the bottom right
    corner. We can take $\BGL(R)$, the classifying space of $R$. This space is
    \emph{not} an $H$-space: its fundamental group is $\GL(R)$ which is not an
    abelian group. For it to have any chance of being an $H$-space we need to
    perform the $+$-construction of Quillen which kills off a perfect normal
    subgroup of the fundamental group of a space and does not alter homology. For details see
    \cite{weibel-kbook}*{Section IV.1}. One key property of the $+$-construction
    that we will need is the following theorem of Quillen.
\end{remark}

\begin{theorem}[Quillen]\label{thm:h-space}
    Let $R$ be an associative ring with unit, the map $i: \BGL(R) \rightarrow
    \BGL(R)^+$ is universal for maps into $H$-spaces. In other words for each
    map $f: \BGL(R) \rightarrow H$ where $H$ is an $H$-space, there is a map
    $g: \BGL(R)^+ \rightarrow H$ such that $f \we g\circ i$ and the induced map
    on homotopy groups is independent of $g$.
\end{theorem}

\begin{proof}
    See \cite{weibel-kbook}*{Section IV.1 Theorem 1.8} and the references therein.
\end{proof}

Having this construction, the two spaces we discussed are intimately related.

\begin{theorem}[Quillen]\label{thm:gp-compl}
    Let $R$ be an associative ring with unit, then the group completion of
    $\coprod \BGL_n(R)$ is weakly equivalent to $\ZZ \times \BGL(R)^{+}$.
\end{theorem}

See \cite{weibel-kbook} for a proof. The plus construction alters the homotopy
type of a space rather drastically. There are other models for the plus
construction like Segal's $\Omega\B$ construction mentioned above. The $\Sing^{\AA^1}$-construction turns out to
provide another model, as we explain in the next section.

\subsection{Applications to representability}

The following theorem was established by Morel and
Voevodsky~\cite{morel-voevodsky}, although a gap was pointed out by Schlichting
and Tripathi~\cite{schlichting-tripathi}, who also provided a fix.

\begin{theorem}\label{thm:krep}
    Let $S$ be a regular noetherian scheme of finite Krull dimension.
    The natural map $\ZZ\times\BGL\rightarrow\Kscr$ in $\Spc_S^{\AA^1}$ is an $\AA^1$-local weak
   equivalence.
\end{theorem}

There is an $\AA^1$-local weak equivalence $$\left(\coprod_n\BGL_n\right)^{\gp}\we\Kscr$$ from
Theorem~\ref{thm:gp-compl} since the $+$-construction is one way to obtain
$\K(R)$ by Quillen. (Note that sheafification takes care of the fact
that there might be non-trivial finitely generated projective $R$-modules.)
Hence, to prove the theorem, we must
construct an $\AA^1$-weak equivalence
$$\ZZ\times\BGL\we\left(\coprod_n\BGL_n\right)^{\gp}.$$
We have already mentioned that $\BGL$ is not an $H$-space, so that group
completion will not formally lead to a weak equivalence. It is rather the
$\Sing^{\AA^1}$-construction which leads to an $H$-space structure on the
$\AA^1$-localization of $\ZZ\times\BGL$.

\begin{lemma}
    If $R$ is a commutative ring, then $\Sing^{\AA^1}\BGL(R)$ is an $H$-space.
\end{lemma}

\begin{proof}
    See~\cite{weibel-kbook}*{Exercise IV.11.9}.
\end{proof}

\begin{proposition}
    If $R$ is a commutative ring,
    the natural map $$\Sing^{\AA^1}\BGL(R) \rightarrow
    \Sing^{\AA^1}\BGL(R)^+$$ is a weak equivalence.
\end{proposition}

\begin{proof}
    The map is a homology equivalence since each
    $\BGL(\Delta^n_R)\rightarrow\BGL(\Delta^n_R)^+$ is a homology equivalence.
    Since both sides are group-like $H$-spaces they are nilpotent, so the fact
    that the map is a homology equivalence implies that it is a weak homotopy
    equivalence.
\end{proof}

%     By Theorem~\ref{thm:h-space}, there is an $H$-space map $\BGL(R)^+ \rightarrow
%     \Sing^{\AA^1}\BGL(R)$. Thus the induced map
%     $\Sing^{\AA^1}\BGL(\Delta^{\bullet}_R)^+ \rightarrow \Sing^{\AA^1}\BGL(R)$
%     is a also an $H$-space map since $\Sing^{\mathbb{A}^1}$ commutes with products.
%     We recall the fact from topology that if $f:X \rightarrow Y$ is a
%     map of path-connected $H$-spaces and if $X \rightarrow Y$ induces a homology
%     equivalence with integral coefficients, then $f$ is indeed a weak
%     equivalence.

%\begin{theorem} \label{thm:maincalc} Let $R$ be a regular commutative ring, then there exists an equivalence of simplicial sets: $ \Sing^{\AA^1}(\BGL(R) \times \ZZ) \rightarrow \Omega \B(\coprod_{d \geq 0} \B \GL_d(R))$
%\end{theorem}

%\begin{proof} 

%We claim that $\Sing^{\AA^1}(\BGL(R) \times \ZZ)$ is a group complete $H$-space.
   %Since $\Sing^{\AA^1}$ commutes with finite products, we claim that $\pi_0(\Sing^{\AA^1}(\BGL(R)) = %\star$ which is true.
    %Now what is left to prove is that we have an homology equivalence for all commutative rings $S$: $%\H_i(\Sing^{\AA^1}(\BGL(R)); S) \simeq \H_i(\coprod_{n \geq 0} \BGL_n(R); S) [\pi_0^{-1}]$.
    %Since $R$ is regular, \todo{we note that $\BGL(\AA^n_R) \simeq \BGL(R)$} and
    %therefore we have an equivalence of simplicial sets:
    %\todo{$Sing^{\AA^1}(\BGL(R)) \simeq \BGL(R)$}, but we know that $\H_i(BGL(R)
    %\times \ZZ; S) \simeq \H_i(\BGL(R)^{+} \times \ZZ; S)$ since the
    %$+$-construction does not alter homology. We then conclude by Theorem~\ref{thm:gp-compl} 
%\end{proof}

\begin{proof}[Proof of Theorem~\ref{thm:krep}]
    More precisely, we claim that we have a weak equivalence of
    simplicial presheaves $\L_{\AA^1}\L_{\Nis} \left(\ZZ\times\BGL\right) \rightarrow
    \mathcal{K}$. Since both simplicial presheaves are, in particular, Nisnevich local we need
    only check on stalks. Therefore we need only check that
    $\Sing^{\AA^1}\left(\ZZ\times\BGL(R)\right) \simeq \mathcal{K}(R)$ for $R$
    a regular noetherian local ring because after this, further application of
    $\Sing^{\AA^1}$ does not change the stalk by $\AA^1$-homotopy invariance. but
    this follows from our work
    above since $\mathcal{K}(R) \simeq\K_0(R)\times\BGL(R)^+$. Since
    $\mathcal{K}$ is $\AA^1$-invariant on smooth affine schemes, the natural
    map $$|\K_0(R) \times\BGL(\Delta^{\bullet}_R)^+| \rightarrow
    \K_0(R) \times\BGL(R)^+$$ is a weak equivalence, and the result follows
    from Quillen.
\end{proof}

\begin{remark} One can further prove that $\BGL$ is represented in
    $\Spc^{\AA^1}_{S, \star}$ by the Grassmanian schemes. In order to do
    this, Morel and Voevodsky used an elegant model for classifying spaces
    in~\cite{morel-voevodsky}, also considered by
    Totaro~\cite{totaro}. 
\end{remark}

As a corollary, we get a calculation of the stable range of the $\AA^1$-homotopy sheaves of $\BGL_n$ and $\BSL_n$. 

\begin{corollary}
    Let $i>1$ and $n\geq 1$. Then if $i \leq n-1$, we have isomorphisms
    $$\pi_i^{\AA^1}\BSL_n\iso\pi_{i}^{\AA^1}(\BGL_{n})
    \iso \K_i.$$
\end{corollary}

\begin{proof}
    This follows from the stable range results in
    Theorem~\ref{thm:stabilitysln} and Corollary~\ref{cor:stabilitygln}.
\end{proof}

% Purity
\section{Purity}\label{sec:purity}

In this section, we prove the purity theorem. The
theorem has its roots in the following theorem from \'{e}tale cohomology: suppose
that $k$ is an algebraically closed field with characteristic prime to an integer $n$ and $Z
\hookrightarrow X$ is a regular closed immersion of $k$-varieties. Suppose further that $Z$
is of pure codimension $c$ in $X$. Then, for any locally constant sheaf of
$\ZZ/n$-modules $\Fscr$
there is a canonical isomorphism
$$g: \H^{r-2c}_{\et}(Z, \Fscr(-n)) \rightarrow \H^r_Z(X, \Fscr),$$
the \textbf{purity isomorphism}.
Here $\H^r_Z(X, -)$ is the \'etale cohomology of $X$ with supports on $Z$, which is characterized as the group fitting into the  long exact sequence
$$\cdots\rightarrow\H^r_Z(X, \Fscr) \rightarrow \H^r_{\et}(X, \Fscr) \rightarrow
\H^r_{\et}(X - Z, \Fscr) \rightarrow \H^{r+1}_Z(X, \Fscr) \rightarrow \cdots.$$
Substituting the isomorphism above into the long exact sequence we obtain the \textbf{Gysin sequence}
$$\cdots\rightarrow\H^{r-2c}_{\et}(Z; \Fscr(-c)) \rightarrow \H^r_{\et}(X; \Fscr) \rightarrow
\H^r_{\et}(X - Z; \Fscr) \rightarrow \H^{r+1-2c}_{\et}(Z; \Fscr(-c)) \rightarrow\cdots.$$
The Gysin sequence is extremely useful for calculation: the naturality of the long exact sequence and 
purity isomorphism leads to calculations of Frobenius weights of smooth varieties $U$ by
embedding them into a smooth projective variety $U \hookrightarrow X$ whose complement is
often a normal crossing divisor~\cite{deligne-icm}.

In topology, the Gysin sequence is also available and is deduced in the following way.
Suppose that $Z \hookrightarrow X$ is a closed immersion of smooth manifolds of (real)
codimension $c$ and $\nu_Z$ is the normal bundle of $Z$ in $X$. The tubular neighborhood
theorem identifies the Thom space of $\nu_Z$ with the cofiber of $X - Z
\rightarrow X$, i.e. there is a weak homotopy equivalence $$\Th(\nu_Z) \simeq \frac{X}{X-Z}.$$
One then proves that there is an isomorphism
$$\tilde{\H}^{i-c}(Z; k) \rightarrow \tilde{\H}^i(\Th(\nu_Z); k)$$
in reduced singular
cohomology with coefficients in a field $k$. In fact, this last isomorphism is true if we replace
ordinary singular cohomology with any complex-oriented cohomology theory
~\cite{may-concise}. Therefore, the crucial step is identifying the cofiber
$\frac{X}{X - Z}$ with the Thom space of the normal bundle. In this light,
the purity theorem in $\AA^1$-homotopy theory may be interpreted as a kind of
tubular neighborhood theorem.

We will now prove the crucial purity theorem of Morel and
Voevodsky~\cite{morel-voevodsky}*{Theorem~2.23}. We benefited from unpublished
notes of Asok and from the exposition of Hoyois in~\cite{hoyois}*{Section~3.5} in the equivariant case. We follow
the latter closely below. The discussion in this section is valid for $S$ a quasi-compact quasi-separated base scheme $S$.

\begin{definition}
    A \df{smooth pair} over a scheme $S$ is a closed embedding $i:Z\hookrightarrow
    X$ of finitely presented smooth $S$-schemes. We will often write such a pair as $(X,Z)$,
    omitting reference to the map $i$. The smooth pairs over $S$ form a
    category $\Sm_S^{\mathrm{pairs}}$ in which the morphisms $(X,Z)\rightarrow
    (X',Z')$ are pullback squares
    \begin{equation*}
        \xymatrix{
            Z\ar[r]\ar[d]   &   X\ar[d]\\
            Z'\ar[r]        &   X'.
        }
    \end{equation*}
    A morphism $f:(X,Z)\rightarrow(X',Z')$ of smooth pairs is \df{Nisnevich} if
    $f:X\rightarrow X'$ is \'etale and if $f^{-1}(Z')\rightarrow Z'$ is an
    isomorphism.
\end{definition}

We will need following local characterization of smooth pairs.

\begin{proposition}\label{prop:nislocalpair}
    Let $i:Z\rightarrow X$ be a smooth pair over a quasi-compact and
    quasi-separated scheme $S$. Assume that the
    codimension of $i$ is $c$ along $Z$. Then, there is a
    Zariski cover $\{U_i\rightarrow X\}_{i\in I}$ and a set of \'etale morphisms
    $\{U_i\rightarrow\AA^{n_i}_S\}_{i\in I}$ such that the smooth pair $U_i\times_X
    Z\rightarrow U_i$ is isomorphic to the pullback of the inclusion of a
    linear subspace $\AA^{n_i-c}_S\rightarrow\AA^{n_i}_S$ for all $i\in I$.
\end{proposition}

\begin{proof}
    See~\cite{sga1}*{Th\'eor\`eme~II.4.10}.
\end{proof}

Certain moves generate all smooth pairs, which lets one prove statements for
all smooth pairs by checking them locally, checking that they transport along
Nisnevich morphisms of smooth pairs, and checking that they hold for
zero sections of vector bundles.

\begin{lemma}\label{lem:pairgeneration}
    Suppose that $\Pbf$ is a property of smooth pairs over a quasi-compact and
    quasi-separated scheme $S$ satisfying the
    following conditions:
    \begin{enumerate}
        \item if $(X,Z)$ is a smooth pair and if $\{U_i\rightarrow X\}$ is a
            Zariski cover such that $\Pbf$ holds for
            $(U_{i_1}\times_X\cdots\times_X
            U_{i_n},Z\times_XU_{i_1}\times_X\cdots\times_X U_{i_n})$ for all
            tuples $i_1,\ldots,i_n\in I$, then $\Pbf$ holds for $(X,Z)$;
        \item if $(V,Z)\rightarrow(X,Z)$ is a Nisnevich morphism of smooth
            pairs, then $\Pbf$ holds for $(V,Z)$ if and only if $\Pbf$
            holds for $(X,Z)$;
        \item $\Pbf$ holds for all smooth pairs of the form $(\AA^n_Z,Z)$.
    \end{enumerate}
    Then, $\Pbf$ holds for all smooth pairs over $S$.
\end{lemma}

\begin{proof}
    By (1), it suffices to check that $\Pbf$ is true Zariski-locally on $X$.
    Pick a Zariski cover $\{U_i\rightarrow X\}$ satisfying the conclusion of
    Proposition~\ref{prop:nislocalpair}. Thus, the problem is reduced to
    showing that if $(X,Z)\rightarrow(\AA^n,\AA^m)$ is a map of smooth pairs
    with $X\rightarrow\AA^n$ \'etale, then $\Pbf$ holds for $(X,Z)$. Indeed,
    all pairs $(U_{i_1}\times_X\cdots\times_X
    U_{i_n},Z\times_XU_{i_1}\times_X\cdots\times_X U_{i_n})$
    have this form by our choice of cover.
    The rest of the argument follows~\cite{morel-voevodsky}*{Lemma~2.28}. Form
    the fiber product $X\times_{\AA^n}(Z\times_S\AA^c)$, where $c=n-m$, and
    $Z\times_S\AA^c\rightarrow\AA^n$ is the product of the maps $Z\rightarrow\AA^m$ and
    $\AA^c\xrightarrow{\id}\AA^c$. Since $Z\rightarrow\AA^m$ is \'etale, we see
    that $Z\times_{\AA^m}Z\subseteq X\times_{\AA^n}(Z\times\AA^c)$ is the
    disjoint union of $Z$ and some closed subscheme $W$. Let
    $U=X\times_{\AA^n}(Z\times_S\AA^c)-W$. The projection maps induce
    Nisnevich maps of pairs $(U,Z)\rightarrow(X,Z)$ and
    $(U,Z)\rightarrow(Z\times_S\AA^c,Z)$. By (3), $\Pbf$ holds for
    $(Z\times_S\AA^c,Z)$ and hence for $(U,Z)$ by (2), and hence for $(X,Z)$ by
    (2) again.
\end{proof}

\begin{definition}
    A morphism $(X,Z)\rightarrow(X',Z')$ of smooth pairs over $S$ is \df{weakly excisive} if 
    the induced square
    $$\xymatrix{
        Z\ar[d]\ar[r]  &   X/(X-Z)\ar[d]\\
        Z'\ar[r]         &   X'/(X'-Z')
    }$$
    is homotopy cocartesian in $\Spc_S^{\AA^1}$.
\end{definition}

The following exercise is used in the proof of the purity theorem.

\begin{exercise}[\cite{hoyois}*{Lemma~3.19}]\label{ex:weakly}
    Let $(X,Z)\xrightarrow{f}(X',Z')\xrightarrow{g}(X'',Z'')$ be composable
    morphisms of smooth pairs over $S$. Prove the following statements.
    \begin{enumerate}
        \item   If $f$ is weakly excisive, then $g$ is weakly excisive if and
            only if $g\circ f$ is weakly excisive.
        \item   If $g$ and $g\circ f$ are weakly excisive, and if
            $g:Z'\rightarrow Z''$ is an $\AA^1$-local weak equivalence, then $f$ is weakly
            excisive.
    \end{enumerate}
\end{exercise}

Finally, we come to the purity theorem itself.

\begin{theorem}[Purity theorem~\cite{morel-voevodsky}*{Theorem~2.23}]
    Let $Z \hookrightarrow X$ be a closed embedding in $\Sm_S$ where $S$ is
    quasi-compact and quasi-separated. If $\nu_Z:
    \N_XZ \rightarrow Z$ is the normal bundle to $Z$ in $X$, then there
    is an $\AA^1$-local weak equivalence
    $$\frac{X}{X - Z} \rightarrow\Th(\nu_Z)$$
    which is natural in $\Ho(\Spc_S^{\AA^1})$ for smooth pairs $(X,Z)$ over $S$.
\end{theorem}

\begin{proof}
    First, we construct the map. Consider the construction
    $$\D_ZX=\Bl_{Z\times_S\{0\}}(X\times_S\AA^1)-\Bl_{Z\times_S\{0\}}(X\times_S\{0\}),$$
    which is natural in smooth pairs $(X,Z)$. The fiber of
    $\D_ZX\rightarrow\AA^1$ at $\{0\}$ is the complement
    $\PP(\N_ZX\oplus\Oscr_Z)-\PP(\N_ZX)$, which is naturally isomorphic to the
    vector bundle $\N_ZX$. Hence, by taking the zero section at $\{0\}$, we get
    a closed embedding $Z\times_S\AA^1\rightarrow\D_ZX$. The fiber at $\{0\}$ of
    $(\D_ZX,Z\times_S\AA^1)$ is $(\N_ZX,Z)$, while the fiber at $\{1\}$ is
    $(X,Z)$. Thus,
    there are morphisms of smooth
    pairs $$(X,Z)\xrightarrow{i_1}(\D_ZX,Z\times_S\AA^1)\xleftarrow{i_0}(\N_ZX,Z),$$ and it
    is enough to prove that $i_1$ and $i_0$ are weakly excisive for all smooth
    pairs $(X,Z)$. Indeed, in that case there are natural $\AA^1$-weak equivalences $X/(X-Z)\we
    \D_ZX/(\D_ZX-Z\times_S\AA^1)\we\N_ZX/(\N_ZX-Z)=\Th(\nu_Z)$ because the cofiber of
    $Z\rightarrow Z\times_S\AA^1$ is contractible.

    Let $\Pbf$ hold for the smooth pair $(X,Z)$ if and only if $i_0$ and $i_1$
    are excisive. We show that $\Pbf$ satisfies conditions (1)-(3) of
    Lemma~\ref{lem:pairgeneration}.

    Let $\{U_i\rightarrow X\}_{i\in I}$ be a Zariski cover of $X$, and let
    $(U_{i_1,\ldots,i_n},Z_{i_1,\ldots,i_n})\rightarrow(X,Z)$ be the induced morphisms of smooth pairs. For
    Suppose that $\Pbf$ holds for each
    $(U_{i_1,\ldots,i_n},Z_{i_1,\ldots,i_n})$. Then, there is a diagram
    \begin{equation*}
        \xymatrix{
            |Z_\bullet|\ar[r]^{i_1}\ar[d]   & |U_\bullet/(U_\bullet-Z_\bullet)|\ar[d]\\
            |Z_\bullet\times_S\AA^1|\ar[r]^>>>>{i_1} & |\D_{Z_i}U_i/(D_{Z_i}U_i-Z_i\times_S\AA^1)|
        }
    \end{equation*}
    of geometric realizations. However, this is the geometric realization of a
    simplicial cocartesian square by hypothesis, so it is itself cocartesian.
    The same argument works for $i_0$, so we see that $\Pbf$ satisfies (1).

    Consider a Nisnevich morphism $(V,Z)\rightarrow(X,Z)$ of smooth pairs, and
    consider the diagram
    \begin{equation*}
        \xymatrix{
            (V,Z)\ar[r]^{i_1}\ar[d] &   (\D_ZV,Z\times_S\AA^1)\ar[d] &
            (\N_ZV,Z)\ar[d]\ar[l]^{i_0}\\
            (X,Z)\ar[r]^{i_1} &   (\D_ZX,Z\times_S\AA^1)       &
            (\N_ZX,Z)\ar[l]^{i_0}.
        }
    \end{equation*}
    We leave it as an easy exercise to the reader to show using
    Exercise~\ref{ex:weakly} that (2) will follow if
    the vertical arrows are all weakly excisive. But,
    the vertical maps are all Nisnevich morphisms. So, it is
    enough to check that Nisnevich morphisms $(V,Z)\rightarrow(X,Z)$ of smooth pairs are weakly excisive.
    Let $U$ be the complement of $Z$ in $X$. By hypothesis, the diagram
    $$\xymatrix{
        U\times_X V\ar[r]\ar[d] &   V\ar[d]\\
        U\ar[r]                       &   X
    }$$
    is an elementary distinguished square, and hence a homotopy cocartesian square in
    $\Spc_S^{\AA^1}$ by Proposition~\ref{prop:edspushout}. In particular, the
    cofiber of $V/(U\times_X V)\rightarrow X/U$ is contractible. Since the
    cofiber of $Z\rightarrow Z$ is obviously contractible, this proves
    $(V,Z)\rightarrow(X,Z)$ is weakly excisive.

    To complete the proof, we just have to show that (3) holds. Of course, in
    the situation $(\AA^n_Z,Z)$ of (3) we can prove the main result of the theorem quite
    easily. However, the structure of the proof requires us to check weak
    excision for $i_0$ and $i_1$. 
    For this we can immediately reduce to the case where $Z=S$, which
    we omit from the notation for the rest of the proof. 
    The blowup
    $\Bl_{\{0\}}(\AA^n\times\AA^1)$ is the total space of an $\AA^1$-bundle
    over $\PP^n$, and the image of $\Bl_{\{0\}}(\AA^n)$ in $\PP^n$ is a hyperplane
    $\PP^{n-1}\subseteq\PP^n$. Hence, there is a morphism of pairs
    $$(\D_{\{0\}}(\AA^n),\{0\}\times\AA^1)\xrightarrow{f}(\AA^n,\{0\}).$$
    Since $\D_{\{0\}}(\AA^n)\rightarrow\AA^n$ is the total space of an
    $\AA^1$-bundle, this morphism is weakly excisive. The composition of $f$
    with $i_1$ is the identity on $(\AA^n,\{0\})$ and hence is
    weakly excisive as well. By Exercise~\ref{ex:weakly}(2), it follows that
    $i_1$ is weakly excisive. Similarly, $f\circ i_0$ is
    the identity on $(\N_{\{0\}}\AA^n,\{0\})\iso(\AA^n,\{0\})$, so $i_0$ is weakly excisive,
    again by Exercise~\ref{ex:weakly}(2).
\end{proof}
\section{Vista: classification of vector bundles}\label{sec:asok-fasel}

In this section we give a brief summary of how to use the theory developed
above to give a straightforward proof of the classification of vector bundles
on smooth affine curves and surfaces. We must understand how to compute
$\AA^1$-homotopy classes of maps to $\BGL_n$ so we can apply the Postnikov
obstruction approach to the classification problem.

Recall from Theorem~\ref{thm:explicit} that
the $\AA^1$-localization functor above may be calculated as a transfinite
composite of $\L_{\Nis}$ and $\Sing^{\AA^1}$.
This process is rather unwieldy. However, things get better if this process
stops at a finite stage, in particular suppose that $\Fscr$ was already a
Nisnevich-local presheaf and suppose that one can somehow deduce that
$\Sing^{\AA^1}X$ was already Nisnevich local, then one can conclude that
$\L_{\AA^1}X \simeq \Sing^{\AA^1}X$ and therefore, using our
formulae for mapping spaces in the $\AA^1$-homotopy category, one concludes
that $$[U,X]_{\AA^1}\iso[U,\L_{\AA^1}X]_s \simeq \pi_0\Sing^{\AA^1}X(U).$$
In general, this does not work.

\begin{remark}
    Work of Balwe-Hogadi-Sawant~\cite{balwe-hogadi-sawant} constructs explicit
    smooth projective varieties $X$ over $\mathbb{C}$ for which
    $\Sing^{\AA^1}X$ is \textit{not} Nisnevich-local, so extra conditions must be
    imposed to calculate the $\AA^1$-homotopy classes of maps naively. However,
    there is often an intimate relation between naive $\AA^1$-homotopies and
    genuine ones: Cazanave constructs in~\cite{cazanave} a monoid structure on
    $\pi_0\Sing^{\AA^1}(\PP^1_k)$ and proves that the map
    $\pi_0\Sing^{\AA^1}(\PP^1_k) \rightarrow [\PP^1_k, \PP^1_k]_{\AA^1}$ is
    group completion, with the group structure on the target induced by the
    $\AA^1$-weak equivalence $\PP^1 \simeq S^1 \wedge \GG_m$.
\end{remark}

\begin{exercise}
    Even for fields, one can show that the sets of isomorphism classes of vector bundles
    over the simplest non-affine scheme $\PP^1_k$ are not $\AA^1$-invariant.
    Construct (e.g. write down explicit transition functions) a vector bundle
    over $\mathbb{P}^1 \times_k \mathbb{A}^1$ that restricts to $\mathcal{O}(0)
    \oplus \mathcal{O}(0)$ on $\PP^1 \times \{1\}$ and $\mathcal{O}(1) \oplus
    \mathcal{O}(-1)$ on $\mathbb{P}^1 \times \{0\}$ for a counter-example.
\end{exercise}

\begin{remark}
    In~\cite{asok-doran}, examples are given of smooth $\AA^1$-contractible varieties
    with families of non-trivial vector bundles of any given dimension. Were
    vector bundles to be representable in $\Spc^{\AA^1}$, such pathologies
    could not occur. These varieties are non-affine.
\end{remark}

As the exercise and remark show, the only hope for computing vector bundles as
$\AA^1$-homotopy classes of maps to $\BGL_n$ is to restrict to affine schemes,
but even there it is not at all obvious that this is possible, as the map
$\BGL_n\rightarrow\L_{\AA^1}\BGL_n$ is not a simplicial weak equivalence.
Remarkably, despite this gulf, Morel and later Asok-Hoyois-Wendt showed that
for smooth affine schemes one can compute vector bundles in this way. In fact,
this follows from a much more general and formal result, which we now explain.

%\subsection{The Affine Homotopy Category}

%Since open immersions of affine schemes may not be affine, it is not immediate restricting the Nisnevich site to affine schemes over a base $S$ characterizes fibrant sheaves in terms of affine Nisnevich excision.

%\begin{definition} Let $Sm/S^{aff}$ be the category of smooth affine schemes over $S$. A Nisnevich cover of a $U$ in $Sm/S^{aff}$ is a Nisnevich cover of $U$ in $Sm/S$ $\{U_i \rightarrow U\}$ where each $U_i \rightarrow U$ is a smooth affine scheme over $S$. 
%\end{definition}

%\begin{theorem} [Asok-Hoyois-Wendt] Suppose that $\Fscr$ has affine Nisnevich excision, then for all $U \in Sm^{aff}/S$, we have that $\Fscr(U) \rightarrow L_{nis}\Fscr(U)$ is an equivalence. 
%\end{theorem}

%Thus the content of the above theorem is that affine Nisnevich descent can still be characterized in terms of elementary distinguished squares with each object being an affine $S$-scheme. 

We say that a presheaf $F$ of sets on $\Sm_S$ satisfies \df{affine
$\AA^1$-invariance} if
the pullback maps $F(U)\rightarrow F(U\times_S\AA^1)$ are isomorphisms for all
finitely presented smooth affine $S$-schemes $U$. Note that we say that an
$S$-scheme is affine if $U\rightarrow\Spec\ZZ$ is affine, so that $U=\Spec R$
for some commutative ring $R$.

\begin{theorem}[\cite{asok-hoyois-wendt}]
    Let $S$ be a quasi-compact and quasi-separated scheme.
    Suppose that $X$ is a simplicial presheaf on $\Sm_S$. Assume that
    $\pi_0(X)$ is affine $\AA^1$-invariant and that $X$ satisfies affine
    Nisnevich excision. For all affine schemes $U$ in $\Sm_S$, the canonical map
    $$\pi_0(X)(U) \rightarrow [U, X]_{\AA^1}$$ is an isomorphism.
\end{theorem}

\begin{proof}[Sketch Proof]
    The key homotopical input to this theorem is the $\pi_*$-Kan condition,
    which ensures that homotopy colimits of simplicial diagram commutes over
    pullbacks~\cite{bousfield-friedlander}. This condition was first used in
    this area by Schlichting~\cite{schlichting-euler}. It provides a concrete criterion to check if the
    functor $\Sing^{\AA^1}(F)$ restricted to smooth affine schemes is indeed Nisnevich
    local (to make this argument precise, the key algebro-geometric input is the equivalence between the Nisnevich cd-structure and the affine Nisnevich cd-structure defined above \cite{asok-hoyois-wendt}*{Proposition 2.3.2} when restricted to affine schemes).
    
     More precisely, for any elementary distinguished square
    \begin{equation*}
        \xymatrix{
            U\times_X V\ar[r]\ar[d] &   V\ar[d]^p\\
            U\ar[r]^i & Y
        }
           \end{equation*}
    we have a homotopy pullback square            \begin{equation*}
        \xymatrix{
           X(Y \times \AA^n) \ar[r]\ar[d] &   X(V \times \AA^n)\ar[d]\\
            X(U \times \AA^n)\ar[r] & X(U\times_X V \times \AA^n)
        }
    \end{equation*}
    of simplicial sets for all $n \geq 0$.

    The $\pi_*$-Kan condition applies with the hypothesis that $\pi_0(X)$
    satisfies affine $\AA^1$-invariance and we may conclude that taking
    $\hocolim_{\Delta^{\op}}$ of the above squares preserve pullbacks and
    therefore we conclude that $\Sing^{\AA^1}(X)$ satisfies affine Nisnevich
    excision.

    Applying the above proposition, we have that for any affine $U$,
    $$\Sing^{\AA^1}(X)(U) \rightarrow \L_{\Nis}\Sing^{\AA^1}(X)(U)$$ is
    a weak equivalence. Since the left hand side is $\AA^1$-invariant, we
    conclude that the right hand side is $\AA^1$-invariant; since being
    $\AA^1$-invariant and Nisnevich local may be tested on affine schemes
    (by~\cite{asok-hoyois-wendt}*{Proposition~2.3.2}), we
    conclude that $\L_{\Nis}\Sing^{\AA^1}(X) \simeq \L_{\AA^1}(X)$.
    Taking $\pi_0$ of the weak equivalence above gets us the desired claim.
\end{proof}

\begin{corollary}[Affine representability of vector bundles]
    Let $S$ be a regular noetherian affine scheme of finite Krull dimension,
    and suppose that the Bass-Quillen conjecture holds for smooth schemes of finite
    presentation over $S$. In this case, the natural map $\Vect_r(U) \rightarrow [U,\BGL_r]_{\AA^1}$
    is an isomorphism for all $U\in\Sm_S^{\mathrm{Aff}}$ and all $r\geq 0$.
\end{corollary}

The Jouanoulou-Thomason homotopy lemma states that, up to $\AA^1$-homotopy, we
may replace a smooth scheme with an affine one.

\begin{theorem}[\cite{jouanolou} and~\cite{weibel-homotopy}]
    Given a smooth separated scheme $U$ over a regular noetherian affine scheme $S$, there exists
    an affine vector bundle torsor $\widetilde{U}\rightarrow U$ such that
    $\widetilde{U}$ is affine.
\end{theorem}

\begin{proof}
    The point is that $U$ is quasi-compact and quasi-separated and hence admits
    an ample family of line bundles (so $U$ is divisorial)
    by~\cite{sga6}*{Proposition~II.2.2.7}.
    The theorem now follows from~\cite{weibel-homotopy}*{Proposition~4.4}.
\end{proof}

This theorem lets us compute in some
sense $[U,\BGL_n]_{\AA^1}$ for any $U\in\Sm_S$, but it is not known at the
moment what kind of objects these are on $U$.

One of the main features of $\AA^1$-localization is the ability to employ
topological thinking in algebraic geometry, if one is willing to work
$\AA^1$-locally. The homotopy sheaves $\pi_i^{\AA^1}(X)$ are
sometimes computable using input from both homotopy theory and algebraic geometry.
At the same time, many algebro-geometric problems are
inherently not $\AA^1$-local in nature so one only gets an actual
algebro-geometric theorem under certain certain conditions, as in
Theorem~\ref{thm:vectsurfaces} below. Let us first start with a review of
Postnikov towers in $\AA^1$-homotopy theory. Our main reference is
\cite{asok-fasel-threefolds}, which in turn uses \cite{morel}, \cite{morel-voevodsky} and
\cite{goerss-jardine}.

Let $G$ be a Nisnevich sheaf of groups and $M$ a Nisnevich sheaf of abelian
groups on which $G$ acts (a $G$-module). In this case, $G$ acts on the Eilenberg-Maclane
sheaf $K(A,n)$, from which we may construct $K^{G}(A, n) := EG
\times^{G} K(A, n)$. The first projection gives us a map
$K^{G}(A, n) \rightarrow G$.

Of primary interest is the Nisnevich sheaf of groups $\pi_1^{\AA^1}(Y)$ for
some pointed $\AA^1$-connected space $Y$. In this case, $\pi_1^{\AA^1}(Y)$ acts
on the higher homotopy sheaves $\pi_n^{\AA^1}(Y)$ where $n \geq 2$.

\begin{theorem}\label{thm:a1postnikov}
    Let $Y$ be a pointed $\AA^1$-connected space. There exists a commutative
    diagram of pointed $\AA^1$-connected spaces
    
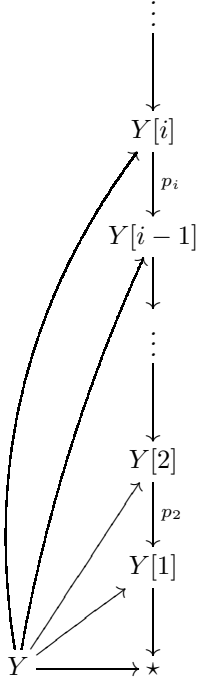
\begin{equation*}
        \xymatrix{
            &\vdots\ar[d]\\
            &Y[i]\ar[d]^{p_i}\\
            &Y[i-1]\ar[d]\\
            &\vdots\ar[d]\\
            &Y[2]\ar[d]^{p_2}\\
            &Y[1]\ar[d]\\
            Y\ar[r]\ar[ur]\ar[uur]\ar@/^/[uuuur]\ar@/^2pc/[uuuuur]&\star
        }
    \end{equation*}
    such that
    \begin{enumerate}
        \item $Y[1]\we B\pi_1^{\AA^1}(Y)$,
        \item $\pi_j^{\AA^1}Y[i]=0$ for $j>i$,
        \item the map $\pi_j^{\AA^1}Y\rightarrow\pi_j^{\AA^1}Y[i]$ is an
            isomorphism of $\pi_1^{\AA^1}Y$-modules for $1\leq j\leq i$,
        \item the $K(\pi_i^{\AA^1}Y,i)$-bundle $Y[i]\rightarrow Y[i-1]$ is a
            \df{twisted principal fibration} in the sense that there is a map
            $$k_{i}:Y[i-1]\rightarrow K^{\pi_1^{\AA^1}Y}(\pi_i^{\AA^1}Y,i+1)$$
            such that $Y[i]$ is obtained as the pullback
            \begin{equation*}
                \xymatrix{
                    Y[i]\ar[r]\ar[d]    &   \B_{\Nis}\pi_1^{\AA^1}Y\ar[d]\\
                    Y[i-1]\ar[r]        &
                    K^{\pi_1^{\AA^1}Y}(\pi_i^{\AA^1}Y,i+1),
                }
            \end{equation*}
        \item and $Y\rightarrow\lim_iY[i]$ is an $\AA^1$-weak equivalence.
    \end{enumerate}
    The tower, which can be made functorial in $Y$, is called the \df{$\AA^1$-Postnikov
    tower} of $Y$.
\end{theorem}

\begin{proof}
    This is left as an exercise, which basically amounts to Nisnevich
    sheafifying and $\AA^1$-localizing the usual Postnikov tower. For an extensive discussion, see \cite{asok-fasel-threefolds}*{Section 6} and the references therein.
\end{proof}

The point of the Postnikov tower is to make it possible to classify maps from
$X$ to $Y$ by constructing maps inductively, i.e., starting with a map
$X\rightarrow Y[1]$, lifting it to $X\rightarrow Y[2]$ while controlling the
choices of lifts, and so on.

\begin{theorem}\label{thm:postnikov-analysis}
    Let $S$ be a quasi-compact quasi-separated base scheme, and let $X$ be a
    smooth noetherian $S$-scheme of Krull dimension at most $d$.
    Suppose that $(Y, y)$ is a pointed $\AA^1$-connected space. The natural map
    $$[X,Y]_{\AA^1}\rightarrow[X,Y[i]]_{\AA^1}$$ is an isomorphism for $i\geq
    d$ and a surjection for $i=d-1$.
\end{theorem}

\begin{proof}
    The first obstruction to lifting a map $X\rightarrow Y[i]$ to $X\rightarrow
    Y$ is the obstruction to lifting it to $X\rightarrow Y[i+1]$. This is
    classified by the $k$-invariant, and hence a class in
    $\H^{i+2}_{\Nis,\pi_1^{\AA^1}Y}(X,\pi_{i+1}^{\AA^1}Y)=[X,K^{\pi_1^{\AA^1}Y}(\pi_{i+1}^{\AA^1}Y,i+2)]_{\AA^1}$.
    One important feature of the theory intervenes at this point: the equivariant
    cohomology group $\H^{i+2}_{\Nis,\pi_1^{\AA^1}Y}(X,\pi_{i+1}^{\AA^1}Y)$ can
    be identified with an ordinary Nisnevich cohomology group of a twisted form
    $(\pi_{i+1}^{\AA^1}Y)_{\lambda}$
    of $\pi_{i+1}^{\AA^1}Y$ in Nisnevich sheaves on $X$:
    $$\H^{i+2}_{\Nis,\pi_1^{\AA^1}Y}(X,\pi_{i+1}^{\AA^1}Y)\iso\H^{i+2}_{\Nis}(X,(\pi_{i+1}^{\AA^1}Y)_\lambda).$$
    See~\cite{morel}*{Appendix~B}.
    This group vanishes if $i+2>d$, or $i+1\geq d$, since the Nisnevich
    cohomological dimension of $X$ is at most $d$ by hypothesis. Thus, the map in the
    theorem is a surjection for $i\geq d-1$. The set of lifts, by the
    $\AA^1$-fiber sequence induced from Theorem~\ref{thm:a1postnikov}(4), is a
    quotient of $\H^{i+1}_{\Nis,\pi_1^{\AA^1}Y}(X,\pi_{i+1}^{\AA^1}Y)$, which
    vanishes for the same reason as above if $i+1>d$, or $i\geq d$. This completes the proof.
\end{proof}

We have an immediate consequence of the existence of the Postnikov towers as follows.

\begin{proposition}
    If $E$ is a rank $n>d$ vector bundle on a smooth affine $d$-dimensional
    variety $X$, then $E$ splits off a trivial direct summand.
\end{proposition}

\begin{proof}
    Using Proposition~\ref{prop:slncover} and Theorem~\ref{thm:stabilitysln},
    we see that $\pi_i^{\AA^1}\BGL_d\rightarrow\pi_i^{\AA^1}\BGL_n$ is an
    isomorphism for $i\leq d-1$, and a surjection for $i=d$. By the
    representability theorem, $E$ is represented by a map $X\rightarrow\BGL_n$
    in the $\AA^1$-homotopy category. Compose this map with
    $\BGL_n\rightarrow\BGL_n[d]$ to obtain $g:X\rightarrow\BGL_n[d]$. Note that $E$ is uniquely determined by
    $g$ by Theorem~\ref{thm:postnikov-analysis}. It suffices to lift
    $g$ to a map $h:X\rightarrow\BGL_d[d]$. The fiber of
    $\BGL_d[d]\rightarrow\BGL_n[d]$ is a $K(A,d)$-space for some Nisnevich
    sheaf $A$ with an action of $\Gm=\pi_1^{\AA^1}\BGL_d$. It follows that the obstruction to lifting $g$ through
    $\BGL_d[d]\rightarrow\BGL_n[d]$ is a class of $\H^{d+1}_{\Nis,\Gm}(X,A)=0$.
\end{proof}

\begin{remark}
    As a Nisnevich sheaf of spaces, $\BGL_n$ is a $K(\pi,1)$-sheaf in the sense
    that it has only one non-zero homotopy group. For the purposes of
    obstruction theory and classification theory this is not terribly useful as
    choosing a lift to the first stage of the Nisnevich-local Postnikov tower
    of $\BGL_n$ is equivalent to specifying a vector bundle. The process of
    $\AA^1$-localization mysteriously acts as a prism that separates the single
    homotopy sheaf into an entire sequence (spectrum) of homotopy sheaves,
    allowing a finer step-by-step analysis.
\end{remark}

\begin{proposition}
    The first few $\AA^1$-homotopy sheaves of $\BGL_2$ are
    \begin{align*}
        \pi_0^{\AA^1}\BGL_2&=\star,\\
        \pi_1^{\AA^1}\BGL_2&\iso\Gm,\\
        \pi_2^{\AA^1}\BGL_2&\iso\K_2^{\mathrm{MW}},
    \end{align*}
    where $\K_2^{\mathrm{MW}}$ denotes the second unramified Milnor-Witt sheaf.
\end{proposition}

\begin{proof}
    The $\AA^1$-connectivity statement $\pi_0^{\AA^1}\BGL_2=\star$ follows from the fact that vector bundles are
    Zariski and hence Nisnevich locally trivial. The fact that
    $\pi_1\BGL_2\iso\Gm$ follows from the stable range result that gives
    $\pi_1\BGL_2\iso\pi_1\BGL_{\infty}\iso\K_1\iso\Gm$, where the last two
    isomorphisms are explained in Section~\ref{sec:kstablerange}. The last
    follows from the $\AA^1$-fiber sequence
    $$\AA^2-\{0\}\rightarrow\BGL_1\rightarrow\BGL_2,$$ the fact that $\BGL_1$
    is a $K(\Gm,1)$-space, and Morel's result~\cite{morel}*{Theorem~6.40},
    which says that $\pi_1^{\AA^1}\AA^2-\{0\}\iso\K_2^{\mathrm{MW}}$.
\end{proof}

Now, for any smooth scheme $X$ and any line bundle $\Lscr$ on $X$,
there is an exact sequence of Nisnevich sheaves
\begin{equation}\label{eq:134}
    0\rightarrow \mathrm{I}^3(\Lscr)\rightarrow\K_2^{\mathrm{MW}}(\Lscr)\rightarrow\K_2\rightarrow 0,
\end{equation}
where the first and second terms are the $\Lscr$-twisted forms (see~\cite{morel}). The sheaf
$\K_2^{\mathrm{MW}}(\Lscr)$ controls the rank $2$ vector bundles on $X$ with
determinant $\Lscr$.

If $X$ is a smooth affine surface, there is a bijection
$[X,\BGL_2]_{\AA^1}\rightarrow[X,\BGL_2[2]]_{\AA^1}$, from which it follows
that the rank $2$ vector bundles on $X$ with determinant $\Lscr$ are classified
by a quotient of
$$\H^2_{\Nis,\Gm}(X,\K_2^{\mathrm{MW}})\iso\H^2_{\Nis}(X,\K_2^{\mathrm{MW}}(\Lscr)).$$
In fact, we will see that the quotient is all of
$\H^2_{\Nis}(X,\K_2^{\mathrm{MW}}(\Lscr))$.

\begin{lemma}
    If $X$ is a smooth affine surface over a quadratically closed field $k$,
    then $\H^n_{\Nis}(X,\I^3(\Lscr))=0$ for $n\geq 2$.
\end{lemma}

\begin{proof}
    This is~\cite{asok-fasel-threefolds}*{Proposition~5.2}.
\end{proof}

It follows from the lemma and the exact sequence~\eqref{eq:134} that the space of lifts is a quotient of
$\H^2_{\Nis}(X,\K_2)\iso\CH^2(X)$, where the isomorphism is due to
Quillen~\cite{quillen} in
the Zariski topology and Thomason-Trobaugh~\cite{thomason-trobaugh} in the Nisnevich topology.
Now, looking at
$$\Gm(X)\iso[X,K(\Gm,0)]_{\AA^1}\rightarrow[X,K^{\Gm}(\K_2^{\mathrm{MW}},2)]_{\AA^1}\rightarrow[X,\BGL_2[2]]_{\AA^1}\rightarrow[X,\BGL_2[1]]_{\AA^1}\iso\Pic(X),$$
we see that the map
$[X,K^{\Gm}(\K_2^{\mathrm{MW}},2)]_{\AA^1}\rightarrow[X,\BGL_2[2]]_{\AA^1}$ is
injective because every element of $\Gm(X)$ extends to an automorphism
of the vector bundle classified by $X\rightarrow\BGL_2$.

It follows that the map $$\Vect_2(X)\rightarrow\CH^1(X)\times\H^2_{\Nis}(X,\K_2)$$ is a
bijection. Thus, we have sketched a proof of the following theorem.

\begin{theorem}\label{thm:vectsurfaces}
    Let $X$ be a smooth affine surface over a quadratically closed field. Then,
    the map $$(c_1,c_2):\Vect_2(X)\rightarrow\CH^1(X)\times\CH^2(X),$$ induced by
    taking the first and second Chern classes, is a bijection.
\end{theorem}

\begin{remark}
    To see that the natural maps involved are the Chern classes, as claimed,
    refer to~\cite{asok-fasel-threefolds}*{Section 6}.
\end{remark}

The fact that the theorem holds over quadratically closed fields is stronger
than the previous results in this direction, which had been obtained without
$\AA^1$-homotopy theory.

Asok and Fasel have carried this program much farther in several papers, for
instance showing in~\cite{asok-fasel-threefolds} that
$\Vect_2(X)\iso\CH^1(X)\times\CH^2(X)$ when $X$ is a smooth affine
\emph{three-fold} over a quadratically closed field. This theorem, which is
outside the stable range, is much more difficult.

% Further exercises
\section{Further directions}\label{sec:exercises}

In most of the exercises below, none of which are supposed to be
easy, it will be useful to bear in mind the universal
properties of $\L_{\Nis}$ and $\L_{\AA^1}$.

\begin{exercise}
    Use the formalism of model categories to construct topological and \'etale
    realization functors out of the $\AA^1$-homotopy category $\Spc_S^{\AA^1}$.
    Dugger's paper~\cite{dugger-universal} on universal homotopy theories may
    come in handy. This problem is studied specifically
    in~\cite{dugger-isaksen-realizations} and~\cite{dugger-isaksen-etale}.
\end{exercise}

\begin{exercise}
    Show that topological realization takes the motivic sphere $S^{a,b}$ where $a\leq
    b$ to the topological sphere $S^a$.
\end{exercise}

% \begin{exercise}
%     Let $G$ be a sheaf of abelian groups on $\Sm_S$. The contraction of a sheaf
%     $G$, denoted by $G_{-1}$ is defined to by the formula: $U \in\Sm_S
%     \mapsto\ker(\mathrm{ins}_1: G(X
%     \times_S \GG_m) \rightarrow G(X))$, where $\mathrm{ins}_1$ is the ``insert at $1$''
%     morphism. Prove that
%     \begin{enumerate}
%         \item the contraction of sheaves is an exact functor on the category of
%             $\AA^1$-local abelian sheaves over $\Sm_S$, and that
%         \item for any $\AA^1$-local abelian sheaf $G$, we have an isomorphism $G(\GG_m
%                 \times_S X) \iso G_{-1}(X) \oplus G(X)$. 
%     \end{enumerate}
% \end{exercise}

\begin{exercise}
    Prove that complex topological $K$-theory is representable in
    $\Spc_S^{\AA^1}$.
\end{exercise}

\begin{exercise}
    Let $\RR$ denote the field of real numbers. Construct a realization functor
    from $\Spc_{\RR}^{\AA^1}$ to the homotopy theory of $\ZZ/2$-equivariant
    topological spaces. Again, see~\cite{dugger-isaksen-realizations}.
\end{exercise}

\begin{exercise}
    Construct a realization functor from $\Spc_S^{\AA^1}$ to Voevodsky's
    category $\mathrm{DM}(S)$ of (big) motives over $S$. It will probably be necessary to search the literature
    for a model category structure for $\mathrm{DM}(S)$.
\end{exercise}

\begin{exercise}
    Show that the realization functor from $\Spc_S^{\AA^1}$ to Voevodsky's
    category factors through the \emph{stable} motivic homotopy category
    obtained from $\Spc_S^{\AA^1}$ by stabilizing with respect to $S^{2,1}\we\PP^1$.
\end{exercise}

\begin{exercise}
    Ayoub~\cite{ayoub} has constructed a $6$-functors formalism for stable
    motivic homotopy theory. Construct some
    functors between $\Spc_S^{\AA^1}$ and $\Spc_U^{\AA^1}$ when $U$ is open in
    $S$ and between $\Spc_S^{\AA^1}$ and $\Spc_Z^{\AA^1}$ when $Z$ is closed in $S$.
\end{exercise}


\begin{bibdiv}
\begin{biblist}

\bib{adamek-rosicky}{book}{
    author={Ad{\'a}mek, Ji{\v{r}}{\'{\i}}},
    author={Rosick{\'y}, Ji{\v{r}}{\'{\i}}},
    title={Locally presentable and accessible categories},
    series={London Mathematical Society Lecture Note Series},
    volume={189},
    publisher={Cambridge University Press, Cambridge},
    date={1994},
    pages={xiv+316},
    isbn={0-521-42261-2},
%     review={\MR{1294136 (95j:18001)}},
%     doi={10.1017/CBO9780511600579},
}

\bib{adamek-rosicky-vitale}{article}{
    author={Ad{\'a}mek, J.},
    author={Rosick{\'y}, J.},
    author={Vitale, E. M.},
    title={What are sifted colimits?},
    journal={Theory Appl. Categ.},
    volume={23},
    date={2010},
    pages={No. 13, 251--260},
    issn={1201-561X},
%     review={\MR{2720191}},
}

% \bib{aw1}{article}{
%     author={Antieau, Benjamin},
%     author={Williams, Ben},
%     title={The period-index problem for twisted topological
%     $K$-theory},
%     journal={Geom. Topol.},
%     volume={18},
%     date={2014},
%     number={2},
%     pages={1115--1148},
%     issn={1465-3060},
% %     review={\MR{3190609}},
% %     doi={10.2140/gt.2014.18.1115},
% }
% 
%  
% % \bib{aw3}{article}{
% %     author = {Antieau, Benjamin},
% %     author = {Williams, Ben},
% %     title = {The topological period-index problem for $6$-complexes},
% %     journal = {J. Top.},
% %     eprint = {http://doi:10.1112/jtopol/jtt042},
% %     year = {2013},
% % }
% 
% \bib{aw4}{article}{
%     author={Antieau, Benjamin},
%     author={Williams, Ben},
%     title={Unramified division algebras do not always contain Azumaya
%     maximal
%     orders},
%     journal={Invent. Math.},
%     volume={197},
%     date={2014},
%     number={1},
%     pages={47--56},
%     issn={0020-9910},
% %     review={\MR{3219514}},
% %     doi={10.1007/s00222-013-0479-7},
% }
% 
% \bib{aw7}{article}{
%     author={Antieau, Benjamin},
%     author={Williams, Ben},
%     title={The prime divisors of the period and index of a Brauer
%     class},
%     journal={J. Pure Appl. Algebra},
%     volume={219},
%     date={2015},
%     number={6},
%     pages={2218--2224},
%     issn={0022-4049},
% %     review={\MR{3299728}},
% %     doi={10.1016/j.jpaa.2014.07.032},
% }

% \bib{sga4-3}{book}{
%     title={Th\'eorie des topos et cohomologie \'etale des sch\'emas. Tome 3},
%     series={Lecture Notes in Mathematics, Vol. 305},
%     note={S\'eminaire de G\'eom\'etrie Alg\'ebrique du Bois-Marie 1963--1964 (SGA 4);
%     Dirig\'e par M. Artin, A. Grothendieck et J. L. Verdier. Avec la collaboration de P. Deligne et B. Saint-Donat},
%     publisher={Springer-Verlag},
%     place={Berlin},
%     date={1973},
%     pages={vi+640},
% %     review={\MR{0354654 (50 \#7132)}},
% }

\bib{artin-mazur}{book}{
    author={Artin, M.},
    author={Mazur, B.},
    title={Etale homotopy},
    series={Lecture Notes in Mathematics, No. 100},
    publisher={Springer-Verlag, Berlin-New York},
    date={1969},
    pages={iii+169},
%     review={\MR{0245577 (39 \#6883)}},
}

\bib{asok-toric}{article}{    AUTHOR = {Asok, Aravind},
    TITLE = {Remarks on {$\mathbb{A}^1$}-homotopy groups of smooth toric
              models},
   JOURNAL = {Math. Res. Lett.},
    VOLUME = {18},
      YEAR = {2011},
    NUMBER = {2},
     PAGES = {353--356},
      ISSN = {1073-2780},
       DOI = {10.4310/MRL.2011.v18.n2.a12},
       URL = {http://dx.doi.org/10.4310/MRL.2011.v18.n2.a12},
}
		


\bib{aravind-fundamentalgrp}{article}{AUTHOR = {Asok, Aravind},
    TITLE = {Splitting vector bundles and {$\mathbb{A}^1$}-fundamental
              groups of higher-dimensional varieties},
   JOURNAL = {J. Topol.},
    VOLUME = {6},
      YEAR = {2013},
    NUMBER = {2},
     PAGES = {311--348},
      ISSN = {1753-8416},
       %DOI = {10.1112/jtopol/jts034},
       %URL = {http://dx.doi.org/10.1112/jtopol/jts034}
}
\bib{asok-doran}{article}{
    author={Asok, Aravind},
    author={Doran, Brent},
    title={Vector bundles on contractible smooth schemes},
    journal={Duke Math. J.},
    volume={143},
    date={2008},
    number={3},
    pages={513--530},
    issn={0012-7094},
%     review={\MR{2423761}},
%     doi={10.1215/00127094-2008-027},
}

\bib{asok-doran-fasel}{article}{
    author={Asok, Aravind},
    author={Doran, Brent},
    author={Fasel, Jean},
    title={Smooth models of motivic spheres and the clutching construction},
    journal={to appear in Int. Math. Res. Notices},
    eprint={https://arxiv.org/abs/1408.0413},
%     number={3},
%     pages={513--530},
%     issn={0012-7094},
%     review={\MR{2423761}},
%     doi={10.1215/00127094-2008-027},
}

\bib{asok-fasel-threefolds}{article}{
    author={Asok, Aravind},
    author={Fasel, Jean},
    title={A cohomological classification of vector bundles on smooth affine threefolds},
    journal={Duke Math. J.},
    volume={163},
    date={2014},
    number={14},
    pages={2561--2601},
    issn={0012-7094},
%     review={\MR{3273577}},
%     doi={10.1215/00127094-2819299},
}

% \bib{asok-fasel-fourfolds}{article}{
%     author={Asok, Aravind},
%     author={Fasel, Jean},
%     title={Splitting vector bundles outside the stable range and $\AA^1$-homotopy sheaves of punctured affine spaces},
%     journal={to appear in J. AMS},
%     eprint={http://arxiv.org/abs/1209.5631},
% %     review={\MR{3273577}},
% %     doi={10.1215/00127094-2819299},
% }

\bib{asok-hoyois-wendt}{article}{
    author = {Asok, Aravind},
    author = {Hoyois, Marc},
    author = {Wendt, Matthias},
    title = {Affine representability results in $\mathds{A}^1$-homotopy theory I: vector bundles},
    journal = {ArXiv e-prints},
    eprint =  {http://arxiv.org/abs/1506.07093},
    year = {2015},
}

\bib{asok-hoyois-wendt-II}{article}{
    author = {Asok, Aravind},
    author = {Hoyois, Marc},
    author = {Wendt, Matthias},
    title = {Affine representability results in $\mathds{A}^1$-homotopy theory II: principal bundles and homogeneous spaces},
    journal = {ArXiv e-prints},
    eprint =  {http://arxiv.org/abs/1507.08020},
    year = {2015},
}



% \bib{atiyah-segal}{article}{
%     author={Atiyah, Michael},
%     author={Segal, Graeme},
%     title={Twisted $K$-theory},
%     journal={Ukr. Mat. Visn.},
%     volume={1},
%     date={2004},
%     number={3},
%     pages={287--330},
%     issn={1810-3200},
%     translation={
%     journal={Ukr. Math. Bull.},
%     volume={1},
%     date={2004},
%     number={3},
%     pages={291--334},
%     issn={1812-3309},
%     },
% %     review={\MR{2172633 (2006m:55017)}},
% }

% \bib{atiyah-segal-cohomology}{article}{
%     author={Atiyah, Michael},
%     author={Segal, Graeme},
%     title={Twisted $K$-theory and cohomology},
%     conference={
%     title={Inspired by S. S. Chern},
%     },
%     book={
%     series={Nankai Tracts Math.},
%     volume={11},
%     publisher={World Sci. Publ.,
%     Hackensack, NJ},
%     },
%     date={2006},
%     pages={5--43},
% %     review={\MR{2307274 (2008j:55003)}},
% %     doi={10.1142/9789812772688_0002},
% }

% \bib{auel}{article}{
%     author = {Auel, Asher},
%     title = {Surjectivity of the total Clifford invariant and Brauer dimension},
%     journal = {ArXiv e-prints},
%     eprint = {http://arxiv.org/abs/1108.5728},
%     year = {2011},
% }
% 
% \bib{auel-remarks}{article}{
%     author={Auel, Asher},
%     title={Remarks on the Milnor conjecture over schemes},
%     conference={
%         title={Galois-Teichmüller Theory and Arithmetic Geometry (Kyoto, 2010)},
%     },
%     book={
%         series={Advanced Studies in Pure Mathematics},
%         volume={63},
%         editor={Nakamura, H.},
%         editor={Pop, F.},
%         editor={Schneps, L.},
%         editor={Tamagawa, A.},
%     },
%     year={2012},
%     pages={1--30},
% }
% 
% % \bib{auslander-goldman}{article}{
% %     author={Auslander, Maurice},
% %     author={Goldman, Oscar},
% %     title={Maximal orders},
% %     journal={Trans. Amer. Math. Soc.},
% %     volume={97},
% %     date={1960},
% %     pages={1--24},
% %     issn={0002-9947},
% % %     review={\MR{0117252 (22 \#8034)}},
% % }
% 
% \bib{auslander-goldman}{article}{
%     author={Auslander, Maurice},
%     author={Goldman, Oscar},
%     title={The Brauer group of a commutative ring},
%     journal={Trans. Amer. Math. Soc.},
%     volume={97},
%     date={1960},
%     pages={367--409},
%     issn={0002-9947},
% %     review={\MR{0121392 (22 \#12130)}},
% }
\bib{ayoub}{article}{
    author={Ayoub, Joseph},
    title={Les six op\'erations de Grothendieck et le formalisme des
    cycles
    \'evanescents dans le monde motivique. I},
    journal={Ast\'erisque},
    number={314},
    date={2007},
    pages={x+466 pp. (2008)},
%     issn={0303-1179},
%     isbn={978-2-85629-244-0},
%     review={\MR{2423375 (2009h:14032)}},
}
% 
% \bib{balmer-walter}{article}{
%     author={Balmer, Paul},
%     author={Walter, Charles},
%     title={A Gersten-Witt spectral sequence for regular schemes},
%     journal={Ann. Sci. \'Ecole Norm. Sup. (4)},
%     volume={35},
%     date={2002},
%     number={1},
%     pages={127--152},
%     issn={0012-9593},
% %     review={\MR{1886007 (2003d:19005)}},
% %     doi={10.1016/S0012-9593(01)01084-9},
% }

\bib{balwe-hogadi-sawant}{article}{
    author={Balwe, Chetan},
    author={Hogadi, Amit},
    author={Sawant, Anand},
    title={$\mathbb{A}^1$-connected components of schemes},
    journal={Adv. Math.},
    volume={282},
    date={2015},
    pages={335--361},
    issn={0001-8708},
%     review={\MR{3374529}},
%     doi={10.1016/j.aim.2015.07.003},
}

% 
% % \bib{baum-browder}{article}{
% % 	title = {The cohomology of quotients of classical groups},
% % 	volume = {3},
% % 	journal = {Topology},
% % 	author = {Baum, Paul F.},
% %     author={Browder, William},
% % 	year = {1965},
% % 	pages = {305--336}
% % }

\bib{sga6}{book}{
title={Th\'eorie des intersections et th\'eor\`eme de Riemann-Roch},
series={Lecture Notes in Mathematics, Vol. 225},
note={S\'eminaire de G\'eom\'etrie Alg\'ebrique du Bois-Marie
1966--1967
(SGA 6);
Dirig\'e par P. Berthelot, A. Grothendieck et L.
Illusie. Avec la
collaboration de D. Ferrand, J. P. Jouanolou, O.
Jussila, S. Kleiman, M.
Raynaud et J. P. Serre},
publisher={Springer-Verlag, Berlin-New
York},
date={1971},
pages={xii+700},
% review={\MR{0354655}},
}

% 
% % \bib{bloch-ogus}{article}{
% %     author={Bloch, Spencer},
% %     author={Ogus, Arthur},
% %     title={Gersten's conjecture and the homology of schemes},
% %     journal={Ann. Sci. \'Ecole Norm. Sup. (4)},
% %     volume={7},
% %     date={1974},
% %     pages={181--201 (1975)},
% %     issn={0012-9593},
% % %     review={\MR{0412191 (54 \#318)}},
% % }
% 
\bib{bott}{article}{
    author={Bott, Raoul},
    title={The space of loops on a Lie group},
    journal={Michigan Math. J.},
    volume={5},
    date={1958},
    pages={35--61},
    issn={0026-2285},
}

\bib{bousfield-friedlander}{article}{
author={Bousfield, A. K.},
author={Friedlander, E. M.},
title={Homotopy theory of $\Gamma $-spaces, spectra, and
bisimplicial
sets},
conference={
title={Geometric applications of homotopy
theory (Proc. Conf.,
Evanston, Ill., 1977), II},
},
book={
series={Lecture
Notes in Math.},
volume={658},
publisher={Springer,
Berlin},
},
date={1978},
pages={80--130},
% review={\MR{513569}},
}


\bib{bousfield-kan}{book}{
    author={Bousfield, A. K.},
    author={Kan, D. M.},
    title={Homotopy limits, completions and localizations},
    series={Lecture Notes in Mathematics, Vol. 304},
    publisher={Springer-Verlag},
    place={Berlin},
    date={1972},
    pages={v+348},
%     review={\MR{0365573 (51 \#1825)}},
}
% 
% % \bib{brown}{article}{
% %     author={Brown, Edgar H., Jr.},
% %     title={The cohomology of $B{\rm SO}_{n}$ and $B{\rm O}_{n}$ with
% %     integer coefficients},
% %     journal={Proc. Amer. Math. Soc.},
% %     volume={85},
% %     date={1982},
% %     number={2},
% %     pages={283--288},
% %     issn={0002-9939},
% % %     review={\MR{652459 (83d:55015)}},
% % %     doi={10.2307/2044298},
% % }
\bib{brown-gersten}{article}{
author={Brown, Kenneth S.},
author={Gersten, Stephen M.},
title={Algebraic $K$-theory as generalized sheaf cohomology},
conference={
title={Algebraic K-theory, I: Higher K-theories},
address={Proc. Conf., Battelle Memorial Inst.,
Seattle, Wash.},
date={1972},
},
book={
publisher={Springer,
Berlin},
},
date={1973},
pages={266--292.
Lecture
Notes in
Math., Vol.
341},
% review={\MR{0347943 (50
% \#442)}},
}
% 
% \bib{cartan}{article}{
%     author = {Cartan, H.},
%     title = {D{\'e}termination des alg{\`e}bres {$\Hoh_*(\pi,n;\ZZ)$}},
%     journal = {S{\'e}minaire H. Cartan},
%     volume = {7},
%     number = {1},
%     pages = {11-01--11-24},
%     publisher = {Secr{\'e}tariat math{\'e}matique},
%     address = {Paris},
%     year = {1954/1955},
% }

\bib{cazanave}{article}{
    author={Cazanave, Christophe},
    title={Algebraic homotopy classes of rational functions},
    journal={Ann. Sci. \'Ec. Norm. Sup\'er. (4)},
    volume={45},
    date={2012},
    number={4},
    pages={511--534 (2013)},
    issn={0012-9593},
%     review={\MR{3059240}},
}

% 
% \bib{chernousov-panin}{article}{
%     author={Chernousov, Vladimir},
%     author={Panin, Ivan},
%     title={Purity of $G_2$-torsors},
%     journal={C. R. Math. Acad. Sci. Paris},
%     volume={345},
%     date={2007},
%     number={6},
%     pages={307--312},
%     issn={1631-073X},
% %     review={\MR{2359087 (2008j:20146)}},
% %     doi={10.1016/j.crma.2007.07.018},
% }
\bib{cisinski}{article}{
    author={Cisinski, Denis-Charles},
    title={Descente par \'eclatements en $K$-th\'eorie invariante par
    homotopie},
    journal={Ann. of Math. (2)},
    volume={177},
    date={2013},
    number={2},
    pages={425--448},
    issn={0003-486X},
%     review={\MR{3010804}},
%     doi={10.4007/annals.2013.177.2.2},
}
% 
% \bib{colliot-thelene-sansuc}{article}{
%     author={Colliot-Th{\'e}l{\`e}ne, J.-L.},
%     author={Sansuc, J.-J.},
%     title={Fibr\'es quadratiques et composantes connexes r\'eelles},
%     journal={Math. Ann.},
%     volume={244},
%     date={1979},
%     number={2},
%     pages={105--134},
%     issn={0025-5831},
% %     review={\MR{550842 (81c:14010)}},
% %     doi={10.1007/BF01420486},
% }
% 
% \bib{colliot-thelene-birational}{article}{
%     author={Colliot-Th{\'e}l{\`e}ne, J.-L.},
%     title={Birational invariants, purity and the Gersten conjecture},
%     conference={
%         address={Santa Barbara, CA},
%         date={1992},
%     },
%     book={
%         series={Proc. Sympos. Pure Math.},
%         volume={58},
%         publisher={Amer. Math. Soc.},
%         place={Providence, RI},
%     },
%     date={1995},
%     pages={1--64},
% %     review={\MR{1327280 (96c:14016)}},
% }
% 
% % \bib{colliot}{article}{
% %     author={Colliot-Th{\'e}l{\`e}ne, Jean-Louis},
% %     title={Exposant et indice d'alg\`ebres simples centrales non ramifi\'ees},
% %     note={With an appendix by Ofer Gabber},
% %     journal={Enseign. Math. (2)},
% %     volume={48},
% %     date={2002},
% %     number={1-2},
% %     pages={127--146},
% %     issn={0013-8584},
% % }
% 
% \bib{colliot-thelene-voisin}{article}{
%     author={Colliot-Th{\'e}l{\`e}ne, Jean-Louis},
%     author={Voisin, Claire},
%     title={Cohomologie non ramifi\'ee et conjecture de Hodge enti\`ere},
%     journal={Duke Math. J.},
%     volume={161},
%     date={2012},
%     number={5},
%     pages={735--801},
%     issn={0012-7094},
% %     review={\MR{2904092}},
% %     doi={10.1215/00127094-1548389},
% }
% 
% % \bib{dejong}{article}{
% %     author={de Jong, Johan},
% %     title={The period-index problem for the Brauer group of an algebraic surface},
% %     journal={Duke Math. J.},
% %     volume={123},
% %     date={2004},
% %     number={1},
% %     pages={71--94},
% %     issn={0012-7094},
% % %     review={\MR{2060023 (2005e:14025)}},
% % %     doi={10.1215/S0012-7094-04-12313-9},
% % }
% 
% \bib{dejong-gabber}{article}{
%     author={de Jong, A. J.},
%     title={A result of Gabber},
%     eprint={http://www.math.columbia.edu/~dejong/},
% }
% 
% 
% % \bib{dejong-starr}{article}{
% %     author={Starr, Jason},
% %     author={de Jong, Johan},
% %     title={Almost proper GIT-stacks and discriminant avoidance},
% %     journal={Doc. Math.},
% %     volume={15},
% %     date={2010},
% %     pages={957--972},
% %     issn={1431-0635},
% % %     review={\MR{2745688 (2012e:14073)}},
% % }

\bib{deligne-icm}{article}{
    author={Deligne, Pierre},
    title={Poids dans la cohomologie des vari\'et\'es alg\'ebriques},
    conference={
    title={Proceedings of the International Congress of
    Mathematicians},
    address={Vancouver, B. C.},
    date={1974},
    },
    book={
    publisher={Canad. Math. Congress, Montreal, Que.},
    },
    date={1975},
    pages={79--85},
%     review={\MR{0432648}},
}
    

% 
% \bib{demeyer}{article}{
%     author={DeMeyer, F. R.},
%     title={Projective modules over central separable algebras},
%     journal={Canad. J. Math.},
%     volume={21},
%     date={1969},
%     pages={39--43},
%     issn={0008-414X},
% %     review={\MR{0234987 (38 \#3299)}},
% }

\bib{dugger-universal}{article}{
    author={Dugger, Daniel},
    title={Universal homotopy theories},
    journal={Adv. Math.},
    volume={164},
    date={2001},
    number={1},
    pages={144--176},
    issn={0001-8708},
%     review={\MR{1870515 (2002k:18021)}},
%     doi={10.1006/aima.2001.2014},
}

\bib{dugger-bg}{article}{
    author={Dugger, Daniel},
    title={The Zariski and Nisnevich descent theorems},
    eprint={http://pages.uoregon.edu/ddugger/desc.html},
    date={2001},
}

\bib{dugger-hollander-isaksen}{article}{
author={Dugger, Daniel},
author={Hollander, Sharon},
author={Isaksen, Daniel C.},
title={Hypercovers and simplicial presheaves},
journal={Math. Proc. Cambridge Philos. Soc.},
volume={136},
date={2004},
number={1},
pages={9--51},
issn={0305-0041},
% review={\MR{2034012 (2004k:18007)}},
% doi={10.1017/S0305004103007175},
}

\bib{dugger-isaksen-realizations}{article}{
    author={Dugger, Daniel},
    author={Isaksen, Daniel C.},
    title={Topological hypercovers and $\mathbb{A}^1$-realizations},
    journal={Math. Z.},
    volume={246},
    date={2004},
    number={4},
    pages={667--689},
    issn={0025-5874},
%     review={\MR{2045835}},
%     doi={10.1007/s00209-003-0607-y},
}

\bib{dugger-isaksen-etale}{article}{
    author={Dugger, Daniel},
    author={Isaksen, Daniel C.},
    title={Etale homotopy and sums-of-squares formulas},
    journal={Math. Proc. Cambridge Philos. Soc.},
    volume={145},
    date={2008},
    number={1},
    pages={1--25},
    issn={0305-0041},
%     review={\MR{2431636}},
%     doi={10.1017/S0305004108001205},
}

\bib{dugger-isaksen}{article}{
    author={Dugger, Daniel},
    author={Isaksen, Daniel C.},
    title={The motivic Adams spectral sequence},
    journal={Geom. Topol.},
    volume={14},
    date={2010},
    number={2},
    pages={967--1014},
    issn={1465-3060},
%     review={\MR{2629898 (2011e:55024)}},
%     doi={10.2140/gt.2010.14.967},
}

\bib{dundas-rondigs-ostvaer}{article}{
    author={Dundas, Bj{\o}rn Ian},
    author={R{\"o}ndigs, Oliver},
    author={{\O}stv{\ae}r, Paul Arne},
    title={Motivic functors},
    journal={Doc. Math.},
    volume={8},
    date={2003},
    pages={489--525},
    issn={1431-0635},
%     review={\MR{2029171}},
}

\bib{dro-book}{book}{
    AUTHOR = {Dundas, B. I. and Levine, M. and {\O}stv{\ae}r, P. A. and
              R{\"o}ndigs, O. and Voevodsky, V.},
     TITLE = {Motivic homotopy theory},
    SERIES = {Universitext},
      NOTE = {Lectures from the Summer School held in Nordfjordeid, August
              2002},
 PUBLISHER = {Springer-Verlag, Berlin},
      YEAR = {2007},
     PAGES = {x+221},
      ISBN = {978-3-540-45895-1; 3-540-45895-6},
       DOI = {10.1007/978-3-540-45897-5},
       URL = {http://dx.doi.org/10.1007/978-3-540-45897-5},
}

\bib{dwyer-spalinski}{article}{
   author={Dwyer, W. G.},
   author={Spali{\'n}ski, J.},
   title={Homotopy theories and model categories},
   conference={
      title={Handbook of algebraic topology},
   },
   book={
      publisher={North-Holland, Amsterdam},
   },
   date={1995},
   pages={73--126},
%    review={\MR{1361887 (96h:55014)}},
%    doi={10.1016/B978-044481779-2/50003-1},
}


% 
% \bib{ekedahl}{article}{
%     author = {Ekedahl, Torsten},
%     title = {Approximating classifying spaces by smooth projective varieties},
%     journal = {ArXiv e-prints},
%     eprint =  {http://arxiv.org/abs/0905.1538},
%     year = {2009},
% }

% \bib{fga-explained}{collection}{
%     author={Fantechi, Barbara},
%     author={G{\"o}ttsche, Lothar},
%     author={Illusie, Luc},
%     author={Kleiman, Steven L.},
%     author={Nitsure, Nitin},
%     author={Vistoli, Angelo},
%     title={Fundamental algebraic geometry},
%     series={Mathematical Surveys and Monographs},
%     volume={123},
%     publisher={American Mathematical Society,
%     Providence, RI},
%     date={2005},
%     pages={x+339},
%     isbn={0-8218-3541-6},
% %     review={\MR{2222646 (2007f:14001)}},
% }
% 
% 
% \bib{fedorov-panin}{article}{
%     author = {Fedorov, Roman},
%     author = {Panin, Ivan},
%     title = {Proof of Grothendieck-Serre conjecture on principal bundles over regular local rings containing infinite fields},
%     journal = {ArXiv e-prints},
%     eprint =  {http://arxiv.org/abs/1211.2678},
%     year = {2012},
% }
% 
% \bib{fujiwara}{article}{
%     author={Fujiwara, Kazuhiro},
%     title={A proof of the absolute purity conjecture (after Gabber)},
%     conference={
%         title={Algebraic geometry 2000, Azumino (Hotaka)},
%     },
%     book={
%         series={Adv. Stud. Pure Math.},
%         volume={36},
%         publisher={Math. Soc. Japan},
%         place={Tokyo},
%     },
%     date={2002},
%     pages={153--183},
% %     review={\MR{1971516 (2004d:14015)}},
% }
% 
% \bib{gabber}{article}{
%     author={Gabber, Ofer},
%     title={Some theorems on Azumaya algebras},
%     conference={
%     title={The Brauer group},
%     address={Sem., Les Plans-sur-Bex},
%     date={1980},
%     },
%     book={
%     series={Lecture Notes in Math.},
%     volume={844},
%     publisher={Springer},
%     place={Berlin},
%     },
%     date={1981},
%     pages={129--209},
% %     review={\MR{611868 (83d:13004)}},
% }
% 
% \bib{gabber-note}{article}{
%     author={Gabber, Ofer},
%     title={A note on the unramified Brauer group and purity},
%     journal={Manuscripta Math.},
%     volume={95},
%     date={1998},
%     number={1},
%     pages={107--115},
%     issn={0025-2611},
% %     review={\MR{1492372 (98m:14018)}},
% %     doi={10.1007/BF02678018},
% }
% 
% \bib{gille-szamuely}{book}{
%     author={Gille, Philippe},
%     author={Szamuely, Tam{\'a}s},
%     title={Central simple algebras and Galois cohomology},
%     series={Cambridge Studies in Advanced Mathematics},
%     volume={101},
%     publisher={Cambridge University Press},
%     place={Cambridge},
%     date={2006},
%     pages={xii+343},
%     isbn={978-0-521-86103-8},
%     isbn={0-521-86103-9},
%     % review={\MR{2266528 (2007k:16033)}},
%     % doi={10.1017/CBO9780511607219},
% }

\bib{goerss-jardine}{book}{
   author={Goerss, Paul G.},
   author={Jardine, John F.},
   title={Simplicial homotopy theory},
   series={Progress in Mathematics},
   volume={174},
   publisher={Birkh\"auser Verlag, Basel},
   date={1999},
   pages={xvi+510},
   isbn={3-7643-6064-X},
%    review={\MR{1711612 (2001d:55012)}},
%    doi={10.1007/978-3-0348-8707-6},
}


\bib{goerss-schemmerhorn}{article}{
   author={Goerss, Paul},
   author={Schemmerhorn, Kristen},
   title={Model categories and simplicial methods},
   conference={
      title={Interactions between homotopy theory and algebra},
   },
   book={
      series={Contemp. Math.},
      volume={436},
      publisher={Amer. Math. Soc., Providence, RI},
   },
   date={2007},
   pages={3--49},
%    review={\MR{2355769 (2009a:18010)}},
%    doi={10.1090/conm/436/08403},
}

% 
% \bib{goresky-macpherson}{book}{
%     author={Goresky, Mark},
%     author={MacPherson, Robert},
%     title={Stratified Morse theory},
%     series={Ergebnisse der Mathematik und ihrer Grenzgebiete (3)},
%     volume={14},
%     publisher={Springer-Verlag},
%     place={Berlin},
%     date={1988},
%     pages={xiv+272},
%     isbn={3-540-17300-5},
% %     review={\MR{932724 (90d:57039)}},
% }

\bib{sga1}{book}{
    author={Grothendieck, Alexander},
    title={Rev\^etements \'etales et groupe fondamental. Fasc. I: Expos\'es 1
    \`a 5},
    series={S\'eminaire de G\'eom\'etrie Alg\'ebrique},
    volume={1960/61},
    publisher={Institut des Hautes \'Etudes Scientifiques, Paris},
    date={1963},
    pages={iv+143 pp. (not consecutively paged) (loose errata)},
%     review={\MR{0217087 (36 \#179a)}},
}

% 
% \bib{grothendieck-brauer-1}{article}{
%     author={Grothendieck, Alexander},
%     title={Le groupe de Brauer. I. Alg\`ebres d'Azumaya et interpr\'etations diverses},
%     conference={
%     title={S\'eminaire Bourbaki, Vol.\ 9},
%     },
%     book={
%     publisher={Soc. Math. France},
%     place={Paris},
%     },
%     date={1995},
%     pages={Exp.\ No.\ 290, 199--219},
% }

% \bib{grothendieck-brauer-2}{article}{
%     author={Grothendieck, Alexander},
%     title={Le groupe de Brauer. II. Th\'eorie cohomologique},
%     conference={
%         title={Dix Expos\'es sur la Cohomologie des Sch\'emas},
%     },
%     book={
%         publisher={North-Holland},
%         place={Amsterdam},
%     },
%     date={1968},
%     pages={67--87},
% %     review={\MR{0244270 (39 \#5586b)}},
% }
% % 
\bib{hatcher}{book}{
    author={Hatcher, Allen},
    title={Algebraic topology},
    publisher={Cambridge University Press},
    place={Cambridge},
    date={2002},
    pages={xii+544},
%     isbn={0-521-79160-X},
%     isbn={0-521-79540-0},
%     review={\MR{1867354 (2002k:55001)}},
}

\bib{hirschhorn}{book}{
    author={Hirschhorn, Philip S.},
    title={Model categories and their localizations},
    series={Mathematical Surveys and Monographs},
    volume={99},
    publisher={American Mathematical Society, Providence, RI},
    date={2003},
    pages={xvi+457},
    isbn={0-8218-3279-4},
%     review={\MR{1944041 (2003j:18018)}},
}

\bib{hollander}{article}{
  AUTHOR = {Hollander, Sharon},
     TITLE = {A homotopy theory for stacks},
   JOURNAL = {Israel J. Math.},
    VOLUME = {163},
      YEAR = {2008},
     PAGES = {93--124},
      ISSN = {0021-2172},
%        DOI = {10.1007/s11856-008-0006-5},
%        URL = {http://dx.doi.org/10.1007/s11856-008-0006-5},
}
%  
% \bib{hoobler}{article}{
%     author={Hoobler, Raymond T.},
%     title={A cohomological interpretation of Brauer groups of rings},
%     journal={Pacific J. Math.},
%     volume={86},
%     date={1980},
%     number={1},
%     pages={89--92},
%     issn={0030-8730},
% %     review={\MR{586870 (81j:13003)}},
% }
 
\bib{hovey}{book}{
    author={Hovey, Mark},
    title={Model categories},
    series={Mathematical Surveys and Monographs},
    volume={63},
    publisher={American Mathematical Society, Providence, RI},
    date={1999},
    pages={xii+209},
    isbn={0-8218-1359-5},
%     review={\MR{1650134}},
}
 
\bib{hoyois}{article}{
    author = {Hoyois, Marc},
    title = {The six operations in equivariant motivic homotopy theory},
    journal = {ArXiv e-prints},
    eprint =  {http://arxiv.org/abs/1509.02145},
    year = {2015},
}
%
\bib{hoyois-allagree}{article}{
    author = {Hoyois, Marc},
    title = {A trivial remark on the Nisnevich topology},
    eprint =  {http://math.mit.edu/~hoyois/papers/allagree.pdf},
    year = {2016},
}

\bib{husemoller}{book}{
    author={Husemoller, Dale},
    title={Fibre bundles},
    edition={2},
    note={Graduate Texts in Mathematics, No. 20},
    publisher={Springer-Verlag},
    place={New York},
    date={1975},
    pages={xv+327},
}
% 
% \bib{jackowski-mcclure-oliver-1}{article}{
%     author={Jackowski, Stefan},
%     author={McClure, James},
%     author={Oliver, Bob},
%     title={Homotopy classification of self-maps of $BG$ via $G$-actions. I},
%     journal={Ann. of Math. (2)},
%     volume={135},
%     date={1992},
%     number={1},
%     pages={183--226},
%     issn={0003-486X},
% %     review={\MR{1147962 (93e:55019a)}},
% %     doi={10.2307/2946568},
% }
% 
\bib{jardine}{article}{
   author={Jardine, J. F.},
   title={Simplicial presheaves},
   journal={J. Pure Appl. Algebra},
   volume={47},
   date={1987},
   number={1},
   pages={35--87},
   issn={0022-4049},
  % review={\MR{906403 (88j:18005)}},
  % doi={10.1016/0022-4049(87)90100-9},
}
% 
\bib{jouanolou}{article}{
    author={Jouanolou, J. P.},
    title={Une suite exacte de Mayer-Vietoris en $K$-th\'eorie alg\'ebrique},
    conference={
    title={Algebraic $K$-theory, I: Higher $K$-theories (Proc. Conf.,
    Battelle Memorial Inst., Seattle, Wash., 1972)},
    },
    book={
    publisher={Springer},
    place={Berlin},
    },
    date={1973},
    pages={293--316. Lecture
    Notes in Math., Vol.
    341},
%     review={\MR{0409476 (53 \#13231)}},
}
% 
% % \bib{kameko-yagita}{article}{
% %     author={Kameko, Masaki},
% %     author={Yagita, Nobuaki},
% %     title={The Brown-Peterson cohomology of the classifying spaces of the
% %     projective unitary groups ${\rm PU}(p)$ and exceptional Lie groups},
% %     journal={Trans. Amer. Math. Soc.},
% %     volume={360},
% %     date={2008},
% %     number={5},
% %     pages={2265--2284},
% %     issn={0002-9947},
% % }
% 
% % \bib{kervaire}{article}{
% %     author={Kervaire, Michel A.},
% %     title={Some nonstable homotopy groups of Lie groups},
% %     journal={Illinois J. Math.},
% %     volume={4},
% %     date={1960},
% %     pages={161--169},
% %     issn={0019-2082},
% % %     review={\MR{0113237 (22 \#4075)}},
% % }
\bib{lam-serre2}{book}{
    author={Lam, T. Y.},
    title={Serre's problem on projective modules},
    series={Springer Monographs in Mathematics},
    publisher={Springer-Verlag, Berlin},
    date={2006},
    pages={xxii+401},
    isbn={978-3-540-23317-6},
    isbn={3-540-23317-2},
%     review={\MR{2235330 (2007b:13014)}},
%     doi={10.1007/978-3-540-34575-6},
}


\bib{levine-milan}{article}{
    AUTHOR = {Levine, Marc},
     TITLE = {Motivic homotopy theory},
   JOURNAL = {Milan J. Math.},
    VOLUME = {76},
      YEAR = {2008},
     PAGES = {165--199},
      ISSN = {1424-9286},
%        DOI = {10.1007/s00032-008-0088-x},
%        URL = {http://dx.doi.org/10.1007/s00032-008-0088-x},
}
\bib{levine-vietnam}{article}{
    AUTHOR = {Levine, Marc},
     TITLE = {An overview of motivic homotopy theory},
   JOURNAL = {Acta Math. Vietnam.},
    VOLUME = {41},
      YEAR = {2016},
    NUMBER = {3},
     PAGES = {379--407},
      ISSN = {0251-4184},
%        DOI = {10.1007/s40306-016-0184-x},
%        URL = {http://dx.doi.org/10.1007/s40306-016-0184-x},
}
\bib{levine-morel}{book}{
    author={Levine, M.},
    author={Morel, F.},
    title={Algebraic cobordism},
    series={Springer Monographs in Mathematics},
    publisher={Springer, Berlin},
    date={2007},
    pages={xii+244},
    isbn={978-3-540-36822-9},
    isbn={3-540-36822-1},
%     review={\MR{2286826 (2008a:14029)}},
}



% 
% % \bib{lieblich}{article}{
% %     author={Lieblich, Max},
% %     title={Twisted sheaves and the period-index problem},
% %     journal={Compos. Math.},
% %     volume={144},
% %     date={2008},
% %     number={1},
% %     pages={1--31},
% %     issn={0010-437X},
% % %     review={\MR{2388554 (2009b:14033)}},
% % %     doi={10.1112/S0010437X07003144},
% % }

\bib{lindel}{article}{
    author={Lindel, Hartmut},
    title={On the Bass-Quillen conjecture concerning projective modules
    over
    polynomial rings},
    journal={Invent. Math.},
    volume={65},
    date={1981/82},
    number={2},
    pages={319--323},
    issn={0020-9910},
%     review={\MR{641133 (83g:13009)}},
%     doi={10.1007/BF01389017},
}

\bib{htt}{book}{
   author={Lurie, Jacob},
   title={Higher topos theory},
   series={Annals of Mathematics Studies},
   volume={170},
   publisher={Princeton University Press, Princeton, NJ},
   date={2009},
   pages={xviii+925},
   isbn={978-0-691-14049-0},
   isbn={0-691-14049-9},
%    review={\MR{2522659 (2010j:18001)}},
%    doi={10.1515/9781400830558},
}

\bib{lurie-sag}{unpublished}{
      author={Lurie, Jacob},
      title={Spectral algebraic geometry},
        date={2016},
  note={\texttt{\href{http://www.math.harvard.edu/~lurie/}{http://www.math.harvard.edu/\textasciitilde lurie/}}},
}

% \bib{may-equivariant}{book}{
%     author={May, J. P.},
%     title={Equivariant homotopy and cohomology theory},
%     series={CBMS Regional Conference Series in Mathematics},
%     volume={91},
%     note={With contributions by M. Cole, G. Comeza{\~{n}}a, S.  Costenoble,
%     A. D. Elmendorf, J. P. C. Greenlees, L. G. Lewis,
%     Jr., R. J. Piacenza, G.
%     Triantafillou, and S. Waner},
%     publisher={American Mathematical Society,
%     Providence, RI},
%     date={1996},
%     pages={xiv+366},
% %     isbn={0-8218-0319-0},
% %     review={\MR{1413302 (97k:55016)}},
% }

\bib{may-concise}{book}{
    author={May, J. P.},
    title={A concise course in algebraic topology},
    series={Chicago Lectures in Mathematics},
    publisher={University of Chicago Press, Chicago, IL},
    date={1999},
    pages={x+243},
    isbn={0-226-51182-0},
    isbn={0-226-51183-9},
%     review={\MR{1702278}},
}


\bib{mazza-weibel-voevodsky}{book}{
AUTHOR = {Mazza, Carlo},
 AUTHOR = {Voevodsky, Vladimir},
 AUTHOR = {Weibel, Charles},
     TITLE = {Lecture notes on motivic cohomology},
    SERIES = {Clay Mathematics Monographs},
    VOLUME = {2},
 PUBLISHER = {American Mathematical Society, Providence, RI; Clay
              Mathematics Institute, Cambridge, MA},
      YEAR = {2006},
     PAGES = {xiv+216},
      %ISBN = {978-0-8218-3847-1; 0-8218-3847-4},
   %MRCLASS = {14F42 (19E15)},
  %MRNUMBER = {2242284},
%MRREVIEWER = {Thomas Geisser},
}

% 
% % \bib{mccleary}{book}{
% %     author={McCleary, John},
% %     title={A user's guide to spectral sequences},
% %     series={Cambridge Studies in Advanced Mathematics},
% %     volume={58},
% %     edition={2},
% %     publisher={Cambridge University Press},
% %     place={Cambridge},
% %     date={2001},
% %     pages={xvi+561},
% %     isbn={0-521-56759-9},
% % }

\bib{mcduff-segal}{article}{
    author={McDuff, D.},
    author={Segal, G.},
    title={Homology fibrations and the ``group-completion'' theorem},
    journal={Invent. Math.},
    volume={31},
    date={1975/76},
    number={3},
    pages={279--284},
    issn={0020-9910},
%     review={\MR{0402733}},
}

 
% % \bib{milnor-stasheff}{book}{
% %     author={Milnor, John W.},
% %     author={Stasheff, James D.},
% %     title={Characteristic classes},
% %     note={Annals of Mathematics Studies, No. 76},
% %     publisher={Princeton University Press},
% %     place={Princeton, N. J.},
% %     date={1974},
% %     pages={vii+331},
% % }

\bib{morel-icm}{article}{
AUTHOR = {Morel, Fabien},
TITLE = {{$\mathbb{A}^1$}-algebraic topology},
 BOOKTITLE = {International {C}ongress of {M}athematicians. {V}ol. {II}},
     PAGES = {1035--1059},
 PUBLISHER = {Eur. Math. Soc., Z\"urich},
      YEAR = {2006},
}

\bib{morel}{book}{
    author={Morel, Fabien},
    title={${\bf A}^1$ algebraic topology over a field},
    series={Lecture Notes in Mathematics, Vol. 2052},
    publisher={Springer Verlag},
    place={Belin},
    date={2012},
}

\bib{morel-voevodsky}{article}{
    author={Morel, Fabien},
    author={Voevodsky, Vladimir},
    title={${\bf A}^1$-homotopy theory of schemes},
    journal={Inst. Hautes \'Etudes Sci. Publ. Math.},
    number={90},
    date={1999},
    pages={45--143},
    issn={0073-8301},
%     review={\MR{1813224 (2002f:14029)}},
}
   
% \bib{murthy-survey}{article}{
%     author={Murthy, M. Pavaman},
%     title={A survey of obstruction theory for projective modules of top rank},
%     conference={
%         title={Algebra, $K$-theory, groups, and education},
%         address={New York},
%         date={1997},
%     },
%     book={
%         series={Contemp. Math.},
%         volume={243},
%         publisher={Amer.  Math. Soc., Providence, RI},
%     },
%     date={1999},
%     pages={153--174},
% %     review={\MR{1732046 (2000m:13013)}},
% %     doi={10.1090/conm/243/03692},
% }

\bib{nisnevich}{article}{
author={Nisnevich, Ye. A.},
title={The completely decomposed topology on schemes and associated
descent spectral sequences in algebraic $K$-theory},
conference={
title={Algebraic $K$-theory: connections with
geometry and topology},
address={Lake Louise, AB},
date={1987},
},
book={
series={NATO Adv.
Sci. Inst. Ser. C
Math. Phys. Sci.},
volume={279},
publisher={Kluwer
Acad.
Publ.,
Dordrecht},
},
date={1989},
pages={241--342},
% review={\MR{1045853}},
}

% 
% \bib{ojanguren}{article}{
%     author = {Ojanguren, Manuel},
%     title = {Wedderburn's theorem for regular local rings},
%     journal = {Linear Algebraic Groups Preprint Server},
%     eprint = {http://www.mathematik.uni-bielefeld.de/LAG/man/500.html},
%     year = {2013},
% }
% 
% \bib{ojanguren-panin}{article}{
%     author={Ojanguren, Manuel},
%     author={Panin, Ivan},
%     title={A purity theorem for the Witt group},
%     journal={Ann. Sci. \'Ecole Norm. Sup. (4)},
%     volume={32},
%     date={1999},
%     number={1},
%     pages={71--86},
%     issn={0012-9593},
% %     review={\MR{1670591 (2000a:11057)}},
% %     doi={10.1016/S0012-9593(99)80009-3},
% }
% 
% \bib{panin-purity}{article}{
%     author={Panin, I. A.},
%     title={Purity conjecture for reductive groups},
%     journal={Vestnik St. Petersburg Univ. Math.},
%     volume={43},
%     date={2010},
%     number={1},
%     pages={44--48},
%     issn={1063-4541},
% %     review={\MR{2662409 (2011j:14100)}},
% %     doi={10.3103/S1063454110010085},
% }
% 
% \bib{parimala-sridharan-2}{article}{
%     author={Parimala, R.},
%     author={Sridharan, R.},
%     title={Nonsurjectivity of the Clifford invariant map},
%     note={K. G. Ramanathan memorial issue},
%     journal={Proc. Indian Acad. Sci. Math. Sci.},
%     volume={104},
%     date={1994},
%     number={1},
%     pages={49--56},
%     issn={0253-4142},
% %     review={\MR{1280057 (95g:11030)}},
% %     doi={10.1007/BF02830873},
% }
% 
% % \bib{puttmann-rigas}{article}{
% %     author={P{\"u}ttmann, Thomas},
% %     author={Rigas, A.},
% %     title={Presentations of the first homotopy groups of the unitary groups},
% %     journal={Comment. Math. Helv.},
% %     volume={78},
% %     date={2003},
% %     number={3},
% %     pages={648--662},
% %     issn={0010-2571},
% % }
\bib{quillen-homotopical}{book}{
   author={Quillen, Daniel G.},
   title={Homotopical algebra},
   series={Lecture Notes in Mathematics, No. 43},
   publisher={Springer-Verlag, Berlin-New York},
   date={1967},
   pages={iv+156 pp. (not consecutively paged)},
%    review={\MR{0223432 (36 \#6480)}},
}% 

\bib{quillen}{article}{
    author={Quillen, Daniel},
    title={Higher algebraic $K$-theory. I},
    conference={
    title={Algebraic $K$-theory, I: Higher $K$-theories},
    address={Proc. Conf., Battelle Memorial Inst.,
    Seattle, Wash.},
    date={1972},
    },
    book={
    publisher={Springer,
    Berlin},
    },
    date={1973},
    pages={85--147.
    Lecture Notes
    in Math., Vol.
    341},
%     review={\MR{0338129 (49 \#2895)}},
}


\bib{quillen-serre}{article}{
    author={Quillen, Daniel},
    title={Projective modules over polynomial rings},
    journal={Invent. Math.},
    volume={36},
    date={1976},
    pages={167--171},
    issn={0020-9910},
%     review={\MR{0427303 (55 \#337)}},
}

% % \bib{raynaud}{article}{
% %    author={Raynaud, Mich{\`e}le},
% %    title={Modules projectifs universels},
% %    journal={Invent. Math.},
% %    volume={6},
% %    date={1968},
% %    pages={1--26},
% %    issn={0020-9910},
% %  %  review={\MR{0236164 (38 \#4462)}},
% % }

\bib{robalo}{article}{
    author={Robalo, Marco},
    title={$K$-theory and the bridge from motives to noncommutative
    motives},
    journal={Adv. Math.},
    volume={269},
    date={2015},
    pages={399--550},
    issn={0001-8708},
%     review={\MR{3281141}},
%     doi={10.1016/j.aim.2014.10.011},
}
% 
% \bib{saltman}{book}{
%     author={Saltman, David J.},
%     title={Lectures on division algebras},
%     series={CBMS Regional Conference Series in Mathematics},
%     volume={94},
%     publisher={Published by American Mathematical Society, Providence, RI},
%     date={1999},
%     pages={viii+120},
%     isbn={0-8218-0979-2},
% %     review={\MR{1692654 (2000f:16023)}},
% }


\bib{schlichting-euler}{article}{
    author={Schlichting, Marco},
    title={Euler class groups, and the homology of elementary and special linear groups},
    journal={ArXiv e-prints},
    date={2015},
    eprint={https://arxiv.org/abs/1502.05424},
}

\bib{schlichting-tripathi}{article}{
    author={Schlichting, Marco},
    author={Tripathi, Girja S.},
    title={Geometric models for higher Grothendieck-Witt groups in
    $\mathbb{A}^1$-homotopy theory},
    journal={Math. Ann.},
    volume={362},
    date={2015},
    number={3-4},
    pages={1143--1167},
    issn={0025-5831},
%     review={\MR{3368095}},
%     doi={10.1007/s00208-014-1154-z},
}

% 
% \bib{schoen}{article}{
%     author={Schoen, Chad},
%     title={Complex varieties for which the Chow group mod $n$ is not finite},
%     journal={J. Algebraic Geom.},
%     volume={11},
%     date={2002},
%     number={1},
%     pages={41--100},
%     issn={1056-3911},
% %     review={\MR{1865914 (2002h:14004)}},
% %     doi={10.1090/S1056-3911-01-00291-0},
% }

\bib{schwede-shipley}{article}{
    author={Schwede, Stefan},
    author={Shipley, Brooke E.},
    title={Algebras and modules in monoidal model categories},
    journal={Proc. London Math. Soc. (3)},
    volume={80},
    date={2000},
    number={2},
    pages={491--511},
    issn={0024-6115},
%     review={\MR{1734325}},
%     doi={10.1112/S002461150001220X},
}

\bib{segal}{article}{
    author={Segal, Graeme},
    title={Categories and cohomology theories},
    journal={Topology},
    volume={13},
    date={1974},
    pages={293--312},
    issn={0040-9383},
%     review={\MR{0353298}},
}

\bib{serre-fac}{article}{
    author={Serre, Jean-Pierre},
    title={Faisceaux alg\'ebriques coh\'erents},
    journal={Ann. of Math. (2)},
    volume={61},
    date={1955},
    pages={197--278},
    issn={0003-486X},
%     review={\MR{0068874 (16,953c)}},
}

\bib{stacks-project}{article}{
    author       = {The {Stacks Project Authors}},
    title        = {\itshape Stacks Project},
    eprint = {http://stacks.math.columbia.edu},
    year         = {2014},
}

\bib{suslin-serre}{article}{
    author={Suslin, A. A.},
    title={Projective modules over polynomial rings are free},
    journal={Dokl. Akad. Nauk SSSR},
    volume={229},
    date={1976},
    number={5},
    pages={1063--1066},
    issn={0002-3264},
%     review={\MR{0469905 (57 \#9685)}},
}

\bib{swan-popescu}{article}{
    author={Swan, Richard G.},
    title={N\'eron-Popescu desingularization},
    conference={
    title={Algebra and geometry},
    address={Taipei},
    date={1995},
    },
    book={
    series={Lect. Algebra
    Geom.},
    volume={2},
    publisher={Int.
    Press,
    Cambridge,
    MA},
    },
    date={1998},
    pages={135--192},
%     review={\MR{1697953 (2000h:13006)}},
}

% 
\bib{thom}{article}{
    author={Thom, Ren{\'e}},
    title={Quelques propri\'et\'es globales des vari\'et\'es
    diff\'erentiables},
    journal={Comment. Math. Helv.},
    volume={28},
    date={1954},
    pages={17--86},
}

\bib{thomas}{book}{
    author={Thomas, Emery},
    title={Seminar on fiber spaces},
    series={Lectures delivered in 1964 in Berkeley and 1965 in Z\"urich.
    Berkeley notes by J. F. McClendon. Lecture Notes in
    Mathematics},
    volume={13},
    publisher={Springer-Verlag, Berlin-New York},
    date={1966},
    pages={iv+45},
%     review={\MR{0203733}},
}

\bib{thomason-trobaugh}{article}{
    author={Thomason, R. W.},
    author={Trobaugh, Thomas},
    title={Higher algebraic $K$-theory of schemes and of derived
    categories},
    conference={
    title={The Grothendieck Festschrift, Vol.\ III},
    },
    book={
    series={Progr. Math.},
    volume={88},
    publisher={Birkh\"auser
    Boston, Boston, MA},
    },
    date={1990},
    pages={247--435},
%     review={\MR{1106918 (92f:19001)}},
%     doi={10.1007/978-0-8176-4576-2_10},
}

% 
\bib{totaro}{article}{
    author={Totaro, Burt},
    title={The Chow ring of a classifying space},
    conference={
    title={Algebraic $K$-theory},
    address={Seattle, WA},
    date={1997},
    },
    book={
    series={Proc. Sympos. Pure Math.},
    volume={67},
    publisher={Amer. Math. Soc.},
    place={Providence, RI},
    },
    date={1999},
    pages={249--281},
%     review={\MR{1743244 (2001f:14011)}},
}
% 
% \bib{totaro-torsion}{article}{
%     author={Totaro, Burt},
%     title={Torsion algebraic cycles and complex cobordism},
%     journal={J. Amer. Math. Soc.},
%     volume={10},
%     date={1997},
%     number={2},
%     pages={467--493},
%     issn={0894-0347},
% %     review={\MR{1423033 (98a:14012)}},
% %     doi={10.1090/S0894-0347-97-00232-4},
% }
% 
% \bib{totaro-injectivity}{article}{
%     author={Totaro, Burt},
%     title={Non-injectivity of the map from the Witt group of a variety to the
%     Witt group of its function field},
%     journal={J. Inst. Math. Jussieu},
%     volume={2},
%     date={2003},
%     number={3},
%     pages={483--493},
%     issn={1474-7480},
% %     review={\MR{1990223 (2004f:19010)}},
% %     doi={10.1017/S1474748003000136},
% }
% 
% \bib{vezzosi}{article}{
%    author={Vezzosi, Gabriele},
%    title={On the Chow ring of the classifying stack of ${\rm PGL}_{3,\mathbf{C}}$},
%    journal={J. Reine Angew. Math.},
%    volume={523},
%    date={2000},
%    pages={1--54},
%    issn={0075-4102},
% %    doi={10.1515/crll.2000.048},
% }
% 
% \bib{vistoli}{article}{
%     author={Vistoli, Angelo},
%     title={On the cohomology and the Chow ring of the classifying space of ${\rm PGL}_p$},
%     journal={J. Reine Angew. Math.},
%     volume={610},
%     date={2007},
%     pages={181--227},
%     issn={0075-4102},
% }

\bib{vistoli}{article}{
    AUTHOR = {Vistoli, Angelo},
     TITLE = {Grothendieck topologies, fibered categories and descent
              theory},
 BOOKTITLE = {Fundamental algebraic geometry},
    SERIES = {Math. Surveys Monogr.},
    VOLUME = {123},
     PAGES = {1--104},
 PUBLISHER = {Amer. Math. Soc., Providence, RI},
      YEAR = {2005},
}

\bib{voevodsky-icm}{article}{
    AUTHOR = {Voevodsky, Vladimir},
    TITLE = {{$\mathbb{A}^1$}-homotopy theory},
 BOOKTITLE = {Proceedings of the {I}nternational {C}ongress of
              {M}athematicians, {V}ol. {I} ({B}erlin, 1998)},
   JOURNAL = {Doc. Math.},
      YEAR = {1998},
    NUMBER = {Extra Vol. I},
     PAGES = {579--604},
      ISSN = {1431-0635},
}

\bib{voevodsky-jpaa1}{article}{
    author={Voevodsky, Vladimir},
    title={Homotopy theory of simplicial sheaves in completely
    decomposable
    topologies},
    journal={J. Pure Appl. Algebra},
    volume={214},
    date={2010},
    number={8},
    pages={1384--1398},
    issn={0022-4049},
%     review={\MR{2593670}},
%     doi={10.1016/j.jpaa.2009.11.004},
}

\bib{voevodsky-jpaa2}{article}{
    author={Voevodsky, Vladimir},
    title={Unstable motivic homotopy categories in Nisnevich and
    cdh-topologies},
    journal={J. Pure Appl. Algebra},
    volume={214},
    date={2010},
    number={8},
    pages={1399--1406},
    issn={0022-4049},
%     review={\MR{2593671}},
%     doi={10.1016/j.jpaa.2009.11.005},
}

\bib{weibel-homotopy}{article}{
author={Weibel, Charles A.},
title={Homotopy algebraic $K$-theory},
conference={
title={Algebraic $K$-theory and algebraic number theory
(Honolulu, HI,
1987)},
},
book={
series={Contemp. Math.},
volume={83},
publisher={Amer.
Math. Soc.,
Providence,
RI},
},
date={1989},
pages={461--488},
% review={\MR{991991}},
% doi={10.1090/conm/083/991991},
}

\bib{weibel}{book}{
    author={Weibel, Charles A.},
    title={An introduction to homological algebra},
    series={Cambridge Studies in Advanced Mathematics},
    volume={38},
    publisher={Cambridge University Press, Cambridge},
    date={1994},
    pages={xiv+450},
    isbn={0-521-43500-5},
    isbn={0-521-55987-1},
%     review={\MR{1269324 (95f:18001)}},
%     doi={10.1017/CBO9781139644136},
}

\bib{weibel-kbook}{book}{
    author={Weibel, Charles A.},
    title={The $K$-book},
    series={Graduate Studies in Mathematics},
    volume={145},
    note={An introduction to algebraic $K$-theory},
    publisher={American Mathematical Society, Providence,
    RI},
    date={2013},
    pages={xii+618},
    isbn={978-0-8218-9132-2},
%     review={\MR{3076731}},
}

% \bib{wendt}{article}{
%     author={Wendt, Matthias},
%     title={Classifying spaces and fibrations of simplicial sheaves},
%     journal={J. Homotopy Relat. Struct.},
%     volume={6},
%     date={2011},
%     number={1},
%     pages={1--38},
% %     review={\MR{2818697}},
% }

% 
% % \bib{williams2013}{article}{
% %   author={Williams, Ben},
% %   title={The $\Gm$--equivariant motivic cohomology of Stiefel varieties},
% %   journal={Algebraic \& Geometric Topology},
% %   volume={13},
% %   date={2013},
% %   pages={747--793},
% % }
% 
% % \bib{whitehead}{book}{
% %    author={Whitehead, George W.},
% %    title={Elements of homotopy theory},
% %    series={Graduate Texts in Mathematics},
% %    volume={61},
% %    publisher={Springer-Verlag},
% %    place={New York},
% %    date={1978},
% %    pages={xxi+744},
% %    isbn={0-387-90336-4},
% % %    review={\MR{516508 (80b:55001)}},
% % }
% 
% % \bib{woodward}{article}{
% %     author={Woodward, L. M.},
% %     title={The classification of principal ${\rm PU}_{n}$-bundles over a $4$-complex},
% %     journal={J. London Math. Soc. (2)},
% %     volume={25},
% %     date={1982},
% %     number={3},
% %     pages={513--524},
% %     issn={0024-6107},
% % }

\end{biblist}
\end{bibdiv}
\end{document}